\documentclass{amsart}

\usepackage{amsmath}
\usepackage{amsfonts}
\usepackage{amssymb}
\usepackage{amsthm}
\usepackage{comment}
\usepackage{epsfig}
\usepackage{psfrag}
\usepackage{mathrsfs}
\usepackage{amscd}
\usepackage[all, import]{xy}
\usepackage{rotating}
\usepackage{lscape}
\usepackage{amsbsy}
\usepackage{verbatim} 
\usepackage{moreverb}
\usepackage{pigpen}
\usepackage{hyperref}
\usepackage[usenames]{color} 
\usepackage{bbm}
\usepackage{eucal}

\makeatletter
\renewcommand\subsubsection{\@startsection {subsubsection}{1}{\z@}%
                                   {-3.5ex \@plus -1ex \@minus -.2ex}%
                                   {2.3ex \@plus.2ex}%
                                   {\normalfont\bf}}
\renewcommand\subsection{\@startsection {subsection}{1}{\z@}%
                                   {-3.5ex \@plus -1ex \@minus -.2ex}%
                                   {2.3ex \@plus.2ex}%
                                   {\normalfont\bf}}
\makeatother

\newtheorem{theorem}{Theorem}[section]
\newtheorem{prop}[theorem]{Proposition}
\newtheorem{lemma}[theorem]{Lemma}
\newtheorem{cor}[theorem]{Corollary}

\numberwithin{equation}{subsection}

\theoremstyle{definition}
\newtheorem{definition}[theorem]{Definition}
\newtheorem{defn}[theorem]{Definition}

\newtheorem{notation}[theorem]{Notation}

\newtheorem{example}[theorem]{Example}
\newtheorem{non-example}[theorem]{Non-example}
\newtheorem{remark}[theorem]{Remark}

\newtheorem{notation/term}[theorem]{Notation/Terminology}
\newtheorem{construction}[theorem]{Construction}

\newtheorem{observation}[theorem]{Observation}

\theoremstyle{remark}

	\newcommand{\nc}{\newcommand}
	\nc{\DMO}{\DeclareMathOperator}
	\nc{\tensor}{\otimes}
	\nc{\unit}{{1}}
	\DMO{\ass}{\sf{Alg}}
	\DMO{\conf}{\sf{Conf}}
	\DMO{\open}{\sf{Open}}
	\DMO{\Mod}{\mathsf{-Mod}}
	\DMO{\Exit}{\sE}
	\DMO{\cech}{\mathsf{Cech}}
	\DMO{\fun}{\mathsf{Fun}}
	\DMO{\sbar}{\mathsf{Bar}}
	\DMO{\coend}{\mathsf{Coend}}
	\DMO{\modpair}{\mathsf{ModPair}}
	\DMO{\maps}{\mathsf{Maps}}
	\DMO{\emb}{\mathsf{Emb}}
	\DMO{\Ab}{\mathsf{Ab}}
	\DMO{\chain}{\mathsf{Chain}}
	\DMO{\ob}{\mathsf{ob}}
	\DMO{\lan}{\mathsf{Lan}}
	\DMO{\orderI}{ {\pi_0\cI^{\sqcup}_{\mathsf{or}}}}
	\DMO{\id}{id}
	\nc{\diskover}{(\disk_1^{\partial,\fr})_{/[-1,1]}}
	\nc{\snglrall}{\mathsf{Snglr}^{\mathsf{csm}}}
	\nc{\Sp}{\mathcal{S}p}

\DeclareMathOperator{\Aut}{\sf Aut}
\DeclareMathOperator*{\colim}{\sf colim}

\DeclareMathOperator{\holim}{\sf holim}

\DeclareMathOperator{\End}{\sf End}

\DeclareMathOperator{\Fun}{\sf Fun}
\DeclareMathOperator{\Iso}{\sf Iso}
\DeclareMathOperator{\Map}{\sf Map}

\DeclareMathOperator{\Gammac}{{\Gamma}_{\!\sf c}}

\DeclareMathOperator{\Ass}{\mathsf{Ass}}

\DMO{\cmpct}{cmpt}
\DMO{\depth}{depth}
\DMO{\snglrr}{Snglr}
\nc{\power}{\mathsf{Power}}

\DeclareMathOperator{\Cat}{\sf Cat_\infty}

\DeclareMathOperator{\m}{\sf Mod}

\DeclareMathOperator{\Ch}{\sf Ch}

\DeclareMathOperator{\Alg}{\mathsf{Alg}}

\DeclareMathOperator{\op}{\sf op}

\DeclareMathOperator{\Emb}{\mathsf{Emb}}

\DeclareMathOperator{\spaces}{\mathsf{Spaces}}
\DeclareMathOperator{\spectra}{\mathsf{Spectra}}

\DeclareMathOperator{\disk}{\cD\mathsf{isk}}
\DeclareMathOperator{\diskd}{\mathsf{Disk}}
\DeclareMathOperator{\fr}{\sf fr}

\DeclareMathOperator{\Fin}{\mathsf{Fin}}
\DeclareMathOperator{\Strat}{\cS{\sf trat}}
\DeclareMathOperator{\stratd}{\sf Strat}

\DeclareMathOperator{\mfld}{\cM\mathsf{fld}}
\DeclareMathOperator{\mfldd}{\mathsf{Mfld}}

\DeclareMathOperator{\snglr}{\cS\mathsf{nglr}}
\DeclareMathOperator{\snglrd}{\mathsf{Snglr}}

\DeclareMathOperator{\bsc}{\cB\mathsf{sc}}
\DeclareMathOperator{\bscd}{\mathsf{Bsc}}

\DeclareMathOperator{\Psh}{\mathsf{PShv}}
\DeclareMathOperator{\psh}{\mathsf{PShv}}

\DeclareMathOperator{\uno}{\mathbbm{1}}

\DeclareMathOperator{\BO}{{\mathsf BO}}

\def\ot{\otimes}
\DeclareMathOperator{\fin}{\sf Fin}
\DeclareMathOperator{\oo}{\infty}
\DeclareMathOperator{\hh}{\sf HC}
\DeclareMathOperator{\free}{\sf Free}

\DeclareMathOperator{\Ran}{\sf Ran}

\DeclareMathOperator{\bcH}{\boldsymbol{\mathcal H}}

\newcommand{\into}{\hookrightarrow}

\newcommand{\ra}{\rightarrow}
\newcommand{\la}{\leftarrow}
\newcommand{\xra}{\xrightarrow}
\newcommand{\xla}{\xleftarrow}
\newcommand{\ov}{\overline}

\newcommand{\w}{\widetilde}

\newcommand{\tl}{\triangleleft}

	\newcommand{\Kan}{\sf Kan}
	\newcommand{\Set}{\mathsf{Set}}

\def\cB{\mathcal B}\def\cC{\mathcal C}\def\cD{\mathcal D}
\def\cE{\mathcal E}\def\cF{\mathcal F}\def\cG{\mathcal G}
\def\cI{\mathcal I}\def\cJ{\mathcal J}\def\cK{\mathcal K}
\def\cM{\mathcal M}\def\cO{\mathcal O}\def\cP{\mathcal P}
\def\cS{\mathcal S}
\def\cV{\mathcal V}\def\cX{\mathcal X}

\def\AA{\mathbb A}\def\DD{\mathbb D}
\def\HH{\mathbb H}

\def\QQ{\mathbb Q}\def\RR{\mathbb R}\def\SS{\mathbb S}

\def\ZZ{\mathbb Z}

\def\sB{\mathsf B}\def\sC{\mathsf C}\def\sD{\mathsf D}
\def\sE{\mathsf E}\def\sH{\mathsf H}
\def\sI{\mathsf I}
\def\sN{\mathsf N}\def\sO{\mathsf O}
\def\sS{\mathsf S}

\def\bDelta{\mathbf\Delta}

	\nc{\hiro}{\textcolor{blue}}
	\nc{\david}{
	\textcolor{red}
	}

\begin{document}

\title{Factorization homology of stratified spaces}
\author{David Ayala}
\author{John Francis}
\author{Hiro Lee Tanaka}
\date{}

\address{Department of Mathematics\\Montana State University\\Bozeman, MT 59717}
\email{david.ayala@montana.edu}
\address{Department of Mathematics\\Northwestern University\\Evanston, IL 60208-2370}
\email{jnkf@northwestern.edu}
\address{Department of Mathematics\\Harvard University\\Cambridge, MA 02138-2901}
\email{hirolee@math.harvard.edu}
\thanks{DA was partially supported by ERC adv.grant no.228082, and by the National Science Foundation under Award No. 0902639. JF was supported by the National Science Foundation under Award 0902974 and Award 1207758. HLT was supported by a National Science Foundation Graduate Research Fellowship, by the Northwestern University Office of the President, by the Centre for Quantum Geometry of Moduli Spaces, and by the National Science Foundation under Award DMS-1400761.}

\begin{abstract} 
This work forms a foundational study of factorization homology, or topological chiral homology, at the generality of stratified spaces with tangential structures. Examples of such factorization homology theories include intersection homology, compactly supported stratified mapping spaces, and Hochschild homology with coefficients. Our main theorem characterizes factorization homology theories by a  generalization of the Eilenberg--Steenrod axioms; it can also be viewed as an analogue of the Baez--Dolan cobordism hypothesis formulated for the observables, rather than state spaces, of a topological quantum field theory. Using these axioms, we extend the nonabelian Poincar\'e duality of Salvatore and Lurie to the setting of stratified spaces -- this is a nonabelian version of the Poincar\'e duality given by intersection homology. We pay special attention to the simple case of singular manifolds whose singularity datum is a properly embedded submanifold and give a further simplified algebraic characterization of these homology theories. In the case of 3-manifolds with 1-dimensional submanifolds, these structure gives rise to knot and link homology theories.
\end{abstract}



\maketitle

\tableofcontents

\section*{Introduction}

The present work forms an initial step in a formalism for topological quantum field theory using stratified spaces. This larger program proposes to understand locality in topological quantum field theory by a fusion of the algebra of factorization homology and the geometry of stratifications. Before delineating this step and, briefly, this future program, we first review the subject of factorization homology.

\medskip

Factorization homology is, heuristically, a procedure which takes an $n$-manifold $M$ and an algebraic input $A$, such as an $\cE_n$-algebra, and produces an object $\int_MA$. The manifold provides the gluing data, the algebra provides the gluing rules, and one can think of integrating the multiplication of the algebra over the gluing data of the manifold. Viewing this object as an invariant of the manifold, the procedure generalizes usual homology theories; viewing the object as an invariant of the algebra, it generalizes Hochschild homology when the manifold is the circle and offers a natural repository for traces or index-type invariants.

\medskip

Such a procedure was introduced for algebraic varieties by Beilinson \& Drinfeld in their work on an algebro-geometric formalism for conformal field theory; see \cite{bd} and \cite{fg}. In \S5.5 of~\cite{HA}, Lurie defined a topological analogue of their construction -- known as factorization homology or topological chiral homology -- and this topological construction likewise generalizes the labeled configuration spaces of Salvatore \cite{salvatore} and Segal \cite{segallocal}. The main theorem of this area, non-abelian Poincar\'e duality, naturally generalizes the James construction and configuration space models of mapping spaces dating to the 1970s in work of McDuff \cite{mcduff} and others -- see \cite{Fact} for a more detailed history.

\medskip

Another recent catalyst for study in this area has been the approach of Costello \& Gwilliam to perturbative quantum field theory in \cite{kevinowen}. In mathematical approaches to topological field theory at least since Atiyah in \cite{atiyah}, it is common to organize the formalism around the functoriality of the state spaces in the theory. This choice leads to cobordisms and, proceeding deeper, higher categories of cobordisms after Baez \& Dolan in \cite{baezdolan}. In contrast, Costello \& Gwilliam, following the factorization algebra structures of \cite{bd}, codify their theory around the structure of observables, or operators, rather than state spaces. In their work, the earlier renormalization machinery of \cite{kevin} is married with the factorization point of view; the intuitive factorization homology procedure becomes a way of constructing a candidate object $\int_MA$ of global observables on a space-time $M$ from the algebra of observables $A$ on Euclidean space. This candidate object $\int_MA$ is intended to accurately capture the global observables on $M$ if the quantum field theory is perturbative. They prove a quantization theorem, designed as a mathematical formulation of physicists methods of perturbative renormalization, using the Batalin--Vilkovisky formalism applied to derived symplectic geometry. From this quantization theorem they construct a number of interesting examples of such perturbative field theories.

\medskip

The structural favoring of observables over state spaces has geometric consequences, namely a favoring of open embeddings over cobordisms. Dual to restricting fields, observables can extend by zero. This is unlike state spaces, where there is no procedure for extension by zero, no naturality with respect to open embeddings, and no values for non-compact manifolds. The cobordism hypothesis with singularities, after Baez--Dolan \cite{baezdolan} and Lurie \cite{cobordism}, gives a proposed classification for certain topological field theories in terms of their state spaces. With this in mind, one can ask if there is a similar classification which applies to this class of topological field theories constructed in \cite{kevinowen} in terms of the Batalin--Vilkovisky formalism, in terms of their observables. Our first main result, Theorem \ref{hmlgy=FH}, can be viewed as just such a classification.

\medskip

In the present work, we lay the foundations for a general theory of factorization homology, following the outline of \cite{Fact} and after the originating work of Lurie in \cite{HA}. We do this for stratified spaces and, more generally, $\cB$-manifolds, where $\cB$ is a collection of basic singularity types endowed with a tangential structure, applying the theory of stratified spaces and tangential structures developed in \cite{aft1}. This extra level of generality is carried out for two reasons. First, the theory without stratifications is related to observables on space-time in perturbative field theory; adding nontrivial stratifications allows one to incorporate boundary conditions and defects in this theory, such as Wilson line operators in Chern--Simons. The second reason concerns the extension of our theory outside of the perturbative range in quantum field theories; our larger program in progress (see \cite{striation} and the papers that follow) uses stratifications to effect this greater generality.

\medskip

We now turn to a linear overview of contents of the current work, which has three parts.

\medskip

In the first part, we cover the definition (Definition \ref{def:fact-hmlgy}) of factorization homology as a Kan extension from $\disk(\cB)$-algebras to $\cB$-manifolds, and we prove its existence. The first main result is that the symmetric monoidal and underlying left Kan extensions are equivalent, and that there is therefore a comprehensible formula (Theorem \ref{fact-explicit}) which computes factorization homology; this generalizes a formula for usual homology. To establish this existence result and this explicit formula we prove a general result giving conditions for existence and agreement of symmetric monoidal and underlying left Kan extensions in Lemma \ref{shape-existence}. To verify the conditions of this lemma requires proving that the $\oo$-category $\disk(\cB)_{/X}$ of basic singularity types embedded in a manifold is sifted, Corollary \ref{quittersifted}; our proof makes use of a localization result relating discrete and topological categories of embedded disks and of Dugger--Isaksen's work in \cite{Dugger--Isaksen}.

\medskip

Together with this explicit existence, we prove in Theorem \ref{:(} the existence of push-forwards for factorization homology along constructible bundles of stratified spaces. An immediate consequence is a Fubini theorem for factorization homology, Corollary \ref{fubini}. Along with an analysis of factorization homology in the case of a closed interval, which we identify as a relative tensor product in Proposition \ref{tensor-prod}, these results culminate in a main structural theorem. To state it, we require some terminology: let $\cB$ be an $\oo$-category of basic singularity types and $\mfld(\cB)$ the collection of stratified spaces locally modeled on $\cB$; let $\cV$ be a symmetric monoidal $\oo$-category which is $\ot$-sifted cocomplete (see Definition \ref{def-siftedcocplt}). Below, we next define the collection of $\cV$-valued homology theories $\bcH(\mfld(\cB), \cV)$ as a full $\infty$-subcategory of symmetric monoidal functors $\Fun^{\ot}(\mfld(\cB),\cV)$ satisfying the following symmetric monoidal generalization of the Eilenberg--Steenrod axioms.

\begin{definition}\label{homology}[Definition \ref{homology}]
The $\infty$-category of \emph{homology theories (over $X$)} is the full $\infty$-subcategory
\[
\bcH\bigl(\mfld(\cB)_{/X},\cV\bigr)~\subset~\Fun^\ot\bigl(\mfld(\cB)_{/X},\cV\bigr)
\]
consisting of those $H$ that satisfy the following two properties:
\begin{itemize}
\item {\bf $\ot$-Excision:} 
Let $W \cong W_- \underset{\RR\times W_0} \bigcup W_+$ denote a collar-gluing among $\cB$-manifolds over $X$.
Then the canonical morphism~(\ref{exc-compare})
\begin{equation}\label{ot-excision}
H(W_-)\underset{H(W_0)}\bigotimes H(W_+)\xra{~\simeq~} H(W)
\end{equation}
is an equivalence in $\cV$.
\item {\bf Continuous:}
Let $W_0\subset W_1\subset \dots \subset X$ be a sequence of open sub-stratified spaces of $X$ with union denoted as $\underset{i\geq 0} \bigcup W_i = : W$.
Then then the canonical morphism
\begin{equation}\label{continuous}
\colim\Bigl(H(W_0) \to H(W_1)\to \dots \Bigr)\xra{~\simeq~} H(W)
\end{equation}
is an equivalence in $\cV$.

\end{itemize}
Absolutely, the $\infty$-category of \emph{homology theories (for $\cB$-manifolds}) is the full $\infty$-subcategory
\[
\bcH\bigl(\mfld(\cB),\cV\bigr)~\subset~\Fun^\ot\bigl(\mfld(\cB),\cV\bigr)
\]
consisting of those $H$ for which, for each $\cB$-manifold $X$, the restriction $H_{|\mfld(\cB)_{/X}}$ is a homology theory for $X$.  

\end{definition}

The preceding definition generalizes to stratified spaces a notion introduced in \cite{Fact}; see also the introduction of that work for a discussion of the present concepts, such as $\ot$-excision. We can now state our main result, which generalizes the main result of {\it loco citato} from manifolds to stratified spaces.

\begin{theorem}[Theorem \ref{hmlgy=FH}]
There is an equivalence between $\disk(\cB)$-algebras in $\cV$ and $\cV$-valued homology theories for $\cB$-manifolds
\[\Alg_{\disk(\cB)}(\cV) \simeq \bcH(\mfld(\cB),\cV)\] defined by sending a $\disk(\cB)$-algebra $A$ to the factorization homology $\int A$.
\end{theorem}

Here the $\ot$-excision property of the homology theory becomes a formulation of the locality of the field theory. The theorem includes a relative version, for a fixed $\cB$-manifold $X$, which in this interpretation is a fixed space-time. Thus, our result both generalizes the Eilenberg--Steenrod axioms and attempts to axiomatize the structure of the observables in a perturbative topological quantum field theory. This latter topic also lies in the domain of the cobordism hypothesis, or the cobordism hypothesis with singularities of \cite{cobordism}. A direct comparison between these approaches is possible, subject to suitable construction of the functor $Z$, below, proposed in the cobordism hypothesis. Namely, there should be a commutative diagram below (subject to a suitable construction of the dashed functors, which we do not provide in the present work):
\[\xymatrix{
\Alg_{\disk_n}(\cV)^{\sim}\ar[d]_{\int}\ar[r]&\bigl(\Alg_{n}(\cV)^{\sim}\bigr)^{\sO(n)}\ar@{-->}[d]^Z\\
\bcH(\mfld(\cB), \cV)^{\sim} \ar@{-->}[r]& \Fun^\ot({\sf Bord}_n, \Alg_{n}(\cV))\\
}\]
\noindent

We briefly explain the terms in this picture: $\disk_n$ and $\mfld$ are the $\oo$-categories of $n$-disks and $n$-manifolds and embeddings; ${\sf Bord}_n$ is the $(\oo,n)$-category of bordisms of manifolds from \cite{cobordism}; and $\Alg_{n}(\cV)$ is the higher Morita category, where $k$-morphisms are framed $(n-k)$-disk algebras in bimodules; the superscript $\sim$ denotes that we have taken the underlying $\oo$-groupoids, restricting to invertible morphisms. The bottom dashed arrow from homology theories valued in $\cV$ to topological quantum field theories valued in $\Alg_{n}(\cV)$, assigns to a homology theory $\cF$ the functor on the bordism category sending a $k$-manifold $M$ to $\cF({M}^\circ\times\RR^{n-k})$, the value of $\cF$ on a thickening of the interior of $M$. There is a similar diagram replacing $n$ with a general  collection of basic singularity types $\cB$, where on the right we substitute bordism categories with singularities.

\medskip

We assert that this diagram commutes, but this assertion relies on the existence of the higher bordism $(\oo,n)$-category ${\sf Bord}_n$ and on the verification of the cobordism hypothesis, i.e., the existence of the dotted arrow $Z$. A proof by Lurie has been outlined in \cite{cobordism}, building on earlier work with Hopkins. We defer the commutativity of this diagram to a future work, after the full completion of the right hand side.

\medskip

We turn to a description now of the the second part of this work. There, we apply our main result, Theorem \ref{hmlgy=FH}, to give a proof of non-abelian Poincar\'e duality, after \cite{salvatore}, \cite{segallocal}, and Section~5.5.6 of~\cite{HA}. Further, our Theorem \ref{non-abel} is a stratified generalization, that there is a homotopy equivalence
\[\int_X A_E \simeq \Gammac(X,E)\]
between a factorization homology theory and a space of sections of a stratified bundle $E$ over a $\cB$-manifold $X$, subject to stratum by stratum connectivity conditions on $E$. The left hand side can be thought of either as a labeled configuration space or a form of nonabelian homology with coefficients in a higher loop space. 

\medskip

In the last section of this paper, we detail the structure involved in particular choices of $\cB$-structures. We observe that the intersection homology theories of Goresky \& MacPherson fit our axiomatics. As such, our stratified theories generalize intersection homology in the way that the unstratified factorization homology of \cite{HA} and \cite{Fact} generalize ordinary homology. We then study two examples: manifolds with boundary and $n$-manifolds with a submanifold of fixed dimension $d$. Both of these have descriptions as special cases of our main theorem, and each of these descriptions mix with the form of the higher Deligne conjecture proved by Lurie \cite{HA}, which asserts that the higher Hochschild cohomology $\hh^\ast_{\sD^{\fr}_d}(B)$ of a framed $d$-disk algebra $B$ carries a framed $(d+1)$-disk algebra structure which is characterized by a universal property. In particular, we prove the following by combining our main theorem and this form of Deligne's conjecture:

\begin{cor} There is an equivalence 
\[
\int\colon \Alg_{\disk_{d\subset n}^{\fr}}(\cV)\leftrightarrows\bcH(\mfld^{\fr}_{d\subset n}, \cV)\colon \rho
\] 
between $\disk_{d\subset n}^{\fr}$-algebras in $\cV$ and $\cV$-valued homology theories for framed $n$-manifolds with a framed $d$-dimensional submanifold with trivialized normal bundle. The datum of a $\disk_{d\subset n}^{\fr}$-algebra is equivalent to the data of a triple $(A,B,\alpha)$, where $A$ is a $\disk_n^{\fr}$-algebra, $B$ is a $\disk_d^{\fr}$-algebra, and $\alpha:\int_{S^{n-d-1}}A\ra \hh^\ast_{\sD^{\fr}_d}(B)$ is a map of $\disk^{\fr}_{d+1}$-algebras.
\end{cor}

Specializing to the case of 3-manifolds with a 1-dimensional submanifold, i.e., to links, the preceding provides an algebraic structure that gives rise to a link homology theory. To a triple $(A, B, \alpha)$, where $A$ is a $\disk_3^{\fr}$-algebra, $B$ is an associative algebra, and $\alpha: \hh_\ast (A) \ra \hh^\ast (B)$ is a $\disk_2^{\fr}$-algebra map, one can then construct a link homology theory, via factorization homology with coefficients in this triple. This promises to provide a new source of such knot homology theories, similar to Khovanov homology. Khovanov homology itself does not fit into this structure, for a very simple reason: an open embedding of a noncompact knot $(U, K) \subset (U', K')$ does not define a map on Khovanov homologies, from ${\sf Kh}(U,K)$ to ${\sf Kh}(U',K')$. (In addition, Khovanov homology still has not been constructed for general manifolds $K$, other than $S^3$ or $\RR^3$.) factorization homology theories can be constructed, however, using the same input as Chern--Simons theory, and these appear to be closely related to Khovanov homology -- these theories will be the subject of another work, one which requires a more involved use of stratifications to capture nonperturbative phenomena.

\begin{remark}
In this work, we use Joyal's {\it quasi-category} model  of $\oo$-category theory \cite{joyal}. 
Boardman \& Vogt first introduced these simplicial sets in \cite{bv}, as weak Kan complexes, and their and Joyal's theory has been developed in great depth by Lurie in \cite{HTT} and \cite{HA}, our primary references; see the first chapter of \cite{HTT} for an introduction. We use this model, rather than model categories or simplicial categories, because of the great technical advantages for constructions involving categories of functors, which are ubiquitous in this work. 
More specifically, we work inside of the quasi-category associated to this model category of Joyal's.  In particular, each map between quasi-categories is understood to be an iso- and inner-fibration; and (co)limits among quasi-categories are equivalent to homotopy (co)limits with respect to Joyal's model structure.

We will also make use of $\Kan$-enriched categories, such as $\snglr$ of stratified spaces and conically smooth embeddings among them. 
By a functor $\cS \ra \cC$ to an $\infty$-category from a $\Kan$-enriched category we will always mean a functor $\sN\cS \ra \cC$ from the simplicial nerve of $\cS$.  

The reader uncomfortable with the language of $\infty$-categories can substitute the words ``topological category" for ``$\oo$-category" wherever they occur in this paper to obtain the correct sense of the results, but they should then bear in mind the proviso that technical difficulties may then abound in making the statements literally true. The reader only concerned with algebra in chain complexes, rather than spectra, can likewise substitute ``pre-triangulated differential graded category" for ``stable $\oo$-category" wherever those words appear, with the same proviso.
\end{remark}

 \medskip

{\bf Acknowledgements.} We are indebted to Kevin Costello for many conversations and his many insights which have motivated and informed the greater part of this work. We also thank Jacob Lurie for illuminating discussions, his inspirational account of topological field theories, and his substantial contribution to the theory of $\infty$-categories. JF thanks Alexei Oblomkov for helpful conversations on knot homology.

\section{Recollections}\label{recollections}
This entire work is founded on a predecessor, \cite{aft1}, which lays our foundation for structured stratified spaces and the $\infty$-categories organizing them.
In this first subsection, we briefly recall some of the concepts introduced there, skipping many specifics. 
The results mentioned in this section will be precise, but we refer any reader who wishes to truly understand the precise definitions to~\cite{aft1}. 
In the second subsection we review some essentials of symmetric monoidal $\infty$-categories, after \S2 of~\cite{HA}.

\subsection{Stratified spaces}\label{sec.stratified-spaces}

Before defining \emph{conically smooth} stratified spaces which occupy this work, we first review the baseline topological notions.
We will regard a poset $P$ as a topological space with closed sets $P_{\leq p}$, for each element $p\in P$.
A {\it stratified topological space indexed by a poset $P$} is a paracompact Hausdorff topological space $X$ with a continuous map $X\ra P$.
For many purposes, the indexing poset $P$ can be taken to be the totally ordered set $[n]=\{0<1<\ldots <n\}$, where $n$ is the maximal dimension of a neighborhood of a point in $X$.
A \emph{stratified topological space} $X$ is a stratified topological space indexed by a poset $P$ (typically omitted from the notation).
The {\it stratum} $X_p\subset X$ is the inverse image of the element $p\in P$.
The \emph{depth} of a stratified topological space is the depth of its stratifying poset, by which we mean the maximum among lengths of strictly increasing sequences, should it exist.
A map between stratified topological spaces $X\ra Y$ is a commutative diagram of topological spaces
\[
\xymatrix{
X\ar[r]\ar[d]&Y \ar[d]\\
P\ar[r]& Q}
\]
where $P$ and $Q$ are the indexing posets for $X$ and $Y$. Such maps are closed under composition, so this forms a category.
We say such a map is an \emph{open embedding} if $X\ra Y$ is an open embedding between topological spaces.

\begin{example}
We observe an essential operation for generating new stratified topological spaces from old ones: taking cones. First, for $P$ is a poset, then the left-cone $P^\tl$ on $P$ is the poset defined by adding a new minimum element to $P$.
Note an identification $[n]^\tl = [n+1]$. For $X$ a stratified topological space indexed by $P$, then the open cone on $X$,
\[
\sC(X):=\{0\} \underset{\{0\}\times X}\coprod [0,\infty)\times X \longrightarrow \{0\} \underset{\{0\}\times P}\amalg [1]\times P = P^{\tl}~,
\]
is a stratified topological space indexed by $P^\tl$. 
Note that the cone $\sC(X)$ carries a natural action by multiplication of the nonnegative reals $\RR_{\geq 0}$, by scaling in the cone coordinate.
Note also that the stratifying poset $P^{\tl}$ has strictly greater depth than the poset $P$, provided the latter exists.
\end{example}

The category of {\it $C^0$ stratified space} is the smallest full subcategory of stratified topological spaces characterized by the following properties:

	\begin{enumerate}
	\item The empty set $\emptyset$ is a $C^0$ stratified space indexed by the empty poset.
	\item If $X$ is a $C^0$ stratified space and both $X$ and its indexing poset $P$ are compact, then the open cone $\sC(X)$ is a $C^0$ stratified space.
	\item If $X$ and $Y$ are $C^0$ stratified spaces, then the product stratified space $X\times Y$ is a $C^0$ stratified space. 
	\item  If $X$ is a $C^0$ stratified space and $(U \to P_{|U}) \to (X \to P)$ is an open embedding, then $(U \to P_{|U})$ is a $C^0$ stratified space.
	\item  If $X$ is a stratified topological space admitting an open cover by $C^0$ stratified spaces, then $X$ is a $C^0$ stratified space.
	\end{enumerate}	
Observe these examples of $C^0$ stratified spaces:
\begin{itemize}
\item a singleton, $\ast= (\ast \to \ast)$;
\item a half-open interval, $[0,\infty) = \bigl([0,\infty)\to [1]\bigr)$, with stratification such that the $0$-stratum is precisely $\{0\}$;
\item Euclidean spaces, $\RR^i = (\RR^i\to \ast)$;
\item any $C^0$ manifold.  

\end{itemize}

We now add a {\it conical smoothness} condition to our stratifications. The assumption of conical smoothness will be present throughout this work. This is, unfortunately and perhaps ineluctably, an inductive definition.
The induction is on the \emph{depth} of the stratifying poset.

A \emph{conically smooth stratified space}, or \emph{stratified space} for short, is a $C^0$ stratified space $X$ that is equipped with a \emph{conically smooth atlas}
\[
\Bigl\{\RR^{i_\alpha}\times \sC(Z_\alpha) \hookrightarrow X\Bigr\}_\alpha
\]
by \emph{basics}.
We now explain these terms.
Each $Z_\alpha$ is a compact stratified space of depth strictly less than that of $X$; so we assume that we have already defined what it means for $Z_\alpha$ to be equipped with a conically smooth atlas.
In general, a \emph{basic} is the data of a non-negative integer $i$ and a compact stratified space $Z$; together, these data define the $C^0$ stratified space $\RR^i\times \sC(Z)$, which might also be referred to as a \emph{basic} when the conically smooth atlas on $Z$ is understood.
Note that, for $P$ the stratifying poset of $Z$, the stratifying poset of $\RR^i\times \sC(Z)$ is $P^{\tl}$; the \emph{cone-locus} of this basic is the stratum $\RR^i = \RR^i\times \ast$, which is that indexed by the adjoined minimum.
Continuing toward our definition of a stratified space, a \emph{conically smooth atlas} is a collection of basics openly embedding as $C^0$ stratified spaces into $X$.  This collection satisfies three conditions: 
\begin{enumerate}
\item This collection of open embeddings forms a basis for the topology of $X$, in particular it forms an open cover of $X$.
\item The \emph{transition maps}, by which we mean the inclusions of open embeddings $\RR^i\times \sC(W) \hookrightarrow \RR^j\times \sC(Z)$ over $X$, are \emph{conically smooth}, which we explain momentarily.
\item This collection is maximal with respect to (1) and (2).  
\end{enumerate}
To complete our definition of a stratified space, it remains to explain the condition for an open embedding between basics to be \emph{conically smooth}.
So consider an $C^0$ open embedding $f\colon \RR^i\times \sC(Y) \hookrightarrow \RR^j\times \sC(Z)$ between basics.
This open embedding $f$ is \emph{conically smooth} in the sense of~\S3.3 of~\cite{aft1} if the following conditions are satisfied.
\begin{itemize}
\item {\bf Away from the cone-locus:}
Each element of the atlas $\psi\colon \RR^k\times \sC(W) \hookrightarrow Y$ determines a composite open embedding $\RR^i\times(0,\infty)\times \RR^k \times \sC(W) \hookrightarrow \RR^i\times \sC(Y) \hookrightarrow \RR^j\times \sC(Z)$.
By way of a smooth identification $(0,\infty)\cong \RR$, we recognize the domain of this composition as a basic.  
Using that $f$ is an open embedding between stratified topological spaces, necessarily this is the data of an open embedding $f\psi\colon \RR^{i+1+k}\times \sC(W) \hookrightarrow \RR^j\times \sC(Z)\smallsetminus \RR^j$ to the complement of the cone-locus.  
As so, for each such $\psi$, the condition of \emph{conical smoothness} requires this open embedding $f\psi$ is a member of the atlas of the target, which carries meaning via the induction in the definition of a stratified space.

\item {\bf Along the cone-locus:}
Should the preimage of the cone-locus be empty, then the above point entirely stipulates the conical smoothness condition on $f$.  
So suppose the preimage of the cone-locus is not empty.
Using that $f$ is an open embedding between stratified topological spaces, upon representing the values of $f = \bigl[(f_{\parallel},f_r,f_\theta)\bigr]$ as coordinates, then $f_r(p,0,y) = 0$ for all $(p,y)\in \RR^i\times Y$.  
For $0< k <\infty$, the condition that $f$ is \emph{conically $C^k$ (along the cone-locus)} requires the assignment
\begin{equation}\label{derivative}
(p,v,s,y)~\mapsto~\underset{t\to0^+} {\sf lim} \Bigl[\bigl(\frac{f_{\parallel}(p+tv,ts,y) - f_{\parallel}(p,0,y)}{t}, \frac{f_r(p+tv,ts,y)}{t}, f_\theta(p+tv,ts,y) \bigr)\Bigr] 
\end{equation}
to be defined as a map $D_{\parallel}f\colon T\RR^i\times \sC(Y) \to T\RR^j\times \sC(Z)$, required to be conically $C^{k-1}$ along the cone-locus (which carries meaning via induction on $k$).  
The condition of \emph{conical smoothness} requires that $f$ is conically $C^k$ for all $k\geq 0$.  
\end{itemize}

Outlined above is a definition of a stratified space, in the sense of~\cite{aft1}.  
We now, more briefly, outline the definition of a \emph{conically smooth} map between two; the assumption of conical smoothness on maps between stratified spaces will be present throughout this work, unless otherwise stated.  
Like the definition of a stratified space, a conically smooth map between two stratified spaces is offered by induction on depth.  
For $X$ and $Y$ stratified spaces, a map $f\colon X\to Y$ between their underlying stratified topological spaces is \emph{conically smooth} if, for each pair of basic charts $\phi\colon \RR^i\times \sC(Y) \hookrightarrow X$ and $\psi\colon \RR^j\times \sC(Z)\hookrightarrow Y$ for which there is a containment $f(\phi(\RR^i)) \subset \psi(\RR^j)$ of the images of the cone-loci, the map $\psi^{-1}\circ f \circ \phi\colon \RR^i\times \sC(Y) \to \RR^j\times \sC(Z)$ is \emph{conically smooth}.
This latter use of the term means conically smooth \emph{away from the cone-locus}, which can be ensured to carry meaning via induction on depth, and conically smooth \emph{along the cone-locus} which means $f$ abides by expression~(\ref{derivative}).

\begin{example}
Note that smooth manifolds are precisely those stratified spaces that have no strata of positive codimension.
If $g:M\ra N$ is a smooth map between compact smooth manifolds, then the map of cones $\sC(g):\sC(M)\ra \sC(N)$ is conically smooth. If additionally $h:\RR^i \ra \RR^j$ is a smooth map between Euclidean spaces, then the product map $h\times \sC(g): \RR^i \times \sC(M)\ra \RR^j\times \sC(N)$ is again conically smooth.
\end{example}
Let $f\colon X\to Y$ be a map between stratified spaces.
We single out three important classes of conically smooth maps:
\begin{itemize}
\item Say $f$ is an \emph{embedding} if it is an isomorphism onto its image, and that it is an \emph{open embedding} if it is an embedding and it is an open map.  
\item Say $f$ is a \emph{refinement} if it is a homeomorphism of underlying topological spaces, and its restriction to each stratum of $X$ is an embedding.  
\item Say $f$ is a \emph{constructible bundle} if the following is true:
\begin{itemize}
\item[~]
Let $Y \to Q$ be the stratification of $Y$.
Then for each $q\in Q$, the restriction to the $q$-stratum $f_|\colon X_{|f^{-1}Y_q} \to Y_q$ is a fiber bundle of stratified spaces. 
\end{itemize}
\end{itemize}
\noindent
Unlike in the preceding definitions, for this next notion of {\it weak} constructibility the map $f$ does not preserve the stratification (i.e., it does not lie over a map of posets).

\begin{itemize}
\item Say a continuous map $f\colon X \to Y$ between the underlying topological spaces of stratified spaces is a \emph{weakly constructible bundle} if there is a diagram among stratified spaces $X\xla{r} \w{X} \xra{\ov{f}} Y$, with $r$ a refinement and $\ov{f}$ a constructible bundle, for which there is an equality between maps of underlying topological spaces: $\ov{f} = f\circ r$.  
\end{itemize}
Conically smooth maps compose, as do those that are open embeddings, thereby yielding the pair of categories
\[
\stratd~\supset~\snglrd~.
\]
The category $\stratd$ admits finite products, and so we regard it as enriched over the Cartesian category $\Fun(\stratd^{\op},\Set)$ of set-valued presheaves on itself through the expression 
\[
\Map_{\stratd}(X,Y) \colon Z\mapsto \stratd_{/Z}(X\times Z, Y\times Z)
\]
where the subscript ${}_{/Z}$ indicates those maps that commute with projecting to $Z$.  
Restricting along the standard cosimplicial object $\Delta^\bullet_e\colon \bDelta \longrightarrow \stratd$ given by $[p]\mapsto \{[p]\xra{t}\RR\mid \sum_{i\in [p]} t_i = 1\}$, gives a natural enrichment of $\stratd$, and thereafter of $\snglrd$, over simplicial sets.
In~\cite{aft1} we show that these enrichments factors as a $\Kan$-enrichments.
In a standard manner, we obtain $\infty$-categories
\[
\snglr\qquad\text{ and }\qquad \Strat
\]
of stratified spaces and \emph{spaces} of open embeddings among them, and respectively of conically smooth maps among them.   
Manifestly, there are natural functors $\snglrd \to \snglr$ and $\stratd \to \Strat$.
There is a full subcategory and an $\infty$-subcategory,
\[
\iota \colon \bscd~\subset~\snglrd\qquad\text{ and }\qquad \iota \colon \bsc~\subset~\snglr
\]
consisting of the \emph{basics}.  
A key result is that the functor $[-]\colon \bsc\to [\bsc]$ to the poset of isomorphism classes is \emph{conservative} (i.e., it does not create isomorphisms), which has immense consequences that are specific to the set-up herein. (See Theorem 4.3.1 of~\cite{aft1}.)

The \emph{tangent classifier} is the restricted Yoneda functor
\[
\snglr\xra{~j~} \Psh(\snglr)\xra{~\iota^\ast~} \Psh(\bsc)~,\qquad X\mapsto \bigl({\sf Entr}(X)\xra{\tau_X} \bsc\bigr)
\]
where, here, we have taken the model for presheaves on $\infty$-categories as \emph{right fibrations}.  
Specifically, each stratified space $X$ determines an $\infty$-category ${\sf Entr}(X):= \bsc_{/X}$ over $\bsc$.
(A key result in~\cite{aft1} is that ${\sf Entr}(X)$ is identified as the \emph{enter-path category} of the stratified space $X$, though we do not explain this result here for we do not use it.)
Imitating the notion of a tangential structure in differential topology, we define a \emph{category of basics} as a right fibration
\[
\cB ~=~(\cB \to \bsc)
\]
and declare a $\cB$-manifold to be a stratified space $X$ together with a lift of its tangent classifier:
\[
\xymatrix{
&&
\cB \ar[d]
\\
{\sf Entr}(X) \ar[rr]^-{\tau_X}  \ar[urr]^-g
&&
\bsc.
}
\]
A $\cB$-manifold $(X,g)$ will typically be denoted simply as $X$.  
More precisely, we define the pullback $\infty$-categories
\[
\xymatrix{
\mfldd(\cB)  \ar[r]  \ar[d]
&
\mfld(\cB)  \ar[r]  \ar[d]
&
\Psh(\bsc)_{/\cB}  \ar[d]
\\
\snglrd  \ar[r]
&
\snglr \ar[r]^-{\tau}
&
\Psh(\bsc)
}
\]
and refer to the upper middle as that of $\cB$-manifolds.

\begin{remark}\label{g-sheaf}
Via the straightening/unstraightening construction (see Section 2.2 of~\cite{HTT} -- this is the $\infty$-categorical analogue of the Grothendieck construction), a right fibration $\cB\to \bsc$ is the same data as a presheaf of spaces: $\bsc^{\op} \to \spaces$.  
Precomposing this presheaf with the tangent classifier for $X$ determines another presheaf of spaces: ${\sf Entr}(X)^{\op}\xra{\tau_X} \bsc^{\op} \xra{\cB} \spaces$.  
Via Theorem~A.9.3 of~\cite{HA}, such a presheaf is the data of a sheaf of spaces on $X$ that happens to be constructible with respect to the given stratification of $X$.  
As so, a $\cB$-manifold is a stratified space $X$ together with a section $g$ the sheaf named just above.  

\end{remark}

Definitionally, there is a fully faithful inclusion 
\[
\iota \colon \cB \hookrightarrow \mfld(\cB)~.
\]
A main result of~\cite{aft1} is that $\mfld(\cB)$ is generated by $\cB$ through the formations of \emph{collar-gluings} and \emph{sequential unions}:
\begin{itemize}
\item A \emph{collar-gluing} is a weakly constructible bundle $X\xra{f} [-1,1]$.  
We typically denote such a collar-gluing as $X\cong X_-\underset{\RR\times X_0} \bigcup X_+$ where $X_- := f^{-1}[-1,1)$ and $X_0 := f^{-1}(0)$ and $X_+ := f^{-1}(-1,1]$; we regard $X$ as the collar-gluing of $X_-$ and $X_+$ along $X_0$.  

\item A \emph{sequential union} is a sequence of open subspaces $X_0\subset X_1\subset \dots \subset X$ of a stratified space for which $\underset{i\geq 0} \bigcup X_i = X$. 

\end{itemize}

It is also useful to single out this notion of a finitely-presented $\cB$-manifolds, namely one which is generated from $\cB$ only by collar-gluings.
We say that an $\oo$-subcategory $\cM$ of $\mfld(\cB)$ is {\it closed under the formation of collar-gluings} if $X$ belongs to $\cM$ whenever $X_0$, $X_-$, and $X_+$ belong to $\cM$, where $X\ra [-1,1]$ is a weakly constructible bundle with those inverse images.
The $\oo$-category $\mfld(\cB)^{\sf fin}$ of {\it finitary} $\cB$-manifolds is the smallest full $\oo$-subcategory of $\mfld(\cB)$ which contains $\cB$ and is closed under collar-gluings.

\begin{example}\label{union-collar-gluing}
Disjoint unions are instances of collar-gluings.
That is, if $\cB$ is a category of basics, and $X:= X_- \sqcup X_+$ is a disjoint union of $\cB$-manifolds, then the map $f:X\ra [-1,1]$, given by $f_{|X_\pm} \equiv \pm 1$, is a collar-gluing.  
\end{example}

We conclude this subsection with a remark on our use of weakly constructible bundles.

\begin{remark}
Wherever we use weakly constructible bundles in this work, we will only work with a single map at a time. That is, we will not consider any sort of category or $\oo$-category of stratified spaces whose morphisms are weakly constructible bundles. This is due to a technical issue: it is unclear how to directly topologize the set of weakly constructible bundles (and then how to ensure a well-defined composition on these spaces of maps). This issue is dealt with in a successor work, \cite{striation}, by the construction of the $\oo$-category ${\sf c}\cB{\sf un}$. This is an $\oo$-category of compact stratified spaces whose morphisms are proper contructible bundles over the asymmetrically stratified interval $\{0\}\subset [0,1]$. A main result of that paper implies that such a proper constructible bundle $X\ra [0,1]$ is equivalent to a weakly constructible bundle map $X_1 \ra X_0$ from the generic fiber to the special fiber. Consequently, one could {\it define} the $\oo$-category of compact stratified spaces and weakly constructible bundles among them to be ${\sf c}\cB{\sf un}^{\op}$, but we have chosen not to incorporate this later work into the present paper.
\end{remark}

\subsection{Symmetric monoidal $\oo$-categories}
We review some aspects of symmetric monoidal $\infty$-categories, as well as $\infty$-operads, that we will make use of later on.
The content of this section is essentially extracted from \S2 of~\cite{HA}.  

We use the notation $\Fin_\ast$ for the category of based finite sets.
We let $\Fin$ denote the category of (possibly empty) finite sets, and $\Fin^{\sf inj}$ the subcategory of only injective morphisms.
There is a functor $(\Fin^{\sf inj})^{\op} \xra{\sf inert} \Fin_\ast$ given on objects as $I\mapsto I_+$ (i.e., attaching a disjoint basepoint) and on morphisms as $(f\colon I\hookrightarrow J)\mapsto (f^+:J_+\to I_+)$ where $f^+$ is specified by requiring both that the diagram of finite sets
\[
\xymatrix{
I_+  \ar[r]^-=  \ar[d]_-{f_+}
&
I_+  
\\
J_+  \ar[r]_-=  
&
J_+  \ar[u]_-{f^+}
}
\] 
commutes and that $f$ sends $J\smallsetminus f(I)$, the complement of the image of $f$, to the basepoint $+$.  
We denote the essential image of ${\sf inert}$ as $\Fin_\ast^{\sf inrt}\subset \Fin_\ast$ and refer to its morphisms as \emph{inert} maps -- these are the maps for which the inverse image of a point (which is not the basepoint) consists of at most a single element. The category $\Fin^{\sf inj}\simeq (\Fin_\ast^{\sf inrt})^{\op}$ carries the standard Grothendieck topology in which the covers are surjective maps.

\begin{defn}[After Definition 2.1.1.10 of~\cite{HA}]\label{def.operad}
An $\infty$-operad is a functor $\cO \to \Fin_\ast$ that satisfies the following points:
\begin{itemize}
\item For each inert map $I_+\xra{f} J_+$, and for each lift $O\in \cO_{|I_+}$, there is a coCartesian morphism $O \xra{\w{f}} O'$ lifting $f$. 
In particular, the restriction $\cO_{|\Fin^{\sf inrt}_\ast} \to \Fin_\ast^{\sf inrt}$ is a coCartesian fibration.

\item  The coCartesian fibration $\cO_{|\Fin^{\sf inrt}_\ast} \to \Fin^{\sf inrt}_\ast \cong (\Fin^{\sf inj})^{\op}$ is a $\Cat$-valued sheaf whose value on $+$ is terminal.

\item Let $(O_i)_{i\in I} \in \cO_{\ast_+}^{\times I} \simeq \underset{i\in I} \prod \cO_{\{i\}_+} \xla{\simeq} \cO_{|I_+}$ be an object, and likewise let $(O'_j)_{j\in J} \in \cO_{|J_+}$ be an object.
The canonical map of spaces 
\[
\Map_{\cO}\bigl((O_i)_{i\in I}, (O'_j)_{j\in J}\bigr) \xra{~\simeq~}\underset{I_+\xra{f}J_+}\coprod\underset{j\in J}\prod \Map_{\cO}\bigl((O_i)_{f(i)=j}, O_j\bigr)
\]
is an equivalence.  

\end{itemize}
Let $\cO$ be an $\infty$-operad.
We say $\cO$ is \emph{unital} if $\cO$ has an initial object.
The \emph{active $\infty$-subcategory}
\[
\cO_{\sf act} ~:=~ \cO_{|\Fin}
\]
is the restriction of $\cO$ along the functor $\Fin\xra{(-)_+}\Fin_\ast$ that adjoins a disjoint base point.  
An active morphism in $\cO$ whose target lies over $\ast_+\in \Fin_\ast$ is called a \emph{multi-morphism}.
\end{defn}

\begin{remark}
Definition~\ref{def.operad} of an $\infty$-operad is an $\infty$-categorical version of a \emph{colored operad}, also known as a \emph{multi-category}.
In that terminology, each object of the fiber $\cO_{|\ast_+}$ is a \emph{color}.  

\end{remark}

\begin{notation}\label{sym-cat-same}
Let $\cO \to \Fin_\ast$ be an $\infty$-operad. We will typically only carry the notation $\cO$ for such an $\infty$-operad.  We will refer to its restriction to $\ast_+$ as the \emph{underlying $\infty$-category}, and again denote it as $\cO$. Context should prevent confusion.
\end{notation}

\begin{remark}
Recall that giving a coCartesian fibration $\cO \to \fin_\ast^{\sf inrt}$ is equivalent to giving a functor $\fin_\ast^{\sf inrt} \to \Cat$ by the straightening/unstraightening construction (see Section 2.2 of~\cite{HTT} -- this is the $\infty$-categorical analogue of the Grothendieck construction). Further, the sheaf condition guarantees that the fiber over the based finite set $I_+$ is an $\infty$-category equivalent to an $I$-fold product of the fiber over $+$.
\end{remark}

\begin{defn}[Definition 2.1.2.7 of~\cite{HA}]
For $\cO$ and $\cP$ $\infty$-operads, the $\infty$-category of \emph{$\cO$-algebras in $\cP$} is the full $\infty$-subcategory
\[
\Alg_\cO(\cP)~\subset~\Fun_{\Fin_\ast}(\cO,\cP)
\]
consisting of those functors over based finite sets that preserve inert-coCartesian morphisms. 
There is an $\infty$-category ${\sf Operad}_\infty$ whose objects are $\infty$-operads and whose space of morphisms from $\cO$ to $\cP$ is the underlying $\infty$-groupoid of $\Alg_\cO(\cP)$. (See Definition 2.1.4.1 of~\cite{HA}.)
\end{defn}

\begin{defn}[Definition~2.0.0.7 of~\cite{HA}; see also Remark~2.1.2.19 of~\cite{HA}.]\label{def.sym-mon}
A \emph{symmetric monoidal $\infty$-category} is a coCartesian fibration $\cV \to \Fin_\ast$ whose restriction $\cV_{|\Fin^{\sf inrt}_\ast} \to \Fin^{\sf inrt}_\ast$ is a $\Cat$-valued sheaf whose value on $+$ is terminal.
(In particular, a symmetric monoidal $\infty$-category is an $\infty$-operad.)
Let $\cV$ be a symmetric monoidal $\infty$-category.
A \emph{symmetric monoidal unit} $\uno_\cV$ for $\cV$ is a target of a morphism in $\cV$ over $+\to \ast_+$.  
For $\cD$ another symmetric monoidal $\infty$-category, the $\infty$-category of \emph{symmetric monoidal functors (from $\cV$ to $\cD$)} is the full $\infty$-subcategory
\[
\Fun^\ot(\cV,\cD)~\subset~\Fun_{\Fin_\ast}(\cV,\cD)
\]
consisting of those functors over based finite sets that preserve coCartesian morphisms.  
There is an $\infty$-category ${\sf Cat}^\ot_\infty$ whose objects are symmetric monoidal $\infty$-categories and whose space of morphisms from $\cV$ to $\cD$ is the underlying $\infty$-groupoid of $\Fun^\ot(\cO,\cV)$. (See Variant 2.1.4.13 of~\cite{HA}.)
\end{defn}

\begin{lemma}\label{creates-limits} 
The forgetful functor 
\[
{\sf Operad}_\infty \longrightarrow \Cat_{/\Fin_\ast}
\]
creates limits.
The underlying $\infty$-category functor
\[
{\sf Cat}^\ot_\infty \longrightarrow \Cat
\]
is conservative and creates limits.  
\end{lemma}

\begin{proof}
The first statement can be verified directly.
Consider a small diagram $\cK\to {\sf Operad}_\infty$, denoted $k\mapsto \cO_k$, and a limit cone $\cK^\tl \to \Cat_{/\Fin_\ast}$.  Denote the value of the cone-point as $\cO$.
Let $I_+\xra{f} J_+$ be an inert map, and let $O \in (\cO)_{|I_+}$ be an object over $I_+$.  
Consider the composite functor 
\[
\cK \to {\sf Operad}_\infty \to \Cat_{/\Fin_\ast} \xra{\{O\}\underset{(-)_{|I_+}}\times \Fun^{\sf coCart}\bigl([1],(-)\bigr)_{|f}} \Cat
\]
given by assigning to an object $k$ the category of coCartesian lifts of $O_k\xra{\w{f}_k} O_k'$ in $\cO_k$ from the image of $O$.
This functor takes values in terminal $\infty$-categories, and so the limit too is terminal and it is equipped with a functor to 
\[
\underset{k\in \cK} \lim\Bigl( \{O\}\underset{(\cO_k)_{|I_+}} \times \Fun\bigl([1],\cO_k\bigr)_{|f} \Bigr)\xla{\simeq} \{O\}\underset{\cO_{|I_+}} \times \Fun\bigl([1],\cO\bigr)_{|f}~.
\]
By design, a morphism $O\xra{\w{f}} O'$ in the essential image of this functor is coCartesian.  

Now consider a functor $\cO_{|I_+} \to \underset{i\in I}\prod \cO_{\{i\}_+}$ induced by the diagram of inert morphisms $(I_+ \xra{a_i} \{i\}_+)_{i\in I}$.  
That this functor is an equivalence follows because limits commute with finite products.
Likewise, for each pair of objects $(O_i)_{i\in I}\in  \cO_{|I_+}$ and $(O'_j)_{j\in J} \in \cO_{|J_+}$, consider the map of spaces
\[
\Map_\cO\bigl((O_i)_{i\in I}, (O'_j)_{j\in J}\bigr) \longrightarrow \underset{j \in J} \prod \Map_{\cO}\bigl( (O_i)_{\{f(i)=j\}}, O_j\bigr)~. 
\]
That this map is an equivalence follows again because limits commute with finite products.  

The second statement is easier to find the literature.  
By Lemma~3.2.2.6 of~\cite{HA}, a forgetful map from a category of algebras over an operad is both conservative and limit-creating. Symmetric monoidal $\infty$-categories are commutative algebra objects in the symmetric monoidal category $\Cat$, hence Lemma~3.2.2.6 of~\cite{HA} applies.
\end{proof}

\begin{remark}
Let $F: \cO \to \cV$ be a symmetric monoidal functor. Since $F$ preserves coCartesian morphisms, it induces a map between the sheaves on $\Fin^{\sf inj}$ associated to $\cO$ and $\cV$. Hence for every $f: I_+ \to J_+$, we obtain a diagram $[1]\times [1] \to \Cat$ as follows:
 \[
 \xymatrix{
 \cO^I  \ar[d]_-{f_\ast}  \ar[rr]^-{F^I}
 &&
 \cV^I  \ar[d]^-{f_\ast}
 \\
 \cO^J  \ar[rr]^-{F^J}
 &&
 \cV^J.
 }
\]
For $J=\ast$, one can interpret the diagram as specifying an equivalence $\bigotimes_\cV^I F(O_i) \simeq F(\bigotimes_\cO^I O_i)$.
\end{remark}

\begin{remark}
Let $\cV$ and $\cD$ be symmetric monoidal $\infty$-categories.
There is an evident full inclusion $\Fun^\ot(\cV,\cD) \subset \Alg_{\cV}(\cD)$ -- it is typically not essentially surjective. The latter corresponds to lax monoidal functors.
\end{remark}

For each finite set $I$, there is then a canonical diagram of $\infty$-categories
\[
\ot \colon \cV^I \xla{~\simeq~}\cV_{I_+}  \xra{\cV_{(I_+ \to \ast)}} \cV~.
\]

\begin{defn}\label{def-siftedcocplt}

For $\cK$ a small $\infty$-category, the symmetric monoidal structure of $\cV$ {\it distributes over $\cK$-shaped colimits} if for each $c\in \cV$, the composite functor \[c\ot -\colon \cV \simeq \ast \times \cV \xra{ c\times {\sf id}_\cV} \cV\times \cV \xra{\ot} \cV\] commutes with colimits of $\cK$-shaped diagrams.  
We say $\cV$ is \emph{$\ot$-sifted cocomplete} if its underlying $\infty$-category admits sifted colimits and its symmetric monoidal structure distributes over sifted colimits, where an $\oo$-category $\cK$ is {\it sifted} if it is nonempty and the diagonal $\cK\ra \cK\times \cK$ is a final functor.

\end{defn}

\begin{lemma}[Section~2.2.4 of~\cite{HA}]\label{sym-env}
The forgetful functor admits a left adjoint
\[
{\sf Env}\colon {\sf Operad}_\infty ~\rightleftarrows~ {\sf Cat}^\ot_\infty~,
\]
referred to as the \emph{symmetric monoidal envelope} functor.  
\end{lemma}

\begin{observation}\label{extending-endos}
There is a canonical filler in the diagram among $\infty$-categories:
\[
\xymatrix{
{\sf Cat}_\infty^\ot\ar@{-->}[rr]^-{(-)^{\op}}  \ar[d]
&&
{\sf Cat}_\infty^\ot  \ar[d]
\\
\Cat  \ar[rr]^-{(-)^{\op}}
&&
\Cat.
}
\]
In other words, the opposite of the underlying $\infty$-category of a symmetric monoidal $\infty$-category is canonically endowed with a symmetric monoidal structure.  
\end{observation}

\begin{example}
Here are some examples of symmetric monoidal $\infty$-categories that are $\ot$-sifted cocomplete.  
\begin{itemize}
\item $({\sf Ch}_\Bbbk,\oplus)$ and $({\sf Ch}_\Bbbk,\ot)$: chain complexes over a ring $\Bbbk$ with equivalences given by quasi-isomorphisms, equipped with direct sum, or with tensor product.  
\item $({\sf Spectra}, \vee)$ and $({\sf Spectra}, \wedge)$: spectra with equivalences generated by stable homotopy equivalences, equipped with wedge sum, or with smash product.  
\item $(\cX,\times)$: any cocomplete Cartesian closed $\infty$-category (for instance $\spaces$ or $\Cat$) with finite product.  
\item Let $\cV$ be any symmetric monoidal $\infty$-category whose underlying $\infty$-category admits small colimits and whose symmetric monoidal structure distributes over small colimits.
Then, for $\cO$ any $\infty$-operad, the $\infty$-category of algebras $\Alg_{\cO}(\cV)$ inherits a standard symmetric monoidal structure which is given pointwise, and this symmetric monoidal $\infty$-category is $\ot$-sifted cocomplete.  
\end{itemize}
An example of a symmetric monoidal $\infty$-category that is \emph{not} $\ot$-sifted cocomplete is $({\sf Ch}_\QQ^{\op}, \ot)$, for $\ot$ does not distribute over totalizations.  

\end{example}

We are about to define an $\infty$-operad $\cV_{/c}$ for $\cV$ a unital symmetric monoidal $\infty$-category and $c\in \cV$ an object in its underlying $\infty$-category.  
Recall the construction from \S2.4.3 of \cite{HA} of an $\infty$-operad $\cD^\amalg$ from an $\infty$-category $\cD$.  Quickly, a functor $\cK \to \cD^\amalg$ is the datum of a functor $\cK\underset{\Fin_\ast}\times \Fin_{\ast, \ast'} \to \cD$ where the second factor in the fiber product is the category of finite sets equipped with an inclusion from the two-element set $\{\ast,\ast'\}$ and maps among such that preserve such inclusions.  
Should $\cD$ be an ordinary category, then the $\cD^{\amalg}$ is an ordinary colored operad with one color for each object of $\cD$, while a multi-morphism from $(d_1,\dots,d_n)$ to $d$ is simply a collection of morphisms $d_i\to d$ for each $1\leq i \leq n$.   

\begin{lemma}\label{unit-initial-cocart}
For $\cO =(\cO\to \Fin_\ast)$ a unital $\infty$-operad,
there is a preferred map of $\infty$-operads
\[
\cO \longrightarrow (\cO_{|\ast_+})^\amalg~.
\]

\end{lemma}

\begin{proof}
We will use the notation $\Fin^{\sf inj}\subset\Fin \xra{(-)_+} \Fin_\ast$ for the category of finite sets and injections among them.  
Consider the category $\Fun_{\sf inj}\bigl([1], \Fin_\ast\bigr)$ consisting of those functors from the two-object totally ordered set $[1]=\{0<1\}$ which classify \emph{injective} maps among based finite sets, and natural transformations among them.   
Evaluation at $0$ gives a functor $\Fun_{\sf inj}\bigl([1], \Fin_\ast\bigr)\xra{{\sf ev}_0} \Fin_\ast$.  
Notice the evident functor $\Fin_{\ast,\ast'}\to \Fun_{\sf inj}\bigl([1], \Fin_\ast\bigr)$, and notice that it factors as an equivalence $\Fin_{\ast,\ast'} \xra{\simeq} {\sf ev}_0^{-1}(\ast_+)$.  

The $\infty$-operad $\cO$ being unital grants that the underlying $\infty$-category $\cO_{|\ast_+}$ has an initial object.  
It follows that the restricted projection $\cO_{|\Fin^{\sf inj}} \to \Fin^{\sf inj}$ is a Cartesian fibration.  
In particular, there is a canonical filler in the diagram of $\infty$-categories
\[
\xymatrix{
\cO \underset{\Fin_\ast} \times \Fun_{\sf inj}\bigl([1], \Fin_\ast\bigr)     \ar@{-->}[rr]  \ar[dr]^{{\sf ev}_0}
&&
\cO  \ar[dl]
\\
&
\Fin_\ast
&
}
\]
in which the implicit functor $\Fun_{\sf inj}\bigl([1], \Fin_\ast\bigr) \xra{{\sf ev}_1} \Fin_\ast$ is evaluation at $1$.
Restricting the second factor in the above fiber product to $\Fin_{\ast,\ast'}$ gives the functor $\cO \to (\cO_{|\ast_+})^\amalg$.  
It is direct to verify that this functor sends inert-coCartesian morphisms to inert-coCartesian morphisms.  

\end{proof}

\begin{cor}\label{slice-operads}
Let $\cO$ be a unital $\infty$-operad.
Let $\cE_{|\ast_+} \to \cO_{|\ast_+}$ be a right fibration over the underlying $\infty$-category.
We denote the pullback $\infty$-category
\[
\xymatrix{
\cE  \ar[r]  \ar[d]
&
(\cE_{|\ast_+})^\amalg  \ar[d]
\\
\cO  \ar[r]
&
(\cO_{|\ast_+})^\amalg ~.
}
\]
After Lemma~\ref{creates-limits}, the composite functor $\cE \to \cO \to \Fin_\ast$ makes $\cE$ into an $\infty$-operad, and the functor $\cE \to \cO$ is a map of $\infty$-operads.  

\end{cor}

\begin{notation}\label{slice-convention}
Let $\cD \to \cM$ be a symmetric monoidal functor between symmetric monoidal $\infty$-categories, each for which the symmetric monoidal unit is initial.  
Let $X\in \cM$ be an object of the underlying $\infty$-category.
We use the notation
\[
\cD_{/X}
\]
for the output of Corollary~\ref{slice-operads} applied to the right fibration $\cD_{/X}\to \cD$.  
By construction, this $\infty$-operad is equipped with an $\oo$-operad map to $\cD$ for which the functor on underlying $\infty$-categories is the standard projection from the slice.  
Say a morphism of $\cD_{/X}$ is \emph{pre-coCartesian} if its image in $\cD$ is coCartesian over $\Fin_\ast$.  
We adopt the following notational convention: for $\cV$ a symmetric monoidal $\infty$-category, the $\infty$-category of $\cD_{/X}$-algebras in $\cV$ is the full $\infty$-subcategory
\[
\Alg_{\cD_{/X}}(\cV)~\subset~ \Fun_{\Fin_\ast}\bigl(\cD_{/X}, \cV\bigr)
\]
consisting of those functors that send pre-coCartesian morphisms to coCartesian morphisms.  

\end{notation}

\section{Factorization}

In this section, we define the factorization homology of $\cB$-manifolds with coefficients in $\disk(\cB)$-algebras, and we verify some of the essential properties of this construction. It is a symmetric monoidal functor from $\cB$-manifolds and embeddings.

For this section we fix:
\begin{itemize}
\item a symmetric monoidal $\infty$-category $\cV$ that is $\ot$-sifted cocomplete,
\item an $\oo$-category of basics $\cB = (\cB \to \bsc)$.
\end{itemize}

\subsection{Disk algebras}
Factorization homology will have a universal property with respect to monoidal and operadic structures on $\mfld(\cB)$ and $\mfld(\cB)_{/X}$, which we present below in Constructions \ref{B-monoidals} and \ref{B-operadics}.

\begin{construction}[The symmetric monoidal structures on $\mfld(\cB)$ and $\disk(\cB)$]\label{B-monoidals}
Using disjoint union, we endow $\snglrd$, $\snglr$, $\mfldd(\cB)$, and $\mfld(\cB)$ with natural symmetric monoidal structures. The details are as follows: Disjoint union makes each of the $\Kan$-enriched categories $\snglrd$ and $\snglr$ into a symmetric monoidal $\Kan$-enriched category in the usual sense.
We then realize each as a symmetric monoidal $\infty$-category (by taking the simplicial nerve, for instance -- see Proposition 2.1.1.27 of~\cite{HA}). 

To further endow $\mfldd(\cB)$ and $\mfld(\cB)$ with symmetric monoidal structures, recall that they are defined via limit diagrams as below
\[\xymatrix{
\mfldd(\cB)\ar[r]\ar[d]&\mfld(\cB)\ar[r]\ar[d]&\psh(\bsc)_{/\cB}\ar[d]\\
\snglrd\ar[r]&\snglr\ar[r]^-\tau&\psh(\bsc)\\}\]
where the functor $\tau$ is the tangent classifier, i.e., the restriction of the Yoneda embedding to basics. 
In light of Lemma~\ref{creates-limits}, to endow $\mfld(\cB)$ with a symmetric monoidal structure it suffices to lift the diagram among $\infty$-categories $\snglr \ra \psh(\bsc)\la\psh(\bsc)_{/\cB}$ to a diagram among symmetric monoidal $\infty$-categories; this is easily accomplished. 
Both $\psh(\bsc)$ and $\psh(\bsc)_{/\cB}$ are cocomplete and hence carry a coCartesian monoidal structure (i.e., the monoidal structure given by coproducts); since the forgetful functor $\psh(\bsc)_{/\cB}\ra \psh(\bsc)$ preserves colimits, it in particular preserves coproducts, and is thus symmetric monoidal with respect to this coCartesian monoidal structure.  
Now, because each basic $U$ is connected, there is a canonical isomorphism of Kan complexes
\[
\underset{i\in I} \coprod \snglr(U,M_i) \xra{~\cong~} \snglr\bigl(U,\underset{i\in I} \bigsqcup M_i\bigr)~.  
\]
So $\snglr \xra{\tau} \psh(\bsc)$ canonically extends as a symmetric monoidal functor.

Finally, since $\snglrd\ra\snglr$ preserves disjoint union, through the same logic $\mfldd(\cB)$ obtains a symmetric monoidal structure.

\end{construction}

\begin{definition}\label{sym-mon-mflds}
The symmetric monoidal $\oo$-categories
\[
\disk(\cB)~\subset~\mfld(\cB)
\qquad\text{ and }\qquad
\diskd(\cB)~\subset~\mfldd(\cB)
\]
are smallest full symmetric monoidal subcategories containing $\cB\subset \mfld(\cB)$ and $\cB\underset{\mfld(\cB)}\times \mfldd(\cB) \subset \mfldd(\cB)$, respectively.
\end{definition}

\begin{remark}
The objects of $\disk(\cB)$ are disjoint unions of elements of $\cB$.
\end{remark}
The slice categories $\disk(\cB)_{/X}$ and $\mfld(\cB)_{/X}$ do not carry symmetric monoidal structures -- for instance, there is no natural way to take a disjoint union of two embeddings if the embeddings overlap. 
Regardless, we endow these categories with the structure of an $\infty$-operad. As usual (see Section~2.1 of~\cite{HA}), one can think of these $\infty$-operads as colored operads, one color for every embedding of a finite disjoint union of basics $U \into X$.

\begin{construction}[$\disk(\cB)_{/X}$ as an $\infty$-operad]\label{B-operadics}
The empty stratified space $\emptyset$ is a symmetric monoidal unit for $\snglrd$ and $\snglr$, and it is also initial in each of the associated underlying $\infty$-categories.
It follows that the same is true for the symmetric monoidal structures on $\mfldd(\cB)$ and $\mfld(\cB)$.
Let $X$ be a $\cB$-manifold.
Notation~\ref{slice-convention} produces the $\infty$-operads $\mfldd(\cB)_{/X}$ and $\mfld(\cB)_{/X}$ as well as $\diskd(\cB)_{/X}$ and $\disk(\cB)_{/X}$.

\end{construction}

\begin{example}
The fiber of the $\infty$-operad $\snglr_{/X}$ over the finite set $\ast_+$ is an $\infty$-category whose objects are conically smooth open embeddings $(Y \into X)$, where $Y$ is a connected stratified space. A morphism $(Y \into X) \to (Y' \into X)$ is specified by a conically smooth open embedding $Y \into Y'$ and an isotopy from $Y \into X$ to the composite $Y \into Y' \into X$. A multi-morphism from $\bigl((Y_1 \into X), (Y_2 \into X)\bigr)$ to $(Y \into X)$ is a conically smooth open embedding $Y_1\sqcup Y_2\hookrightarrow Y$ from the disjoint union, together with a pair of isotopies $\gamma_\nu \colon Y_\nu \times [0,1]\to X$ from the given $Y_\nu \into X$ to the composition $Y_\nu \into Y_1\sqcup Y_2\into X$.  
Notice that this last composition makes use of $\emptyset$ being both initial and the symmetric monoidal unit.  
\end{example}

\begin{remark}
We caution the reader about a notational conflict: There are equivalences
\[
\bigl(\disk(\cB)_{/X}\bigr)_{|\ast_+}~\simeq ~ \disk(\cB)_{/X}\qquad \text{ and }\qquad \bigl(\mfld(\cB)_{/X}\bigr)_{|\ast_+}~\simeq ~ \mfld(\cB)_{/X}~.
\]   
Here, the lefthand side of each equivalence is the underlying $\infty$-category of the $\infty$-operad defined in Construction~\ref{B-operadics}, while the righthand side is the $\infty$-category of disks/$\cB$-manifolds equipped with a map to $X$.
Furthermore, the maps of $\oo$-operads
\[
\disk(\cB)_{/X} \to \disk(\cB)\qquad\text{ and }\qquad \mfld(\cB)_{/X} \to \mfld(\cB)
\]
restrict to the standard functors on underlying $\infty$-categories.  

\end{remark}

\begin{observation}\label{B-independent}
Let $X= (X,g)$ be a $\cB$-manifold.
Then the $\infty$-operad $\disk(\cB)_{/X}$ is independent of the $\cB$-structure on the underlying stratified space of $X$.  
More precisely, because $\cB \to \bsc$ is a right fibration, then the map
\[
\disk(\cB)_{/X} \xra{~\simeq~} \disk(\bsc)_{/X}
\]
is an equivalence of $\infty$-operads.  
For the same reason, for $(U\hookrightarrow X)$ a point of $\disk(\cB)_{/X}$, there is a canonical identification of $\infty$-categories
\[
\disk(\cB)_{/U}~{}~\simeq~{}~\bigl(\disk(\cB)_{/X}\bigr)_{/(U\hookrightarrow X)}~.
\] The analogous assertions hold if replacing $\disk(\cB)$ by $\mfld(\cB)$.

\end{observation}

\begin{notation}
Defined in~\cite{aft1} is the category of basics $\sD_n$, governing smooth $n$-manifolds, and the various elaborations governing smooth $n$-manifolds with boundary, possibly equipped with a framing.  
We use the shorthand notation
	\[
	\disk_n = \disk(\sD_n)~,
	\qquad
	\disk_n^{\fr} = \disk(\sD_n^{\fr})~,
	\qquad
	\disk_n^{\partial} = \disk(\sD_n^{\partial})~,
	\qquad
	\disk_{d\subset n}^{\fr} = \disk(\sD_{d\subset n}^{\fr})~,
	\]
and likewise for other elaborations on $\sD$. For example, the underlying $\infty$-category of $\disk(\sD_n)$ is the $\infty$-category whose objects are diffeomorphic to a finite disjoint union of standard $\RR^n$, and whose morphisms are given by smooth open embeddings between them. The underlying $\infty$-category of $\disk_n^{\fr}$ has objects finite, disjoint unions of $\RR^n$ together with a framing. The space of morphisms between two framed manifolds has the homotopy type of the space of smooth embeddings together with a choice of homotopy between framings. In all these cases, the symmetric monoidal structure is disjoint union.
\end{notation}

\begin{notation}\label{notation.disks}

Prompted by the upcoming Proposition~\ref{smoothing}, and by Notation~\ref{slice-convention}, we will make the following notational conventions
\[
\Alg_{\diskd(\cB)}(\cV)~:=~\Fun^\ot\bigl(\diskd(\cB),\cV\bigr)\qquad\text{ and }\qquad\Alg_{\disk(\cB)}(\cV)~:=~\Fun^\ot\bigl(\disk(\cB),\cV\bigr)
\]
and refer to their objects as $\diskd(\cB)$-algebras and $\disk(\cB)$-algebras, respectively. 

\end{notation}

\begin{remark}
We caution the reader concerning our non-standard notation, which is summarized as follows.
\begin{itemize}
\item As Notation~\ref{sym-cat-same}, we do not distinguish between the notation for a symmetric monoidal $\infty$-category and its underlying $\infty$-category.
\item As Notation~\ref{slice-convention}, we do not distinguish between the notation for certain slice $\infty$-categories and their corresponding $\infty$-operads.  
\item As Notation~\ref{notation.disks}, we use $\Alg_{\disk(\cB)}(\cV)$ for \emph{symmetric monoidal} functors, not just maps of operads.  
Likewise for the other variants such as that concerning $\mfld(\cB)$.
\item As Notation~\ref{slice-convention} we use $\Alg_{X}(\cV)$ for those maps of $\infty$-operads $\disk(\cB)_{/X} \to \cV$  that send \emph{pre-coCartesian} edges to coCartesian edges.  
Likewise for the other variants such as that concerning $\mfld(\cB)$.

\end{itemize}
\end{remark}

\begin{example}\label{n-disk-algebras}
A $\disk^{\fr}_1$-algebra in $\cV$ is canonically identified as an $A_\infty$-algebra in $\cV$.  

\end{example}

In what follows, we let $\cE_n$ denote the $\infty$-operad of little $n$-disks.

\begin{prop}\label{smoothing}
Let $\cV$ be a symmetric monoidal $\infty$-category.
Let $n$ be a finite cardinality.  
There is a canonical equivalence of $\infty$-categories
\[
\Alg_{\disk_n^{\sf fr}}(\cV)\xra{\simeq}\Alg_{\cE_n}(\cV)~.
\]
With respect to the given $\sO(n)$-action on the righthand side of the above equivalence, there is an equivalence of $\infty$-categories
\[
\Alg_{\disk_n}(\cV) \xra{\simeq} \Alg_{\cE_n}(\cV)^{\sf \sO(n)}
\]
to the $\sO(n)$-invariants.  

\end{prop}
\begin{proof}
The evident $\cE_n$-algebra in $\disk_n^{\sf fr}$ induces a symmetric monoidal functor ${\sf Env}(\cE_n)\to \disk_n^{\sf fr}$ from the symmetric monoidal envelope.  
The mapping homotopy types of both sides of this functor are in terms of the spaces $\conf_I(\RR^n):= \{I\hookrightarrow \RR^n\}$, spaces of injective maps into $\RR^n$ from finite sets.  
Inspecting this symmetric monoidal functor reveals that it is an equivalence.  
This gives the first equivalence.

There is the symmetric monoidal right fibration $\disk_n^{\sf fr} \to \disk_n$, which manifestly factors through the $\sO(n)$-invariants.
The fiber over $\RR^n$ is canonically identified as $\Emb(\RR^n,\RR^n) \simeq \sO(n)$, with the translation action of $\sO(n)$.  
The second equivalence follows.  

\end{proof}

\begin{example}\label{deligne-boundary}
Consider the category of basics $\sD^{\partial,\fr}_n$. Then a $\sD^{\partial,\fr}_n$ manifold is an $n$-manifold possibly with boundary, equipped with a framing of the $n$-manifold and a splitting of the framing on the boundary as a product framing. $\disk_n^{\partial,\fr}$-algebras are equivalent to algebras over the Swiss-cheese operad of \cite{voronov}; an object can be regarded as a triple $(A,B,\alpha)$ of a $\disk^{\fr}_n$-algebra $A$, a $\disk^{\fr}_{n-1}$-algebra $B$, and a map
	\[
	\alpha\colon A\longrightarrow {\sf HC}^\ast_{\sD^{\fr}_{n-1}}(B) := {\m_B^{\disk^{\fr}_{n-1}}}(B,B)
	\] 
which is a map of $\disk^{\fr}_n$-algebras; this reformulation is the higher Deligne conjecture, proved in this generality in~\cite{HA} and~\cite{thomas}.
\end{example}

\subsection{Factorization homology}

The following is the main definition and object of interest in the present paper, that of factorization homology with coefficients in a disk algebra.

Recall Notation~\ref{slice-convention} for the $\infty$-category of $\disk_n$-algebras.  
\begin{definition}[Factorization homology]\label{def:fact-hmlgy}
Let $\cV$ be a symmetric monoidal $\infty$-category, and let $\cB$ be an $\oo$-category of basics.
The \emph{(absolute) factorization homology} functor is a left adjoint to the restriction
\\
\[
\xymatrix{
\int_-\colon \Alg_{\disk(\cB)}(\cV) \ar@(-,u)@{-->}[rr] 
&&
\Fun^\ot\bigl(\mfld(\cB), \cV\bigr)  ~. \ar[ll]
}
\]
For $X$ a $\cB$-manifold, the \emph{(relative) factorization homology} functor is a left adjoint to the restriction
\\
\[
\xymatrix{
\int_-\colon \Alg_{X}(\cV) \ar@(-,u)@{-->}[rr] 
&&
\Alg_{\mfld(\cB)_{/X}}(\cV)  ~. \ar[ll]
}
\]

\end{definition}

The left adjoint in the previous definition need not exist. The essential result provided in the following is that for a large class of targets $\cV$, factorization homology both exists and has a relatively simple expression, agreeing with the left adjoint at the level of underlying $\oo$-categories; i.e., without remembering the monoidal structures.

\begin{theorem}\label{fact-explicit}
Let $\cV$ be a symmetric monoidal $\infty$-category which is $\ot$-sifted cocomplete, and let $\cB$ be an $\oo$-category of basics.
Then each of the absolute and the relative factorization homology functors exists and each is fully faithful, and each evaluates as
\begin{equation}\label{eq.fact-explicit}
\int_X A~\simeq~ \colim\Bigl(\disk(\cB)_{/X} \to \disk(\cB) \xra{A} \cV\Bigr)~\simeq~\disk(\cB)_{/X}\underset{\disk(\cB)} \bigotimes A
\end{equation}
-- here, all terms are as the underlying $\infty$-categories of the respective $\infty$-operads.

\end{theorem}

We give a proof of this result predicated on several results to come. In Lemma~\ref{shape-existence}, we identify the general formal features for symmetric monoidal and operadic left Kan extension to exist and agree with the underlying left Kan extension. In the following sections, we show that these formal features hold in our example of interest.

\begin{proof}[Proof of Theorem \ref{fact-explicit}]
This is an application of Lemmas~\ref{shape-existence} and \ref{shape-existence-operad}.

To apply this result, we need these three facts:
\begin{itemize}
\item Let $X$ be a $\cB$-manifold.  
Then the slice $\infty$-category $\disk(\cB)_{/X}$ is sifted.
This is Corollary~\ref{quittersifted} to come.

\item 
The functor $\disk(\cB)_{/\uno_{\disk(\cB)}} \to \mfld(\cB)_{/\uno_{\mfld(\cB)}}$ is final.  
In our case, this functor is an equivalence, manifestly.  

\item 
Let $X$ and $X'$ be $\cB$-manifolds.
Then the functor $\disk(\cB)_{/X}\times \disk(\cB)_{/X'} \to \disk(\cB)_{/X\sqcup X'}$ is final.
In our case, this functor is actually an equivalence, by inspection.  

\end{itemize}

\end{proof}

\begin{lemma}\label{shape-existence}
Let $\cV$ be a symmetric monoidal $\infty$-category; let $\iota \colon \cB \to \cM$ be a symmetric monoidal functor with small domain and locally small codomain.
Consider the commutative diagram of solid arrows
\[
\xymatrix{
\Fun^\ot(\cB, \cV)  \ar@{-->}@(-,u)[rr]^-{\iota_\natural^\ot}  \ar[dd]
&&
\Fun^\ot(\cM,\cV)  \ar[ll]^-{\iota^\ast}  \ar[dd]
\\
&&
\\
\Fun(\cB, \cV)  \ar@{-->}@(-,u)[rr]^-{\iota_\natural}  
&&
\Fun(\cM,\cV)  \ar[ll]_-{\iota^\ast} 
}
\]
where $\iota^\ast$ is restriction along $\iota$, and the vertical arrows forget that a given functor was symmetric monoidal. Suppose both:
\begin{itemize}
\item[(1)] the underlying $\infty$-category of $\cV$ admits sifted colimits;
\item[(2)] for each $M\in \cM$, the slice $\infty$-category $\cB_{/M}$ is sifted.
\end{itemize}
Then $\iota^\ast$ has a left adjoint $\iota_\natural$ as indicated, which can be calculated as
\begin{equation}\label{lke-expression-1}
\iota_\natural F \colon M\mapsto \colim\bigl(\cB_{/M} \to \cB \xra{F} \cV\bigr)~\simeq~ \iota^\ast M \underset{\cB} \ot F~.
\end{equation}
The last expression is a coend, and we have identified $M\in \cM$ with its image under the Yoneda functor.
In addition, suppose
\begin{itemize}
\item[(3)] the symmetric monoidal structure for $\cV$ distributes over sifted colimits; 

\item[(4)] the functor between slice $\infty$-categories over units 
\[
\cB_{/\uno_\cB} \to \cB_{/\uno_\cM}
\]
is final; and 

\item[(5)] for each pair of objects $M, M'\in \cM$, the tensor product functor 
\[
\bigotimes \colon \cB_{/M} \times \cB_{/M'}  \longrightarrow \cB_{/M\ot M'}
\]
is final.
\end{itemize}
Then there is a left adjoint $\iota_\natural^\ot$ as indicated, and the downward right square commutes.  
If $\iota$ is fully faithful then so are each of $\iota_\natural$ and $\iota_\natural^\ot$.  

\end{lemma}

\begin{proof}
Conditions (1) and (2) grant that each value of the expression~(\ref{lke-expression-1}) exists.  
Lemma 4.3.2.13 of~\cite{HTT} states that these expressions depict a functor as indicated.
Proposition 4.3.3.7 of~\cite{HTT} states that this functor satisfies the universal property of being a left adjoint to $\iota^\ast$.

We now argue that conditions (3) and (4) grant that, for each symmetric monoidal functor $\cB\xra{F} \cV$, and for each based map among finite sets $I_+ \xra{f} J_+$, the diagram of $\infty$-categories 
\[
\xymatrix{
\cM^I  \ar[d]_-{f_\ast}  \ar[rr]^-{(\iota_\natural F)^I}
&&
\cV^I  \ar[d]^-{f_\ast}
\\
\cM^J  \ar[rr]^-{(\iota_\natural F)^J}
&&
\cV^J
}
\]
commutes.
Each based map $f \colon I_+ \to J_+$ is canonically a composition of a surjective active map $f^{\sf surj}$ followed by an injective active map $f^{\sf inj}$ followed by an \emph{inert} map $f^{\sf inrt}$, and so it is enough to verify commutativity for each such class of maps.
The case of inert maps is obvious, because then $f_\ast$ is projection and $(\iota_\natural F)^I$ is defined as the $I$-fold product of functors.   
The case of injective active maps amounts to verifying that $\iota_\natural F$ sends a symmetric monoidal unit to a symmetric monoidal unit.  This follows from Condition~(4) because $F$ does so.  

The case of surjective active maps follows from the case that $f\colon I_+ \to \ast_+$ is given by $i\mapsto \ast$, so that $f_\ast = \bigotimes^I$ is the $I$-fold tensor product.
Well, because $F$ is symmetric monoidal, there is a canonical arrow
$
\iota_\natural F \circ \bigotimes^I 
\longrightarrow 
\bigotimes^I\circ (\iota_\natural F)^I
$
between functors $\cM^I \to \cV$ that we will argue is an equivalence.
This arrow evaluates on $(M_i)_{i\in I}$ as the horizontal one in the following natural diagram inside $\cV$:
\[
\Small
\xymatrix{
\colim\bigl(\cB_{/\underset{i\in I}\bigotimes M_i} \to \cB\xra{F} \cV \bigr)  \ar[rr]
&&
\underset{i\in I}\bigotimes \colim\bigl(\cB_{/M_i} \to \cB \xra{F} \cV\bigr) 
\\
&
\colim\bigl( \underset{i\in I}\prod \cB_{/M_i} \to \cB^I \xra{F^I} \cV^I\xra{\bigotimes^I} \cV \bigr) \ar[ul]^-{(5)}  \ar[ur]_-{(3)}
&
.
}
\]
The arrow labeled by~(3) is an equivalence precisely because of condition~(3). 
Condition~(5) implies that the functor $\prod_{i\in I} \cB_{/M_i} \to \cB_{/\underset{i\in I}\bigotimes M_i}$ is final, for it is a composition of final functors.
It follows that the arrow labeled by~(5) is an equivalence, after observing the following diagram among $\infty$-categories:
\[
\xymatrix{
\underset{i\in I} \prod \cB_{/M_i}  \ar[rr]^-{\bigotimes^I}  \ar[d]
&&
\cB_{/\underset{i\in I} \bigotimes M_i}  \ar[d]
\\
\cB^I  \ar[rr]^-{\bigotimes^I}  \ar[d]_-{F^I}
&&
\cB  \ar[d]^-F
\\
\cV^I \ar[rr]^-{\bigotimes^I} 
&&
\cV~.
}
\]

\end{proof}

Recall Notation~\ref{slice-convention}.  
We now state a likewise result as Lemma~\ref{shape-existence} for the case of the $\infty$-operads of Notation~\ref{slice-convention}.
The proof of this lemma is the same as that for Lemma~\ref{shape-existence}.  

\begin{lemma}\label{shape-existence-operad}
Let $\cV$ be a symmetric monoidal $\infty$-category.
Let $\iota \colon \cD \to \cM$ be a symmetric monoidal functor between symmetric monoidal $\infty$-categories with $\cD$ small and $\cM$ locally small. 
Suppose both the symmetric monoidal unit for each of $\cD$ and $\cM$ is initial.
Let $X\in \cM$ be an object of the underlying $\infty$-category, and consider the $\infty$-operads $\cD_{/X}$ and $\cM_{/X}$ over $\cD$ and $\cM$, respectively.  
Consider the solid commutative sub-diagram of restriction and forgetful functors
\[
\xymatrix{
\Alg_{\cD_{/X}}(\cV)  \ar@{-->}@(-,u)[rr]^-{\iota_\natural^\ot}  \ar[dd]
&&
\Alg_{\cM_{/X}}(\cV)  \ar[ll]^-{\iota^\ast}  \ar[dd]
\\
&&
\\
\Fun(\cD_{/X}, \cV) \ar@{-->}@(-,u)[rr]^-{\iota_\natural}  
&&
\Fun(\cM_{/X},\cV)  \ar[ll]_-{\iota^\ast} 
}
\]
where the superscript ${}^\ast$ denotes the evident restriction and the vertical arrows restrict to active $\infty$-subcategories.
Suppose both:
\begin{itemize}
\item[(1)] the underlying $\infty$-category of $\cV$ admits sifted colimits;
\item[(2)] for each $Z=(Z\to X)\in \cM_{/X}$, the slice $\infty$-category $(\cD_{/X})_{/Z} \simeq \cD_{/Z}$ is sifted.
\end{itemize}
Then $\iota^\ast$ has a left adjoint $\iota_\natural$ as indicated, which can be calculated as
\begin{equation}\label{lke-expression}
\iota_\natural A \colon (Z\to X) \mapsto \colim\bigl(\cD_{/Z} \to \cD_{/X} \xra{A} \cV\bigr)~\simeq~ \iota^\ast (Z\to X)\underset{\cD_{/X}} \ot A
\end{equation}
-- here in the expression of the coend we identify $(Z\to X)\in \cM$ with its image under the Yoneda functor.
In addition, suppose
\begin{itemize}
\item[(3)] the symmetric monoidal structure for $\cV$ distributes over sifted colimits; 

\item[(4)] the functor between slice $\infty$-categories over units 
\[
\cD_{/\uno_\cD} \to \cM_{/\uno_\cM}
\]
is final; 

\item[(5)] for each pair of objects $(Z\to X), (Z'\to X) \in \cM_{/X}$, the tensor product functor 
\[
\bigotimes \colon \cD_{/Z} \times \cD_{/Z'}  \longrightarrow \cD_{/Z\ot Z'}
\]
is final.
\end{itemize}
Then there is a left adjoint $\iota_\natural^\ot$ as indicated, and the downward right square commutes.  
If $\iota$ is fully faithful then so are each of $\iota_\natural$ and $\iota_\natural^\ot$.  
\end{lemma}

\subsection{Disks and finite subsets}
We pause our development of factorization homology to relate the $\infty$-category $\disk(\bsc)_{/X}$, which appears in the Definition~\ref{def:fact-hmlgy} of factorization homology, with the enter-path $\infty$-category of another stratified space associated to $X$: the Ran space of $X$.

Let $i$ be a finite cardinality.  
For $X$ a topological space, the \emph{bounded Ran space} $\Ran_{\leq i}(X)$ is the topological space consisting of subsets $S\subset X$ with bounded cardinality $|S|\leq i$ for which the inclusion $S\hookrightarrow X$ is surjective on connected components;
the topology on $\Ran_{\leq i}(X)$ is the coarsest for which, for each set $I$ with cardinality $|I|=i$ equal to $i$, the map ${\sf Image}\colon X^I \to \Ran_{\leq i}(X)$ is continuous.  
For $f \colon X\hookrightarrow Y$ an open embedding that is surjective on connected components, the evident map $\Ran_{\leq i}(f) \colon \Ran_{\leq i}(X) \to \Ran_{\leq i}(Y)$ is an open embedding.

In~\S3 of~\cite{aft1} we prove the following result, which in particular states that $\Ran_{\leq i}(X)$ inherits the structure of a stratified space from one on $X$.  
We use the notation 
\begin{equation}\label{surj-notation}
\snglr^{\sf surj}~\subset ~\snglr
\end{equation}
for the $\infty$-subcategory with the same objects and with those morphisms that induce surjections on connected components.  
\begin{prop}[Proposition 3.7.5 of \cite{aft1}]\label{ran-singular}
Let $i$ be a finite cardinality.  
There is a functor between $\infty$-categories
\begin{equation}\label{rans}
\Ran_{\leq i}\colon \snglr^{\sf surj} \longrightarrow \snglr
\end{equation}
for which the underlying topological space of each value $\Ran_{\leq i}(X)$ is the bounded Ran space of the underlying space of $X$, and whose value on a morphism $X\xra{f} Y$ is $\Ran_{\leq i}(f)$.  

\end{prop}

The next result gives a description of the topology of $\Ran_{\leq i}(X)$ in terms of that of $X$.  
\begin{lemma}\label{ran-basis}
Let $i$ be a finite cardinality.
Let $X$ be a stratified space.
Consider the collection of open embeddings 
\[
\bigl\{\Ran_{\leq i}(U) \xra{~\Ran(f)~} \Ran_{\leq i}(X)\bigr\}_{\{U\xra{f} X\}}
\]
indexed by conically smooth open embeddings from finite disjoint unions of basics which are surjective on connected components.
This collection of open embeddins forms a basis for the topology of $\Ran_{\leq i}(X)$.  

\end{lemma}
\begin{proof}
Let $S\in \Ran_{\leq i}(X)$ be an element, and let $S\in O\subset \Ran_{\leq i}(X)$ be an open neighborhood.
We must show that there is a disjoint union of basics $U\hookrightarrow X$ openly embedding into $X$ as a surjection on components so that $S$ belongs to the image of the map $\Ran_{\leq i}(U) \to \Ran_{\leq i}(X)$ and this image lies in $O$.  

Fix a set $I$ of cardinality $i$.  
Choose a map $I\xra{f_0}X$ for which the image ${\sf Image}(f_0)=S$ equals $S$.  
From the definition of the topology of $\Ran_{\leq i}(X)$, the subset $O\subset \Ran_{\leq i}(X)$ being open implies the subset
\[
\w{O}:=\bigl\{I\xra{f} X\mid {\sf Image}(f)\in O  \bigr\}~\subset~X^I
\]
is open. 
From the definition of the product topology on $X^I$, for each $s\in S$ there is an open subset $s \in V_{s}\subset X$ so that the product $\underset{i\in I} \prod V_{f_0(i)} \subset \w{O}$ lies inside $\w{O}$.  
From the definition of a stratified space in the sense of~\S3 of~\cite{aft1}, the collection of conically smooth open embeddings from basics into $X$ forms a basis for the topology of $X$.  
Therefore, we can choose each $V_s$ above to be a basic $s\in U_s\hookrightarrow X$.
Because $X$ is in particular Hausdorff, we can make these choices so that the intersection $U_s \cap U_{s'}$ is empty whenever $s\neq s'\in S$.  
We obtain an open embedding $U:=\underset{s\in S}\bigsqcup U_s \hookrightarrow X$ from a finite disjoint union of basics which is surjective on components.  
By construction, this embedding has the desired properties stated at the beginning of this proof.

\end{proof}

Through Proposition~\ref{ran-singular}, and with the notation of~(\ref{surj-notation}), there is the functor between over $\infty$-categories
\begin{equation}\label{continuous-ran-slice}
\Ran_{\leq i}\colon \snglr^{\sf surj}_{/X} \to \snglr^{\sf surj}_{/\Ran_{\leq i}(X)}
\end{equation}
for each stratified space $X$.  
For the next result, we will denote the full $\infty$-subcategory 
\[
\disk(\bsc)^{{\sf surj}, \leq i}_{/X}~\subset~ \snglr^{\sf surj}_{/X}
\]
consisting of those open embeddings $U\hookrightarrow X$ that are surjective on components and for which $U$ has at most $i$ connected components, each of which is isomorphic to a basic stratified space.  

\begin{lemma}\label{EE-to-E}
Let $i$ be a finite cardinality. 
Let $X$ be a stratified space.
The functor~(\ref{continuous-ran-slice}) restricts as an equivalence of $\infty$-categories 
\[
\Ran_{\leq i}\colon \disk(\bsc)^{{\sf surj},\leq i}_{/X} \xra{~\simeq~} \bsc_{/\Ran_{\leq i}(X)}~.  
\]

\end{lemma}

\begin{proof}
The factorization happens because $\Ran_{\leq i}(U)$ is a basic whenever $U$ is an (at most) $i$-fold disjoint union of basics -- this statement is proved in~\S3 of~\cite{aft1}.  
We will argue that this functor is an equivalence by showing it is an equivalence on maximal $\infty$-subgroupoids, then that it is an equivalence on spaces of morphisms.

Keeping with the notation of~\cite{aft1}, for each stratified space $Z$ we will denote the over $\infty$-category
\[
{\sf Entr}(Z)~:=~\bsc_{/Z}
\]
of basic stratified spaces over $Z$.   
In~\S4 of~\cite{aft1} we show that, for $Z$ a stratified space, the maximal $\infty$-subgroupoid of ${\sf Entr}(Z)$ is $\underset{[B]} \coprod {\sf Entr}_{[B]}(Z)$ where the coproduct is indexed by isomorphism classes of basics, and ${\sf Entr}_{[B]}(Z) =  \sB\End(B)_{/Z}$ is the $\infty$-subcategory of those basics over $Z$ which are isomorphic to $B$.   
Also in~\S4 of~\cite{aft1} we show that ${\sf Entr}_{[B]}(Z) \simeq \Emb\bigl(B,Z\bigr)_{\Aut_0(B)}$ is the coinvariants by the origin preserving automorphisms of $B$, and thereafter that this latter space evaluates at the center of $B$ as an equivalence with the underlying space of the $[B]$-stratum $Z_{[B]}\subset Z$.  
In summary, the maximal $\infty$-subgroupoid of ${\sf Entr}(Z)$ is canonically identified as
\[
\underset{[B]}\coprod {\sf Entr}_{[B]}(Z)~{}~\simeq~{}~ \underset{[B]}\coprod Z_{[B]}~.  
\]

Similarly, the maximal $\infty$-subgroupoid of $\disk(\bsc)^{{\sf surj},\leq i}_{/Z}$ is canonically identified as
\[
\underset{[U]} \coprod \underset{\alpha \in A}\prod  \disk(\bsc_{[U_\alpha]})^{{\sf surj},= i_\alpha}_{/X}~{}~\simeq~{}~ \underset{[U]} \coprod \underset{\alpha \in A} \prod \conf_{I_\alpha}(W_{[U_\alpha]})_{\Sigma_{I_\alpha}}
\]
which we now explain.
The coproducts are indexed by isomorphism classes of finite disjoint unions of basics $U = \underset{\alpha \in A} \bigsqcup (U_\alpha)^{\sqcup I_\alpha}$, grouped here according to isomorphism type, for which $\underset{\alpha \in A} \sum |I_\alpha|\leq i$.
$\bsc_{[U_\alpha]}$ is the full $\infty$-subcategory of $\bsc$ consisting of those basics which are isomorphic to $U_\alpha$; and $i_\alpha := |I_\alpha|$ is the cardinality.  
So the left $[U]$-cofactor is the largest $\infty$-subgroupoid of $\disk(\bsc)^{\sf surj}_{/X}$ consisting of those open embeddings $V \hookrightarrow X$ for which $V$ is isomorphic to an $i_\alpha$-fold disjoint union of the basic $U_\alpha$.  
For $M$ a smooth manifold, $\conf_J(M)$ is the underlying space of the open submanifold $M^J$ consisting of the injections $J\hookrightarrow M$.
The subscript denotes the coinvariants by the evident $\Sigma_J$-action.

Thus, the map of maximal $\infty$-subgroupoids induced by the functor $\disk(\bsc)^{{\sf surj},\leq i}_{/X} \xra{\Ran_{\leq i}} \bsc_{/\Ran_{\leq i}(X)}$ is canonically identified as the map of spaces
\[
\underset{[U]}\coprod \underset{\alpha \in A} \prod \conf_{I_\alpha}(X_{[U_\alpha]})_{\Sigma_{I_\alpha}}~{}~\simeq~{}~\underset{[U]} \coprod \underset{\alpha \in A}\prod  \disk(\bsc_{[U_\alpha]})^{{\sf surj},= i_\alpha}_{/X}
\longrightarrow
\underset{[B]}\coprod {\sf Entr}_{[B]}\bigl(\Ran_{\leq i}(X)\bigr)~{}~\simeq~{}~
\underset{[B]}\coprod \bigl(\Ran_{\leq i}(X)\bigr)_{[B]}
\]
in where the indexing set of the left coproduct consists of those such isomorphism classes $[U]$ for which $U$ has at most $i$ components.  
Lemma~\ref{ran-basis} grants that this map is a bijection on the sets indexing the coproducts.  
That this map of spaces is an equivalence then follows because, by detailed inspection, there is a canonical isomorphism of stratified spaces
\[
\underset{\alpha \in A} \prod \conf_{I_\alpha}(X_{[U_\alpha]})_{\Sigma_{I_\alpha}}~{}~\cong~{}~\bigl(\Ran_{\leq i}(X)\bigr)_{[\Ran_{\leq i}(U)]}
\]
whose map of underlying spaces is a summand of the above composite map.

We now consider the map of spaces of morphisms induced by the functor $\disk(\bsc)^{{\sf surj},\leq i}_{/X} \xra{\Ran_{\leq i}} {\sf Entr}\bigl(\Ran_{\leq i}(X)\bigr)$.
This map fits as the top horizontal arrow in the diagram of spaces
\[
\xymatrix{
\Map\bigl([1],\disk(\bsc)^{{\sf surj},\leq i}_{/X}\bigr)  \ar[rr]  \ar[d]_{{\sf ev}_1}
&&
\Map\Bigl([1],{\sf Entr}\bigl(\Ran_{\leq i}(X)\bigr)\Bigr)  \ar[d]^{{\sf ev}_1}
\\
\bigl(\disk(\bsc)^{{\sf surj},\leq i}_{/X}\bigr)^{\sim}  \ar[rr]
&&
\Bigl({\sf Entr}\bigl(\Ran_{\leq i}(X)\bigr)\Bigr)^{\sim}~.
}
\]
We have already shown that the bottom horizontal map is an equivalence of spaces.
To show that the top horizontal map is an equivalence of spaces, it is enough to show that, for each point $(U\hookrightarrow X)$ of the bottom left space, the map of associated fibers is an equivalence.

Let $(U\hookrightarrow X)\in \disk(\bsc)^{{\sf surj},\leq i}_{/X}$.  Both $\disk(\bsc)^{{\sf surj},\leq i}_{/X}$ and ${\sf Entr}\bigl(\Ran_{\leq i}(X))$ being slice-$\infty$-categories, we canonically identify the map of fibers over $(U\hookrightarrow X)$ as the map of maximal $\infty$-subgroupoids
\[
\bigl( \disk(\bsc)^{{\sf surj},\leq i}_{/U} \bigr)^{\sim} \longrightarrow \Bigl({\sf Entr}\bigl(\Ran_{\leq i}(U)\bigr)\Bigr)^{\sim}
\]
induced by $\Ran_{\leq i}$.
We have already shown that this map is an equivalence, and so concludes this proof.

\end{proof}

\subsection{Localizing with respect to isotopy equivalences}
Here we explain that the $\infty$-category $\disk(\cB)_{/X}$, which appears in the Definition~\ref{def:fact-hmlgy} of factorization homology, can be witnessed as a localization of its un-topologized version $\diskd(\cB)_{/X}$ on the collection of those inclusions of finite disjoint unions of disks $U\subset V$ in $X$ for which this inclusion is isotopic to an isomorphism.  
This comparison plays a fundamental role in recognizing certain colimit expressions in this theory, for instance that support the pushforward formula of~\S\ref{sec.pushforward}.

Given a topological space $X$ and a finite set $J$, we let $\conf_J(X)$ denote the topological space of injections from $J$ to $X$.  There is an evident action of the symmetric group $\Sigma_J$ on this space, and we denote the coinvariants as $\conf_J(X)_{\Sigma_J}$. Finally, given a basic $U \cong \RR^n \times C(Z)$, we denote by $\Aut_0(U)$ the Kan complex of isomorphisms $U \to U$ that preserve the origin $(0,\ast) \in U$. (See Section~4.3 of~\cite{aft1}.) In the following lemma, $\sB$ is the classifying space functor.

\begin{lemma}\label{EE-equivs}
For $\cB$ a category of basics, the maximal $\infty$-groupoid of $\disk(\cB)$ is canonically identified as
\[
\bigl(\disk(\cB)\bigr)^\sim~{}~\simeq~{}~ \underset{[U = \underset{i\in I} \bigsqcup U_i^{\sqcup J_i}]} \coprod \prod_{i\in I} \sB\bigl(\Sigma_{J_i} \wr \End_{\cB}(U_i)\bigr)
\]
where the coproduct is indexed by isomorphism classes of finite disjoint unions of objects of $\cB$, whose connected components are grouped here according to isomorphism type.
In particular, the symmetric monoidal functor $[-]\colon \disk(\cB) \to \Fin$ is conservative.  

For $X$ a $\cB$-manifold, the underlying $\infty$-groupoid of $\disk(\cB)_{/X}$ is canonically identified as the space
\[
\bigl(\disk(\cB)_{/X}\bigr)^\sim~{}~\simeq~{}~ \underset{[U = \underset{i\in I} \bigsqcup U_i^{\sqcup J_i}]} \coprod \prod_{i\in I} \conf_{J_i}(X_{[U_i]})_{\Sigma_{J_i}}
\]
given in terms of unordered configuration spaces of various strata of $X$.  

\end{lemma}

\begin{proof}
The statement concerning $\disk(\cB)$ follows immediately from the characterization of $\bsc$ in~\cite{aft1}.  

Upon the canonical equivalence of $\infty$-categories $\disk(\cB)_{/X} \simeq \disk(\bsc)_{/X}$, we can assume $\cB = \bsc$, so that $\mfld(\cB) \simeq \snglr$ is an $\infty$-category associated to a ${\sf Kan}$-enriched category.  
In \S4.3 of~\cite{aft1} it is shown that the inclusion of $\sf Kan$-groups $\Aut_0(U) \to \End_{\bsc}(U)$ is an equivalence of underlying Kan complexes.  
Upon these considerations, to show the first statement it is sufficient to show that, for each finite set $J$, and each basic $U$, the $\Sigma_J$-equivariant map which is evaluation at each center of $U$
\[
{\sf ev}_{(0)_{j\in J}}\colon \Map_{\snglr}\bigl(U^{\sqcup J}, X\bigr)_{\Aut_0(U)^J} ~\xra{~\simeq~}~\conf_J(X_{[U]}) 
\]
is an equivalence of $\Sigma_J$-spaces.  
We do this by induction on the cardinality $|J|$, with the base case $|J|=1$ offered by results in~\cite{aft1}.  
Now suppose $|J|>1$, and choose a non-empty proper subset $J'\subset J$.
Restriction along this inclusion of subsets gives the commutative diagram
\[
\xymatrix{
\Map_{\snglr}\bigl(U^{\sqcup J}, X\bigr)_{\Aut_0(U)^J}  \ar[rr] \ar[d]^-{{\sf ev}_0}
&&
\Map_{\snglr}\bigl(U^{\sqcup J'}, X\bigr)_{\Aut_0(U)^{J'}}  \ar[d]^-{{\sf ev}_0}
\\
\conf_J(X_{[U]})  \ar[rr]
&&
\conf_{J'}(X_{[U]})
}
\]
which is evidently appropriately equivariant.
In a standard manner, the map of fibers of the horizontal maps is canonically identified as the map of $\Sigma_{J\smallsetminus J'}$-spaces
\[
\Map_{\snglr}\bigl(U^{\sqcup J\smallsetminus J'}, X\smallsetminus J'\bigr)_{\Aut_0(U)^{J\smallsetminus J'}}  \xra{~{\sf ev}_0~}  \conf_{J\smallsetminus J'}(X_{[U]}\smallsetminus J')~,
\]
which is an equivalence by induction. 
The result follows.  

\end{proof}

Let $\cB$ be a category of basics; let $X$ be a $\cB$-manifold.
Consider the $\infty$-subcategory 
\begin{equation}\label{eqn.I_X}
\cI_X\subset \diskd(\cB)_{/X}
\end{equation}
consisting of the same objects but only those morphisms $(U\hookrightarrow X) \hookrightarrow (V\hookrightarrow X)$ whose image in $\disk(\cB)_{/X}$ is an equivalence. Note that in~\ref{eqn.I_X}, we think of $X$ as an object of $\mfldd(\cB)$, while when we refer to $\disk(\cB)_{/X}$, we think of $X$ as an object of $\mfld(\cB)$.

\begin{prop}\label{EEd-vs-EE}
The standard functor
\begin{equation}\label{disc-to-cont}
\diskd(\cB)_{/X} \longrightarrow \disk(\cB)_{/X}
\end{equation}
witness an equivalence between $\infty$-operads from the localization (among $\infty$-operads):
\[
\bigl(\diskd(\cB)_{/X}\bigr)[\cI_X^{-1}]~\simeq~\disk(\cB)_{/X}~.
\]
In particular, for each symmetric monoidal $\infty$-category $\cV$, restriction defines a fully faithful functor
\[
\Alg_{\disk(\cB)_{/X}}(\cV)  ~{}~\hookrightarrow~{}~\Alg_{\diskd(\cB)_{/X}}(\cV)
\]
whose image consists of those $\diskd(\cB)_{/X}$ algebras that carry each isotopy equivalence to an equivalence in $\cV$.  
\end{prop}

\begin{proof}
First note that the $\infty$-subcategory $\cI_X\subset \diskd(\cB)_{/X}$ lies over equivalences in $\Fin_\ast$.
As so, the localization $\bigl(\diskd(\cB)_{/X}\bigr)[\cI_X^{-1}]$ (among $\infty$-categories) canonically lies over $\Fin_\ast$.  
Therefore, should the canonical functor $\bigl(\diskd(\cB)_{/X}\bigr)[\cI_X^{-1}]\xra{\simeq}\disk(\cB)_{/X}$ from the localization (among $\infty$-categories) be an equivalence between $\infty$-categories, then this localization among $\infty$-categories is in fact an $\infty$-operad itself; furthermore, this localization among $\infty$-categories also presents the $\infty$-operadic localization. 
We are thusly reduced to showing that the functor~(\ref{disc-to-cont}) is a localization between $\infty$-categories.

Note that the functor~(\ref{disc-to-cont}) is essentially surjective, manifestly.  
We will argue that the resulting functor from the localization is an equivalence by showing it is an equivalence on underlying $\infty$-groupoids, then that it is an equivalence on spaces of morphisms.
We recall that, as in Observation \ref{B-independent}, since $\mfld(\cB) \to \snglr$ is a right fibration there are canonical equivalences of $\infty$-categories
\[
\diskd(\cB)_{/X}~ \simeq~ \diskd(\bsc)_{/X} ~=:~ \sD_X \qquad\text{ and }\qquad \disk(\cB)_{/X}~\simeq~ \disk(\bsc)_{/X} ~=:~ \cD_X
\]
and we adopt the indicated notation for this proof.
Therefore, we can assume $\cB = \bsc$.  

The underlying $\infty$-groupoid of $\sD_X$ is the classifying space $\sB \cI_X$.  
In light of the above coproduct in Lemma~\ref{EE-equivs}, fix an isomorphism type $[U=\underset{i\in I} \bigsqcup U_i^{\sqcup J_i}]$ of a finite disjoint union of a basic stratified space.
Consider the full subcategory $\cI_X^{[U]}\subset \cI_X$ consisting of those $(V\hookrightarrow X)$ for which $V\cong U$.
We thus seek to show that the resulting functor $\cI_X^{[U]} \to \underset{i\in I} \prod \conf_{J_i}(X_{[U_i]})_{\Sigma_{J_i}}$ witnesses an equivalence from the classifying space.  
We will do this by first showing that the functor $\underset{i\in I} \prod \conf_{J_i}(X_{[U_i]})\cap - \colon \cI_X^{[U]} \to \underset{i\in I} \prod \cI^{[(U_i^{\sqcup J_i})_{[U_i]}}_{X_{[U_i]}}$ induces an equivalence on classifying spaces, then observing a canonical equivalence of spaces $\sB \cI^{[(U_i^{\sqcup J_i})_{[U_i]}]}_{X_{[U_i]}} \simeq \conf_{J_i}(X_{[U_i]})_{\Sigma_{J_i}}$.  

Consider the slice category $(\cI_X^{[U]})^{(R_i\hookrightarrow X_{[U_i]})_{i\in I}/}$.
An object is an object $(U'\hookrightarrow X)$ of $\cI_X^{[U]}$ for which, for each $i\in I$, there is an inclusion $R_i\subset U'_{[U_i]}$ which is a bijection on components, and a morphism is an inclusion of such.
Because such $(U'_{[U_i]}\hookrightarrow X_{[U_i]})$ form a base for the topology about $R_i\hookrightarrow X_{[U_i]}$, then this category is filtered. 
In particular, the classifying space $\sB(\cI_X^{[U]})^{(R_i\hookrightarrow X_{[U_i]})_{i\in I}/}\simeq \ast$ is contractible.  
It follows that the functor $\underset{i\in I} \prod \conf_{J_i}(X_{[U_i]})\cap - \colon \cI_X^{[U]} \to \underset{i\in I} \prod \cI^{[(U_i^{\sqcup J_i})_{[U_i]}]}_{X_{[U_i]}}$ induces an equivalence on classifying spaces, as desired.  

Now, let $M$ be a smooth manifold, so that $I=\ast$ is a singleton. 
The category $\cI_M^{[U]}$ forms a basis for the standard Grothendieck topology on $\conf_{J_i}(M)_{\Sigma_{J_i}}$.  
It follows from Corollary 1.6 of~\cite{Dugger--Isaksen} that the canonical map of topological spaces
\[
\underset{(\underset{j\in J_i} \sqcup R_j\hookrightarrow M)\in \cI^{[U]}_M} \colim~ \underset{j\in J_i} \prod R_j~{}~\xra{~\simeq~}~{}~ \conf_{J_i}(M)_{\Sigma_{J_i}}
\]
is a homotopy equivalence.  
Because each term $R_j$ in this colimit is contractible, then this colimit is identified as the classifying space
\[
\sB \cI^{[U]}_M~\simeq~ \underset{(\underset{j\in J_i} \sqcup R_j\hookrightarrow M)\in \cI^{[U]}_M} \colim~ \underset{j\in J_i}\prod R_j~.
\]
Applying this to the case $M=X_{[U_i]}$, we conclude that $\sB \cI^{[(U_i^{\sqcup J_i})_{[U_i]}}_{X_{[U_i]}} \simeq \conf_{J_i}(M)_{\Sigma_{J_i}}$.
In summary, we have verified that the map of underlying $\infty$-groupoids
\[
\bigl(\sD_X[\cI_X^{-1}]\bigr)^{\sim} \xra{~\simeq~} \bigl(\cD_X\bigr)^{\sim}
\]
is an equivalence.  

We now show that the functor induces an equivalence on spaces of morphisms.
Consider the diagram of spaces
\[
\xymatrix{
\bigl(\sD_U[\cI_U^{-1}]\bigr)^{\sim}  \ar[r]  \ar[d]_-{(U\hookrightarrow X)}
&
\cD_U^{\sim}  \ar[d]^-{(U\hookrightarrow X)}
\\
\bigl(\sD_X[\cI_X^{-1}]\bigr)^{(1)} \ar[r]  \ar[d]_-{{\sf ev}_1}
&
\cD_X^{(1)}  \ar[d]^-{{\sf ev}_1}
\\
\bigl(\sD_X[\cI_X^{-1}]\bigr)^{\sim}  \ar[r]
&
\cD_X^{\sim}
}
\]
where a superscript ${}^{(1)}$ indicates a space of morphisms, and the upper vertical arrows are given as $(V\hookrightarrow U)\mapsto \bigl((V\hookrightarrow X)\hookrightarrow (U\hookrightarrow X)\bigr)$.
Our goal is to show that the middle horizontal arrow is an equivalence.
We will accomplish this by showing that the diagram is a map of fiber sequences, for we have already shown that the top and bottom horizontal maps are equivalences.

The right vertical sequence is a fiber sequence because such evaluation maps are coCartesian fibrations, in general.  
Then, by inspection, the fiber over $(U\hookrightarrow X)$ is the underlying $\infty$-groupoid of the slice $(\cD_X)_{/(U\hookrightarrow X)}$.
This slice is canonically identified as $\cD_U$.  

Let us show that the left vertical sequence is a fiber sequence.  
The space of morphisms of $\bigl(\sD_X[\cI_X^{-1}]\bigr)^{(1)}$ is the classifying space of the subcategory of the functor category $\Fun^{\cI_X}\bigl([1],\sD_X\bigr) \subset \Fun\bigl([1],\sD_X\bigr)$ consisting of the same objects but only those natural transformations by $\cI$.  
We claim the fiber over $(U\hookrightarrow X)$ of the evaluation map is canonically identified as in the sequence
\[
\bigl(\sD_U[\cI_U^{-1}]\bigr)^{\sim} \xra{~(U\hookrightarrow X)~}\bigl(\sD_X[\cI_X^{-1}]\bigr)^{(1)} \xra{~{\sf ev}_1~} \bigl(\sD_X[\cI_X^{-1}]\bigr)^{\sim}~.
\]
This claim is justified through Quillen's Theorem B, for the named fiber is the classifying space of the slice category $(\cI_X)_{/(U\hookrightarrow X)}$ which is canonically isomorphic to $\cI_U$.  
To apply Quillen's Theorem B, we must show that each morphism $(U\hookrightarrow X) \hookrightarrow (V\hookrightarrow X)$ in $\cI$ induces an equivalence of spaces $\sB\bigl((\cI_X)_{/(U\hookrightarrow X)}\bigr) \simeq \sB\bigl((\cI_X)_{/(V\hookrightarrow X)}\bigr)$.  
This map of spaces is canonically identified as $\sB \cI_U \to \sB \cI_V$.
Through the previous analysis of this proof, this map is further identified as the map of spaces $\underset{[B]} \coprod U_{[B]} \to \underset{[B]} \coprod V_{[B]}$ induced from the inclusion $U\hookrightarrow V$.
Because $U$ and $V$ are abstractly isomorphic, then results of~\cite{aft1} give that each such inclusion is isotopic through stratified open embeddings to an isomorphism.  
We conclude that Quillen's Theorem B applies. See also Theorem~5.3 of~\cite{barwick}.
\end{proof}

\begin{lemma}\label{refinement-localize}

Let $r\colon \w{X}\to X$ be a refinement between stratified spaces.
Then there is a functor 
\[
\disk(\bsc)_{/\w{X}} \longrightarrow \disk(\bsc)_{/X}
\]
which is a localization.

\end{lemma}

\begin{proof}

There is an obvious functor $\snglrd(\bsc)_{/\w{X}}\to \snglrd(\bsc)_{/X}$ given by assigning to $(\w{i}\colon \w{Z}\hookrightarrow\w{X})$ the object $(i\colon Z\hookrightarrow X)$ where $Z=\w{i}(\w{Z})\subset X$ is the image with the inherited stratification. 
Observe that this functor sends finite disjoint unions of basics conically smoothly openly embedding into $\w{X}$, to finite disjoint unions conically smoothly openly embedding into $X$.  
Also observe that this functor sends stratified isotopy equivalences to stratified isotopy equivalences.  
After these observations, the desired functor happens through Proposition~\ref{EEd-vs-EE}.

Notice that the map of underlying $\infty$-groupoids $\bigl(\disk(\bsc)_{/\w{X}}\bigr)^\sim \to \bigl(\disk(\bsc)_{/X}\bigr)^\sim$ respects the presentation of these spaces in Lemma~\ref{EE-equivs} as a coproduct of a product.  
Through the same logic as in the proof of Proposition~\ref{EEd-vs-EE}, it is enough to show that the functor
\[
\bsc_{/\conf_i(\w{M})} \longrightarrow \bsc_{/\conf_i(M)}
\]
is a localization for each finite cardinality $i$ and each refinement $\w{M}\xra{r} M$ of a smooth manifold.  
Because $\w{M}\to M$ is a refinement, then so is the map of stratified spaces $\conf_i(\w{M})\to \conf_i(M)$, by inspection.  
The result is an instance of Proposition~1.2.14 of~\cite{aft1}.  

\end{proof}

\subsection{Pushforward}\label{sec.pushforward}
For this subsection, we fix a symmetric monoidal $\infty$-category $\cV$ that is $\ot$-sifted cocomplete.  

\begin{lemma}\label{f-inv-factr}
Let $X$ be a $\cB$-manifold and $Y$ a $\cB'$-manifold, and let $f\colon X\to Y$ be a constructible bundle of the underlying stratified spaces. Taking inverse images defines a functor of $\oo$-operads
\[f^{-1}: \disk(\cB')_{/Y}\longrightarrow \mfld(\cB)_{/X}\]
which on objects sends $U\hookrightarrow Y$ to $f^{-1}U\hookrightarrow X$. As a consequence, there is a natural pushforward functor \[f_\ast: \Alg_{X}(\cV)\longrightarrow \Alg_{Y}(\cV)\] that sends an algebra $A$ on $X$ to the algebra on $Y$ taking values 
\[
f_\ast A(U\hookrightarrow Y) ~{}~\simeq~{}~ \int_{f^{-1}U}A ~.
\] 

\end{lemma}

\begin{proof}

Note that taking point-wise inverse images defines a functor of discrete categories $\diskd(\bsc)_{/Y} \ra \snglrd_{/X}$ which, additionally, preserves the multi-category structure of Construction \ref{B-operadics}. We first show that this functor can be naturally extended to the topological case, i.e., that there is a preferred filler in the diagram of $\infty$-operads
\[
\xymatrix{
\diskd(\bsc)_{/Y}  \ar[rr]^-{f^{-1}}  \ar[d]
&&
\snglrd_{/X}  \ar[d]
\\
\disk(\bsc)_{/Y}  \ar@{-->}[rr]^-{f^{-1}}
&&
\snglr_{/X}~.
}
\]
Because $f$ is a constructible bundle, the collections $\cI$ and $\cJ$ of isotopy equivalences in $\diskd(\bsc)_{/Y}$ and $\snglrd_{/X}$, respectively, are mapped to one another by $f^{-1}$. Further, they map to equivalences in $\disk(\bsc)_{/Y}$ and $\snglr_{/X}$, respectively. By the universal property of localization, we can thus factor the previous diagram as
\[
\xymatrix{
\diskd(\bsc)_{/Y}  \ar[rr]^-{f^{-1}}  \ar[d]
&&
\snglrd_{/X}  \ar[d]
\\
\diskd(\bsc)_{/Y}[\cI^{-1}]  \ar[rr]^-{f^{-1}}  \ar[d]^{\sim}
&&
\snglrd_{/X}[\cJ^{-1}]  \ar[d]
\\
\disk(\bsc)_{/Y}  \ar@{-->}[rr]^-{f^{-1}}
&&
\snglr_{/X}~.
}
\]
Proposition~\ref{EEd-vs-EE} states that the bottom left downward arrow is an equivalence, as indicated, thereby determining the filler.  

Recall from Observation \ref{B-independent} that $\cB$-manifold structure on $X$ defines an equivalence $\snglr_{/X} \simeq \mfld(\cB)_{/X}$, directly from the definition of structures by right fibrations; the same is true for the $\cB'$-manifold structures on basics in $Y$. Under the hypotheses of the Lemma, we consequently have a map of $\oo$-operads $\disk(\cB')_{/Y}\ra \mfld(\cB)_{/X}$ defined by $f^{-1}$. The desired functor is the composition
\[\xymatrix{
f_\ast \colon \Alg_{X}(\cV) \ar[rr]^-{\int_-}&&\Alg_{\mfld(\cB)_{/X}}(\cV)\ar[rr]^-{(f^{-1})^\ast}&&\Alg_{Y}(\cV)~.\\
}\]

\end{proof}

We have the following result about this pushed-forward algebra $f_\ast A$.

\begin{theorem}[Pushforward]\label{:(}
Let $\cV$ be a symmetric monoidal $\oo$-category which is $\ot$-sifted cocomplete; let $X$ be a $\cB$-manifold with $f:X\ra Y$ a weakly constructible bundle over a stratified space $Y$. 
There is a commutative diagram:
\[\xymatrix{
\Alg_{X}(\cV)\ar[rr]^{\int_X}\ar[dr]_{f_\ast }&&\cV\\
&\Alg_{Y}\ar[ur]_{\int_Y}(\cV)\\}
\] In particular, for any $\disk(\cB)$-algebra $A$, there is a canonical equivalence in $\cV$:
\[ 
\int_Y f_\ast A~{}~ \simeq~{}~ \int_XA~.
\]
\end{theorem}

The proof of this theorem will involve an auxiliary $\oo$-category $\cX_f$ built from the constructible bundle $f$.

\begin{definition}\label{def.disk-f}
Let $X$ be a $\cB$-manifold, and let $Y$ be a $\cB'$-manifold. For $f:X\ra Y$ a constructible bundle, the $\oo$-category $\cX_f$ is the limit of the following diagram
\[
\xymatrix{
\disk(\cB)_{/X}\ar@{_{(}->}[dr]
&
&
\Fun\bigl([1],\mfld(\cB)_{/X}\bigr)\ar[rd]_{{\sf ev}_1}   \ar[ld]^{{\sf ev}_0}    
&
&
\disk(\cB')_{/Y}   \ar[dl]^{f^{-1}}
\\
&
\mfld(\cB)_{/X}
&
&
\mfld(\cB)_{/X}
}
\] 
where $\Fun\bigl([1],\mfld(\cB)_{/X}\bigr)$ is the $\oo$-category of 1-morphisms in $\mfld(\cB)_{/X}$. 
\end{definition}

Informally, $\cX_f$ consists of compatible triples $(U, V, U\hookrightarrow f^{-1}V)$, where $V$ is a finite disjoint union of basics in $Y$; $U$ is a finite disjoint union of basics in $X$; the embedding $U\hookrightarrow f^{-1}V$ is compatible with the embeddings $f^{-1}V\hookrightarrow X$ and $U\hookrightarrow X$. 
We are able to utilize this $\oo$-category $\cX_f$ because of the following key finality property.

\begin{lemma}\label{final} For any constructible bundle $X\xra{f} Y$ between stratified spaces, the functor ${\sf ev}_0:\cX_f \ra \disk(\cB)_{/X}$ is final.
\end{lemma}
\begin{proof} 
The functor ${\sf ev}_0$ is a Cartesian fibration of $\oo$-categories. Thus, to check finality, by Lemma 4.1.3.2 of \cite{HTT}, it suffices to show that, for each $U\in \disk(\cB)_{/X}$, the fiber ${\sf ev}_0^{-1}V$ has contractible classifying space.
So let $V\in \disk(\cB)_{X}$.
Unwinding the definition of $\cX_f$, we must show contractibility of the classifying space of the iterated slice $\infty$-category $\bigl(\disk(\cB)_{/Y}\bigr)^{U/}$. 

The projection from this slice $\bigl(\disk(\cB)_{/Y}\bigr)^{U/} \to \disk(\cB)_{/Y}$ is a left fibration, and so it is classified by a functor, justifiably written as 
\[
\Map_{\mfld(\cB)_{/X}}(U,f^{-1}-)\colon \disk(\cB)_{/Y} \to \spaces~,
\]
whose colimit is identified as the relevant classifying space 
\[
\sB\bigl(\disk(\cB)_{/Y}\bigr)^{U/}~{}~\simeq~{}~\underset{(V\hookrightarrow Y)\in \disk(\cB)_{/Y}}\colim ~\Map_{\mfld(\cB)_{/X}}(U,f^{-1}V)~.  
\]
So we seek to show the righthand colimit is contractible.

Formal is that the sequence of maps
\[
\Map_{\mfld(\cB)_{/X}}(U,f^{-1}V)~\longrightarrow~\Map_{\mfld(\cB)}(U,f^{-1}V)~\longrightarrow~\Map_{\mfld(\cB)}(U,X)
\]
is a fiber sequence (here the fiber is taken over any implicit morphism $U\hookrightarrow X$, thereby giving meaning to the lefthand space).  
So we seek to show the map from the colimit
\[
\underset{(V\hookrightarrow Y)\in \disk(\cB)_{/Y}}\colim \Map_{\mfld(\cB)}(U,f^{-1}V)~{}~\longrightarrow~{}~ \Map_{\mfld(\cB)}(U,X)
\]
is an equivalence of spaces. 
We recognize this map of spaces as the map of fibers over $U\in \disk(\cB)$ of the map of right fibrations over $\disk(\cB)$:
\[
\underset{(V\hookrightarrow Y)\in \disk(\cB)_{/Y}}\colim \disk(\cB)_{/f^{-1}V} \longrightarrow \disk(\cB)_{/X}~.
\]
Being right fibrations, it is enough to show that this functor is an equivalence on underlying $\infty$-groupoids.

Through Lemma~\ref{EE-equivs} this is the problem of showing, for each finite set $J$ and each basic $U$, that the map of spaces
\[
\underset{(V\hookrightarrow Y)\in \disk(\cB)_{/Y}}\colim \conf_J\bigl((f^{-1}V)_{[U]}\bigr)_{\Sigma_J} ~\longrightarrow~ \conf_J(X_{[U]})_{\Sigma_J}
\]
is an equivalence. 
So we can assume that the underlying stratified space of $X=M$ is an ordinary smooth manifold.  

Lemma~\ref{EEd-vs-EE} grants that the forgetful map
\[
\underset{(V\hookrightarrow Y)\in \diskd(\bsc)_{/Y}}\colim \conf_J\bigl(f^{-1}V\bigr)_{\Sigma_J} 
~\xra{~\simeq~}~
\underset{(V\hookrightarrow Y)\in \disk(\cB)_{/Y}}\colim \conf_J\bigl(f^{-1}V\bigr)_{\Sigma_J} 
\]
is an equivalence of spaces.  
Now notice that, for each $(V\hookrightarrow Y)\in \diskd(\bsc)_{/Y}$, the map $\conf_J(f^{-1}V) \to \conf_J(M)$ an open embedding.  Also, for each point $c\colon J\hookrightarrow M$ the image $f\bigl(c(J)\bigr)\subset Y$ has cardinality at most $J$.  So there is an object $(V\hookrightarrow Y)$ of $\diskd(\bsc)_{/Y}$ whose image contains the subset $f\bigl(c(J)\bigr)$.  We see then that the collection of open embeddings
\[
\Bigl\{ \conf_J(f^{-1}V)_{\Sigma_J} \hookrightarrow \conf_J(M)_{\Sigma_I}\mid (V\hookrightarrow Y)\in \diskd(\bsc)_{/Y} \Bigr\}
\]
form an open cover.

Proved in~\cite{aft1} is that open embeddings of basics into $Y$ form a basis for the topology of $Y$.
It follows that the collection of (at most) $|J|$-tuples of disjoint basics in $Y$ form an open cover of $Y$ in such a way that any finite intersection of such is again covered by such.  
It follows that the collection of open embeddings above is an open cover for which any finite intersection of its terms is again covered by terms of the collection.  
In particular, this collection of open embeddings forms a hypercover of $\conf_J(M)_{\Sigma_J}$.
Corollary 1.6 of \cite{Dugger--Isaksen} gives that the map
\[
\underset{(V\hookrightarrow Y)\in \diskd(\bsc)_{/Y}}\colim \conf_J\bigl(f^{-1}V\bigr)_{\Sigma_J} 
~\xra{~\simeq~}~
\conf_J(M)_{\Sigma_J}
\]
is an equivalence of spaces.
This completes the proof. 

\end{proof}

This has the following important corollary, which completes the proof of Theorem \ref{fact-explicit}.

\begin{cor}\label{quittersifted}
For $X$ a $\cB$-manifold, the $\oo$-category $\disk(\cB)_{/X}$ is sifted.
\end{cor}
\begin{proof}

The $\oo$-category $\disk(\cB)_{/X}$ is evidently nonempty, as it contains the object $(\emptyset \hookrightarrow X)$. 
We must then prove that the diagonal functor $\disk(\cB)_{/X} \to \disk(\cB)_{/X}\times \disk(\cB)_{/X}$ is final.
This diagonal functor fits into a diagram among $\infty$-categories
\[
\xymatrix{
\disk_{\nabla} \ar[rr]^-{{\sf ev}_0}  \ar[d]_-{{\sf ev}_1}
&&
\disk(\cB)_{/X\sqcup X}  
\\
\disk(\cB)_{/X}  \ar[rr]^-{\sf diag}
&&
\disk(\cB)_{/X}\times \disk(\cB)_{/X}    \ar[u]^-{\simeq}_{\sqcup}
}
\]
that we now explain.
The upper left $\infty$-category is that of Definition~\ref{def.disk-f} applied to the fold map $\nabla \colon M \sqcup M \to M$; as so, it is equipped with the indicated projection functors.
The right vertical arrow is induced by the symmetric monoidal structure on $\mfld(\cB)$, which is disjoint union.  
This right vertical arrow is an equivalence; an inverse is given by declaring its projection to each factor to be given by intersecting with the corresponding cofactor of the disjoint union.
Therefore, to prove that the diagonal functor is final it is sufficient to prove that both of the projection functors ${\sf ev}_1$ and ${\sf ev}_0$ are final.  
The finality of ${\sf ev}_0$ is Lemma~\ref{final}.

We explain that ${\sf ev}_1$ is final.
Note that the functor $\nabla^{-1} \colon \disk(\cB)_{/X} \to \mfld(\cB)_{/X\sqcup X}$ factors through the full $\infty$-subcategory $\disk(\cB)_{/X\sqcup X}$.  
As so, there is a canonical identification between $\infty$-categories
\[
\disk_\nabla~\simeq~ \disk(\cB)_{/X}\underset{\disk(\cB)_{/X\sqcup X}}\times  {\sf Ar}(\disk(\cB)_{/X\sqcup X})
\]
over $\disk(\cB)_{/X}$.
Through this identification, the composite functor 
\[
\disk(\cB)_{/X} \xra{~\nabla~} \disk(\cB)_{/X\sqcup X} \xra{~\sf const~} {\sf Ar}(\disk(\cB)_{/X\sqcup X})
\]
determines a right adjoint to the functor ${\sf ev}_1$.  
Finality of ${\sf ev}_1$ follows.

\end{proof}

Using Lemma \ref{final}, we can now prove our push-forward formula for factorization homology.

\begin{proof}[Proof of Theorem \ref{:(}] 
After Lemma~\ref{refinement-localize}, we can assume that the \emph{weakly} constructible bundle $f$ is actually constructible.  

The functor \[f_\ast A:\xymatrix{\disk(\bsc)_{/Y}\ar[r]^{f^{-1}}& \disk(\cB)_{/X}\ar[r]^{\int A}& \cV}\] is the left Kan extension of $A:\cX_f\ra \cV$ along the functor ${\sf ev}_1:\cX_f\ra \disk(\bsc)_{/Y}$.
 Therefore there is a natural equivalence \[\colim_{\cX_f} A \simeq \colim_{U\in\disk(\bsc)_{/Y}}\int_{f^{-1}U}A \simeq \int_N f_\ast A~.\] By Lemma \ref{final}, the functor ${\sf ev}_0:\cX_f \ra \disk(\cB)_{/X}$ is final, which implies the equivalence $\colim_{\cX_f} A \simeq \int_MA$, and the result follows.
\end{proof}

In the case of the projection map $X\times Y\ra Y$ off of a product, this result has the following consequence.

\begin{cor}[Fubini]\label{fubini}
Let $\cB_-$ and $\cB_+$ and $\cB$ be $\oo$-categories of basics with a functor $\cB_-\times \cB_+ \to \cB$ over the product functor $\bsc\times\bsc\xra{\times} \bsc$.  
Let $A$ be a $\disk(\cB)$-algebra in $\cV$.  
Let $X$ be a $\cB_-$-manifold, and $Y$ be a $\cB_+$-manifold; so $X\times Y$ is canonically equipped as a $\cB$-manifold.
There is a canonical equivalence in $\cV$:
\[
\int_{X\times Y} A~{}~\simeq~{}~\int_Y \int_X A~.
\]

\end{cor}

\subsection{Algebras over a closed interval}

\begin{definition}[$\Ass^{\sf RL}$]\label{Assoc}
Let $\Ass^{\sf RL}$ denote an $\infty$-operad corepresenting triples $(A;Q,P)$ consisting of an associative algebra together with a unital left and a unital right module.  
Specifically, it is a unital multi-category whose space of colors is the three-element set $\{M, R, L\}$, and with spaces of multi-morphisms given as follows.  Let $I\xra{\sigma}\{M,R,L\}$ be a map from a finite set.
\begin{itemize}
\item $\Ass^{\sf RL}(\sigma, M)$ is the set of linear orders on $I$ for which no element is related to an element of $\sigma^{-1}(\{R,L\})$.
In other words, should $\sigma^{-1}(\{R,L\})$ be empty, then there is one multi-morphism from $\sigma$ to $M$ for each linear order on $\sigma^{-1}(M)$; should $\sigma^{-1}(\{R,L\})$ not be empty, then there are no multi-morphisms from $\sigma$ to $M$.

\item $\Ass^{\sf RL}(\sigma, L)$ is the set of linear orders on $I$ for which each element of $\sigma^{-1}(L)$ is a minimum, and no element is related to an element of $\sigma^{-1}(R)$.  
In other words, should $\sigma^{-1}(\{R\})$ be empty and $\sigma^{-1}(\{L\})$ have cardinality at most $1$, then there is one multi-morphism from $\sigma$ to $M$ for each linear order on $\sigma^{-1}(M)$;  should $\sigma^{-1}(\{R\})$ not be empty or $\sigma^{-1}(\{L\})$ have cardinality greater than $1$, then there are no multi-morphisms from $\sigma$ to $M$.

\item $\Ass^{\sf RL}(\sigma, R)$ is the set of linear orders on $I$ for which each element of $\sigma^{-1}(R)$ is a maximum, and no element is related to an element of $\sigma^{-1}(L)$.
In other words, should $\sigma^{-1}(\{L\})$ be empty and $\sigma^{-1}(\{R\})$ have cardinality at most $1$, then there is one multi-morphism from $\sigma$ to $M$ for each linear order on $\sigma^{-1}(M)$;  should $\sigma^{-1}(\{L\})$ not be empty or $\sigma^{-1}(\{R\})$ have cardinality greater than $1$, then there are no multi-morphisms from $\sigma$ to $M$.

\end{itemize}
Composition of multi-morphisms is given by concatenating linear orders.  

\end{definition}

Consider the oriented $1$-manifold with boundary $[-1,1]$, which is the closed interval, that we regard as a structured stratified space in the sense of~\cite{aft1}.  
Taking connected components depicts a map of symmetric monoidal $\infty$-categories
\[
[-]\colon \diskd^{\sf or}_1 \longrightarrow {\sf Env}\bigl(\Ass^{\sf RL}\bigr)
\]
to the symmetric monoidal envelope,
where $[\RR] = M$, $[\RR_{\geq 0}] = L$, and $[\RR_{\leq 0}] = R$.  

\begin{observation}\label{disk-assoc}
The symmetric monoidal functor $\diskd^{ \partial, \sf or}_1 \xra{[-]}{\sf Env}\bigl(\Ass^{\sf RL}\bigr)$ factors as an equivalence of symmetric monoidal $\infty$-categories:
\begin{equation}\label{disk-comb}
\disk^{\partial, \sf or}_1  \xra{~\simeq~}  {\sf Env}\bigl(\Ass^{\sf RL}\bigr)
\end{equation}
-- this follows by inspecting the morphism spaces of the $\infty$-category $\disk^{\partial, \sf or}_1$, which are discrete.
\end{observation}

\begin{observation}\label{interval-assoc}
Consider the ordinary category $\sO^{\sf RL}$ for which an object is a linearly ordered finite set $(I,\leq)$ together with a pair of disjoint subsets $R\subset I \supset L$ for which each element of $R$ is a minimum and each element of $L$ is a maximum, and for which a morphism $(I,\leq, R,L)\to (I',\leq',R',L')$ is an order preserving map $I\xra{f} I'$ for which $f(R\sqcup L)\subset R'\sqcup L'$.  
(Note that the cardinality of each of $L$ and of $R$ is at most one.)
Concatenating linear orders makes $\sO^{\sf RL}$ into a multi-category, and it is equipped with a canonical maps of operads to ${\sf Env}\bigl(\Ass^{\sf RL}\bigr)$.  
By inspection, the equivalence~(\ref{disk-comb}) of Observation~\ref{disk-assoc} lifts to an equivalence of $\infty$-operads, 
\begin{equation}\label{interval-comb}
[-]\colon \disk^{\partial, {\sf or}}_{1/[-1,1]} \xra{\simeq} \sO^{\sf RL}~.
\end{equation} 

\end{observation}

After these observations, there is this immediate consequence.
\begin{cor}\label{assoc-interval}
There is a canonical equivalence of $\infty$-categories:
\[
\int_-\colon \Alg_{\Ass^{\sf RL}}(\cV) \xra{~\simeq~} \Alg_{[-1,1]}(\cV)~.
\]
\end{cor}

\begin{proof}
So the remaining point to check is that restriction along the map of $\infty$-operads $\disk^{\partial, \sf or}_{1/[-1,1]}\to \disk^{\partial, \sf or}_1$ implements an equivalence of symmetric monoidal $\infty$-categories ${\sf Env}\bigl(\disk^{\partial, \sf or}_{1/[-1,1]}\bigr) \xra{\simeq} \disk^{\partial, \sf or}_1$ from the symmetric monoidal envelope. 
Well, for each object $U\in \sD^{\partial, \sf or}_1$ it is standard that the space of morphisms $\mfld^{\partial, \sf or}_1\bigl(U,[-1,1]\bigr) \simeq \ast$ is contractible; and it follows that this functor gives an equivalence on spaces of objects of underlying $\infty$-categories.  
That this functor gives an equivalence on the space of morphisms follows because the functor on the active $\infty$-subcategories $\disk^{\partial, \sf or}_{1/[-1,1]} \to \disk^{\partial, \sf or}_1$ is a right fibration.  

\end{proof}

\begin{prop}\label{tensor-prod}
Let $(A;P,Q)$ be an $\Ass^{\sf RL}$-algebra in $\cV$; which is to say an associative algebra $A$ together with a unital left and a unital right $A$-module.
Applying Observation~\ref{disk-assoc}, regard $(A;P,Q)$ as a $\disk^{\partial, \sf or}_1$-algebra in $\cV$. 
There is a canonical equivalence in $\cV$:
\[
Q\underset{A}\ot P\xra{~\simeq~} \int_{[-1,1]} (A;P,Q)~.  
\]

\end{prop}

\begin{proof}
There is the standard fully faithful functor $\bDelta^{\op}\subset \sO^{\sf RL}$ whose essential image consists of those objects $(I,\leq , R, L)$ for which $R\neq \emptyset \neq L$.  
Adjoining minima and maxima gives a left adjoint to this functor, and so it is final.   
Through Observation~\ref{interval-assoc}, there is a final functor $\bDelta^{\op} \to \disk^{\partial, {\sf or}}_{1/[-1,1]}$; and the resulting simplicial object 
\[
{\sf Bar}_\bullet\bigl(Q,A,P\bigr)\colon \bDelta^{\op} \to \disk^{\partial, {\sf or}}_{1/[-1,1]} \to \disk^{\partial, \sf or}_1 \xra{[-]}{\sf Env}\bigl(\Ass^{\sf RL}\bigr) \xra{(A;P,Q)} \cV
\]
is identified as the two-sided bar construction, as indicated. 
We conclude the equivalence in $\cV$:
\[
Q\underset{A}\ot P \simeq \colim\bigl(\bDelta^{\op} \xra{{\sf Bar}_\bullet\bigl(Q,A,P\bigr)} \cV\bigr)\xra{\simeq} \int_{[-1,1]}(A;P,Q)~.
\]

\end{proof}

Lemma~\ref{f-inv-factr} and Proposition~\ref{tensor-prod} assemble as the next result.  
\begin{cor}\label{excision-arrow}
Let $F\colon \mfld(\cB) \longrightarrow \cV$ be a symmetric monoidal functor.  
Let $X\cong X_- \underset{\RR\times X_0}\bigcup X_+$ be a collar-gluing among $\cB$-manifolds.  
Then there is a canonical arrow in $\cV$:
\begin{equation}\label{exc-compare}
F(X_-)\underset{F(X_0)} \bigotimes F(X_+) \longrightarrow F(X)~.
\end{equation}

\end{cor}

\begin{proof}
The collar-gluing is prescribed by a constructible bundle $\w{X}\xra{f}[-1,1]$ from a refinement of $X$. 
From Lemma~\ref{f-inv-factr}, there is the composite map of $\infty$-operads 
\[
f^{-1}\colon \disk^{\partial, {\sf or}}_{1/[-1,1]} \xra{f^{-1}} \mfld(\cB)_{/\w{X}} \to \mfld(\cB) \xra{F} \cV~.
\] 
The universal property of factorization homology as a left adjoint gives the canonical arrow:
\[
F(X_-) \underset{F(X_0)}\bigotimes F(X_+)~{}~ \underset{\rm Prop~\ref{tensor-prod}}\simeq ~{}~\int_{[-1,1]} F_f~{}~ \longrightarrow ~{}~ F\bigl(f^{-1}([-1,1])\bigr) =  F(X)~;
\]
namely, this arrow is the counit of the adjunction $\int \colon \Alg_{\disk(\cB)}(\cV) \rightleftarrows \Fun^\ot\bigl(\mfld(\cB),\cV\bigr)$, evaluated on $F$.

\end{proof}

\subsection{Homology theories}

One an formulate an $\oo$-categorical analogue of the Eilenberg--Steenrod axioms for a functor from spaces or manifolds to the $\oo$-category of chain complexes or spectra: the functor should take certain gluing diagrams to pushout squares, and it should preserve sequential colimits. From these conditions, one recovers usual generalized homology theories. These axioms admit a generalization when one replaces chain complexes or spectra with a symmetric monoidal $\oo$-category $\cV$. 
The formulation is complicated slightly by the fact that the monoidal structure on $\cV$ is not required to be coCartesian. 
Nevertheless, we make the following definitions, which generalize those of \cite{Fact}.

A main result of~\cite{aft1} gives a precise articulation of the heuristic statement that the $\infty$-category $\mfld(\cB)$ is generated by $\cB$ through the formation of collar-gluings and sequential unions. 
After Example~\ref{union-collar-gluing}, this result of~\cite{aft1} has an immediate symmetric monoidal reformulation.
\begin{cor}[After~\cite{aft1}]\label{generation}
Let $\disk(\cB) \subset \sS \subset \mfld(\cB)$ be a full sub-symmetric monoidal $\infty$-category that is closed under the following two formations:
\begin{itemize}
\item Let $X\cong X_-\underset{\RR\times X_0}\cup X_+$ be a collar-gluing among $\cB$-manifolds.
If each of $X_+$, $X_-$, and $\RR\times X_0$ is an object of $\sS$, then $X$ too is an object of $\sS$.

\item  Consider a sequence $X_0\subset X_1\subset \dots \subset X$ of open subspaces whose union $\underset{i\geq 0}\bigcup X_i = X$ is entire.
If each $X_i$ is an object of $\sS$, then $X$ too is an object of $\sS$.  
\end{itemize}
Then the inclusion $\sS \subset \mfld(\cB)$ is an equality.

\end{cor}

\begin{definition}\label{homology}
The $\infty$-category of \emph{homology theories (over $X$)} is the full $\infty$-subcategory
\[
\bcH\bigl(\mfld(\cB)_{/X},\cV\bigr)~\subset~\Fun^\ot\bigl(\mfld(\cB)_{/X},\cV\bigr)
\]
consisting of those $H$ that satisfy the following two properties:
\begin{itemize}
\item {\bf $\ot$-Excision:} 
Let $W \cong W_- \underset{\RR\times W_0} \bigcup W_+$ denote a collar-gluing among $\cB$-manifolds over $X$.
Then the canonical morphism~(\ref{exc-compare})
\begin{equation}\label{ot-excision}
H(W_-)\underset{H(W_0)}\bigotimes H(W_+)\xra{~\simeq~} H(W)
\end{equation}
is an equivalence in $\cV$.
\item {\bf Continuous:}
Let $W_0\subset W_1\subset \dots \subset X$ be a sequence of open sub-stratified spaces of $X$ with union denoted as $\underset{i\geq 0} \bigcup W_i = : W$.
Then then the canonical morphism in $\cV$
\begin{equation}\label{continuous}
\colim\Bigl(H(W_0) \to H(W_1)\to \dots \Bigr)\xra{~\simeq~} H(W)
\end{equation}
is an equivalence.

\end{itemize}
Absolutely, the $\infty$-category of \emph{homology theories (for $\cB$-manifolds}) is the full $\infty$-subcategory
\[
\bcH\bigl(\mfld(\cB),\cV\bigr)~\subset~\Fun^\ot\bigl(\mfld(\cB),\cV\bigr)
\]
consisting of those $H$ for which, for each $\cB$-manifold $X$, the restriction $H_{|\mfld(\cB)_{/X}}$ is a homology theory for $X$.  

\end{definition}

In the coming sections we will see a variety of examples of homology theories, for various categories of basics $\cB$.

In the next result we make use of Day convolution; we give a brief synopsis, taken from~\S4.8.1 of~\cite{HA}.  
For $\cD$ a symmetric monoidal $\infty$-category, Day convolution endows the $\infty$-category $\Psh(\cD)$ of presheaves (on the underlying $\infty$-category of $\cD$) with a symmetric monoidal structure.
Furthermore, with this symmetric monoidal structure, the Yoneda functor $\cD \to \Psh(\cD)$ is symmetric monoidal, and in fact presents the free symmetric monoidal cocompletion of $\cD$.  
In particular, each symmetric monoidal functor $\cD\to \cM$ to another symmetric monoidal $\infty$-category determines a restricted symmetric monoidal Yoneda functor $\cM \to \Psh(\cD)$.  
\begin{cor}[Universal homology theory]\label{univ}
The symmetric monoidal restricted Yoneda functor $\disk(\cB)_{/-}\colon \mfld(\cB) \to \Psh\bigl(\disk(\cB)\bigr)$ is a homology theory.

\end{cor}

\begin{proof}
From its universal property, in the symmetric monoidal $\infty$-category $\Psh\bigl(\disk(\cB)\bigr)$ is the tautological $\disk(\cB)$-algebra which is the symmetric monoidal Yoneda functor $\disk(\cB) \to \Psh\bigl(\disk(\cB)\bigr)$.  
Note that the Day convolution symmetric monoidal structure distributes over all colimits, and therefore $\Psh\bigl(\disk(\cB)\bigr)$ is $\ot$-sifted cocomplete.
We can therefore apply Theorem~\ref{:(}.
Namely, let $f\colon X\to [-1,1]$ be a collar-gluing, written as $X\cong X_-\underset{\RR\times X_0}\bigcup X_+$.  
We must show the canonical arrow of~(\ref{exc-compare}),
\[
\disk(\cB)_{/X_-}\underset{\disk(\cB)_{/\RR\times X_0}}\bigotimes \disk(\cB)_{/X_+} \xra{\simeq}\disk(\cB)_{/X}~,
\]
is an equivalence between right fibrations over $\disk(\cB)$. 
Theorem~\ref{:(}, applied to the weakly constructible map $f\colon X\to[-1,1]$ gives just this.

\end{proof}

\begin{remark}\label{confs}
Inside of Corollary~\ref{univ} are a number of interesting geometric statements; we will indicate one such now.
Let $M$ be an ordinary smooth $n$-manifold.  
Let $I$ be a finite set.  
There is the restricted right fibration $\bigl(\disk(\cB)_{/M}\bigr)_{|\sB \Sigma_I} \to \sB\Sigma_I$, which is just the data of a map of spaces.
The fiber of this map over $I$ is equivalent to the space $\conf_I(M)$ of injections $I\hookrightarrow M$.
The fact that $\disk_{n/-}$ satisfies $\ot$-excision yields the following relationship among such configuration spaces.

Let $M = M_- \underset{\RR\times M_0}\bigcup M_+$ be a collar-gluing.
Then there is a weak homotopy equivalence of spaces
\[
\conf_\bullet(M_-)\underset{\conf_\bullet(\RR\times M_0)}\bigotimes \conf_\bullet(M_+)~\simeq~\conf_I(M)
\]
where the lefthand side is a two-sided bar construction; specifically, it is the geometric realization of a simplicial space whose space of $p$-simplices is weakly equivalent to 
\[
\underset{J_- \sqcup J_1\sqcup \dots \sqcup J_p\sqcup J_+ \cong I}\coprod \conf_{J_-}(M_-)\times\Bigl( \underset{1\leq k\leq p} \prod \conf_{J_k}(\RR\times M_0)\Bigr) \times \conf_{J_+}(M_+)
\]
and whose face maps are given by ordered embeddings by the $\RR$-coordinate.  
For $I=\ast$ a singleton, this simply recovers the underlying homotopy type of $M$ as the pushout $M_- \underset{M_0} \coprod M_+$.  

\end{remark}

The following result justifies some of our terminology.

\begin{cor}[Factorization homology satisfies $\ot$-excision and is continuous]\label{fact-excisive}
Let $\cV$ be a symmetric monoidal $\infty$-category that is $\ot$-sifted cocomplete.
Then the each of the absolute and, for $X$ a $\cB$-manifold, the relative factorization homology functors factor
\[
\int_-\colon \Alg_{\disk(\cB)}(\cV)~\longrightarrow~\bcH\bigl(\mfld(\cB),\cV\bigr)
\qquad\text{ and }\qquad
\int_-\colon \Alg_{X}(\cV)~\longrightarrow~
\bcH\bigl(\mfld(\cB)_{/X},\cV\bigr)
\]
through homology theories.

\end{cor}

\begin{proof}
That $\int_-$ satisfies the $\ot$-excision axiom follows by applying Theorem~\ref{:(} to a collar-gluing $X\to[-1,1]$, which, by definition, is a weakly constructible bundle.

Now, consider a sequential union $X_0\subset X_1\subset \dots \subset X$.
This open cover of $X$ has the following property.
\begin{itemize}
\item[~]
Let $S\subset X$ be a finite subset.  Then there is an $i$ for which $S\subset X_i$.  It follows that, for each finite set $J$, the collection of open subsets
\[
\Bigl\{ \conf_J(X_i)_{\Sigma_J} \subset \conf_J(X)_{\Sigma_J}\mid i\geq 0\bigr\}
\]
is a hypercover.
\end{itemize}
It follows from Corollary 1.6 of~\cite{Dugger--Isaksen} that, for each finite set $J$, the map from the colimit
\[
\underset{i\geq 0} \colim   \conf_J(X_i)_{\Sigma_J}~\xra{~\simeq~} ~ \conf_J(X)_{\Sigma_J}
\]
is an equivalence of spaces.
It follows from Lemma~\ref{EE-equivs} that the functor
\[
\colim_i\disk(\cB)_{/X_i} \to \disk(\cB)_{/X}
\]
is an equivalence of $\infty$-categories.  
In particular, for each $\disk(\cB)$-algebra $A$, the canonical arrow from colimits in $\cV$ (which exist, because $\cV$ admits filtered colimits)
\[
\underset{i\geq 0}\colim \int_{X_i} A~{}~\simeq~{}~\underset{i\geq 0} \colim ~\underset{(U\hookrightarrow X_i)\in \disk(\cB)_{/X_i}} \colim A(U)~{}~\xra{~\simeq~}~{}~ \underset{(U\hookrightarrow X)\in \disk(\cB)_{/X}} \colim A(U)~{}~\simeq~{}~ \int_X A
\]
is an equivalence.  

\end{proof}

\begin{cor}\label{mfld-sftd-cpltn}

The restricted Yoneda functor $\disk(\cB)_{/-} \colon \mfld(\cB) \to \Psh\bigl(\disk(\cB)\bigr)$ factors through $\Psh_{\Sigma}\bigl(\disk(\cB)\bigr)$, the free sifted cocompletion of $\disk(\cB)$.  

\end{cor}
\begin{proof}
This is immediate from Corollary~\ref{generation}, after Corollary~\ref{fact-excisive}.

\end{proof}

\begin{remark}[Factorization homology is not a homotopy invariant]
We follow up on Remark~\ref{confs}.  
It is known that $\conf_I(M)$ is not a homotopy invariant of the argument $M$~(\cite{simple}).  
We conclude formally that the functor $\mfld(\cB) \to \Psh\bigl(\disk(\cB)\bigr)$ does not factor through the essential image of the underlying space functor $\mfld(\cB) \to \spaces$.
In other words, factorization homology is \emph{not} a homotopy invariant of manifolds, in general.

\end{remark}

\begin{theorem}[Characterization of factorization homology]\label{hmlgy=FH}
Let $X$ be a $\cB$-manifold.
The factorization homology functors each implement an equivalence of $\infty$-categories
\[
\int_-\colon \Alg_{\disk(\cB)}(\cV) \xra{~\simeq~} \bcH\bigl(\mfld(\cB),\cV\bigr)
\qquad \text{ and }\qquad
\int_-\colon \Alg_{X}(\cV) \xra{~\simeq~} \bcH\bigl(\mfld(\cB)_{/X},\cV\bigr)
~.\]

\end{theorem}

\begin{proof}
The proofs of the two equivalences are identical, so we only give that of the first.
Immediately from Corollary~\ref{generation} we have that the forgetful functor $\bcH\bigl(\mfld(\cB),\cV\bigr) \to \Alg_{\disk(\cB)}(\cV)$ is conservative.  

Now, let $H\colon \mfld(\cB) \to \cV$ be a symmetric monoidal functor.
There is a canonical arrow $\int_- (H_{|\disk(\cB)}) \to H$ between symmetric monoidal functors.
Corollary~\ref{fact-excisive} gives that the domain of this arrow is a homology theory.
So $H$ is a homology theory if this arrow is an equivalence.  
Conversely, because $\int_-$ is fully faithful (Proposition~\ref{fact-explicit}) the above paragraph gives that this canonical arrow is an equivalence whenever $H$ is a homology theory.

\end{proof}

\section{Homotopy invariant homology theories}
In this section we give two classes of examples of homology theories.  
The first class is quite formal, and depends only on the homotopy type of the tangent classifier, ${\sf Entr}(X)\xra{\tau_X} \cB$, of a $\cB$-manifold.
The second class is not homotopy invariant in general, yet we identify an understood subclass for when these examples only depend on the \emph{proper} homotopy type of the tangent classifier -- this is the statement of non-abelian Poincar\'e duality for structured stratified spaces.
\\
Fix an $\oo$-category of basics $\cB$.  

\subsection{Classical homology theories}
Recall from~\S\ref{recollections} the tangent classifier functor
$
\mfld(\cB) \xra{\tau} \Psh(\cB),
$
which is symmetric monoidal with respect to coproduct of presheaves (see Construction~\ref{B-monoidals}).  

To state the next result, for $\cC$ an $\infty$-category we denote the full $\infty$-subcategory $\cC\subset\Psh^{\sf Ind\text{-}fin}(\cC)\subset \Psh(\cC)$ which is the smallest that is closed under the formation of finite colimits and filtered colimits.  
\begin{prop}[\S3.3 of~\cite{aft1}]\label{universal-excision}
The symmetric monoidal functor $\mfld(\cB) \xra{\tau} \Psh(\cB)$ is a homology theory, and it factors through $\Psh^{\sf Ind\text{-}fin}(\cB)$.  

\end{prop}

\begin{cor}
Let $\cV$ be an $\infty$-category that admits pushouts and filtered colimits, which we regard as a symmetric monoidal $\infty$-category whose symmetric monoidal structure is given by coproduct.  
Let $F\colon \Psh^{\sf Ind\text{-}fin}(\cB) \to \cV$ be a symmetric monoidal functor that preserves pushouts and filtered colimits (so $F$ is the left Kan extension of its restriction $F_{|\cB}$).
Then the composition 
\[
F\tau \colon \mfld(\cB) \xra{\tau} \Psh^{\sf Ind\text{-}fin}(\cB) \xra{F} \cV
\]
is a homology theory.  

\end{cor}

\begin{example}
\begin{itemize}
\item[~]

\item The \emph{underlying space} functor $\mfld(\cB) \to \spaces$ is a homology theory -- here, we are equipping the target with the symmetric monoidal structure given by coproduct.  
The restriction of this functor to $\mfld(\cB)^{\sf fin}$ (see \S1.1) factors through $\spaces^{\sf fin}$.  

\item The \emph{underlying space of the `frame bundle'} functor $\mfld(\cB) \to \bigl(\spaces_{/|\cB|}\bigr)$ is a homology theory -- here, we are equipping the target with the symmetric monoidal product given by coproduct.  
Note that the restriction of this functor to $\mfld(\cB)^{\sf fin}$ factors through $\bigl(\spaces_{/|\cB|}\bigr)^{\sf fin}$.  

\item Let $\cV$ be a presentable $\infty$-category.
Consider a functor $E\colon \cB\to \cV$.
Use the same notation $E\colon \Psh(\cB) \to \cV$ for the left Kan extension.  
This left Kan extension preserves coproducts, and therefore defines a symmetric monoidal functor, which we again give the same notation, between their coCartesian symmetric monoidal strutures:
$
E\colon \Psh(\cB)^{\amalg} \longrightarrow \cV^{\amalg}.
$
We obtain a composite symmetric monoidal functor
\[
\mfld(\cB) \xra{\tau} \Psh(\cB)^\amalg \xra{E} \cV^{\amalg}~.
\]
In~\S3.3 of~\cite{aft1} we show that $\tau$ carries collar-gluings to pushouts and sequential open covers to sequential colimits.
Using this, and because this left Kan extension $E$ preserves sifted colimits, this symmetric monoidal functor $E\tau$ is a homology theory. 
\\
Should $\cV$ be equipped with an auxiliary symmetric monoidal structure that distributes over colimits, and should the original functor $E\colon \cB\to \cV$ factor through the full $\infty$-subcategory $\cV^{\sf dual}\subset \cV$ consisting of the dualizable objects, then the restriction of this homology theory $E\tau$ to $\mfld(\cB)^{\sf fin}$ factors through $\cV^{\sf dual}$.  
\begin{itemize}
\item Let $E$ be a spectrum.  The assignment $X\mapsto E\wedge X_+$ depicts a homology theory $\mfld(\cB) \to \spectra$, where the latter is equipped with wedge sum as its symmetric monoidal structure.
Should $E$ be dualizable with respect to smash product, then the restriction of this homology theory to $\mfld(\cB)^{\sf fin}$ factors through $\spectra^{\sf dual}$. 
In particular, the suspension spectrum $\Sigma^\infty_+ X$ of the underlying space of a finitary stratified space $X$ is dualizable.

\item Let $V$ be a chain complex over a commutative ring $\Bbbk$.  
The assignment $X\mapsto \sC_\ast(X;V)$ depicts a homology theory $\mfld(\cB) \to \Ch_\Bbbk$, where the latter is equipped with direct sum as its symmetric monoidal structure.  
Should $V$ be dualizable with respect to tensor products over $\Bbbk$, then the restriction of this homology theory to $\mfld(\cB)^{\sf fin}$ factors through $\Ch_\Bbbk^{\sf dual}$.  
\end{itemize}
\end{itemize}

\end{example}

\subsection{1-point compactifications}

We will make use of the concept of a zero-pointed embedding from \cite{ZP}. There is a $\Kan$-enriched functor 
\[
{\sf ZEmb}\colon \snglr^{\op} \times \snglr^{\op} \longrightarrow \Kan_\ast
\]
to pointed Kan complexes, given as follows.
Its value on $(X,Y)$ is the simplicial set, written as ${\sf ZEmb}(X_+,Y^+)$, for which a $p$-simplex is a diagram of conically smooth open embeddings among submersions over $\Delta^p_e$
\[
X \times \Delta^p_e\xra{~f_X~} W  \xla{~f_Y~} Y \times \Delta^p_e
\]
witnessing an open cover of $W$ -- the simplicial structure maps are evident, and the distinguished point is the case where $W$ is the disjoint union.  
That this simplicial set is a Kan complex as claimed follows from a similar argument for why the simplicial set $\snglr(X,Y)$ is a Kan complex which is explained in~\cite{aft1}; details for ${\sf ZEmb}$ can be found in~\cite{ZP}.  
The action $\snglr(X,X')\times \snglr(Y,Y') \times {\sf ZEmb}(X'_+,{Y'}^+) \to {\sf ZEmb}(X_+,Y^+)$ is given on $p$-simplices as 
\[
\bigl(g,h; (X'\times \Delta^p_e \xra{f_{X'}} W' \xla{f_{Y'}} Y'\times \Delta^p_e)\bigr) \mapsto \bigl(X\times \Delta^p_e \xra{f_{X'}g} f_{X'}(g(X)) \cup f_{Y'}(h(Y)) \xla{f_{Y'}h} Y\times \Delta^p_e\bigr)~,
\]
which is quickly noticed to be compatible with the simplicial structure maps. 
It is manifest that this action is compatible with the composition among conically smooth open embeddings.  
Equipping all stratified spaces present in the definition of the functor $\sf ZEmb$ with a $\cB$-structure, and each map as one of $\cB$-manifolds, enhances $\sf ZEmb$ to a functor
\[
\mfld(\cB)^{\op}\times \mfld(\cB)^{\op} \to \spaces_\ast
\]
to based spaces.  
The restricted adjoint to this functor is the \emph{1-point compactified} tangent classifier
\begin{equation}\label{tau-plus}
\tau^+\colon \bigl(\mfld(\cB)^{\sf fin}\bigr)^{\op} \to \Psh^{\sf fin}_\ast(\cB)~,\qquad X\mapsto \bigl({\sf Entr}(X)^+ \xra{~\tau_{X^+}~} \cB\bigr)~,
\end{equation}
where the target is endowed with the symmetric monoidal structure given by coproduct, written here as right fibrations (details and context can be found in~\cite{ZP}).  

Here is a relative version of the functor $\tau^+$.  
Let $X$ be a $\cB$-manifold.
Let us define the following functor:
\[
\tau^{X^+}\colon \mfldd(\cB)_{/X} \longrightarrow \Psh_\ast(\cB)_{/{\sf Entr}(X^+)}~, \qquad (O\subset X)\mapsto \bigl({\sf Entr}(O_{X^+}) \xra{\tau_{O}^{X^+}} \cB\bigr)~.
\]  
Let $O\subset X$ be an open subspace of the underlying stratified space.
Consider the sub-$\Kan$-enriched functor ${\sf ZEmb}(-_+,O_{X^+}) \subset {\sf ZEmb}(-_+,X^+) \colon \bsc^{\op} \to \Kan$ whose value on $U$ is the sub-simplicial set of ${\sf ZEmb}(U,X)$ consisting of those $U\times \Delta^p_e \xra{f_U}W\xla{f_X} X\times \Delta^p_e$ for which $f_{U} \amalg (f_X)_{|O} \to W$ is an open cover -- easy to check is that this simplicial set is indeed a Kan complex. 
Again, equipping each such $W$ with a $\cB$-structure and each map as one of $\cB$-manifolds, there is the presheaf $\cB^{\op} \to \spaces_\ast$.  
We will use the notation ${\sf Entr}(O_{X^+}) \to \cB$ for the associated right fibration, which is equipped with a section.
The assignment $(O\subset X)\mapsto \bigl({\sf Entr}(O_{X^+}) \xra{\tau_{O}^{X^+}} \cB\bigr)$ depicts the desired functor.
Notice that this functor $\tau^{X^+}$ canonically extends as a map of $\infty$-operads, where the target is equipped with coproduct as its symmetric monoidal structure.

\begin{lemma}\label{pre-excision}
Let $X\cong X_- \underset{\RR\times X_0}\bigcup X_+$ be a collar-gluing among finitary $\cB$-manifolds. 
Then, in the canonical diagram of pointed presheaves on $\cB$
\[
\xymatrix{
{\sf Entr}\bigl((\RR\times X_0)_{X^+}\bigr)  \ar[r]  \ar[d]
&
{\sf Entr}\bigl((X_+)_{X^+}\bigr)  \ar[r]  \ar[d]
&
{\sf Entr}\bigl((X_+)^+\bigr)  \ar[d]^-=
\\
{\sf Entr}\bigl((X_-)_{X^+}\bigr)  \ar[r]  \ar[d]
&
{\sf Entr}(X^+)  \ar[d]  \ar[r]
&
{\sf Entr}\bigl((X_+)^+\bigr)
\\
{\sf Entr}\bigl((X_-)^+\bigr)  \ar[r]^-=
&
{\sf Entr}\bigl((X_-)^+\bigr) 
&
}
\]
the upper left square is a pushout, and the linear sequences of maps are cofibration sequences.  

\end{lemma}

\begin{proof}
The appearance of $\cB$ is superficial, and so the statement is equivalent to the one with $\cB = \bsc$.  

Because the collar-gluing, written as a constructible bundle $X\xra{f}[-1,1]$, is one among \emph{finitary} stratified spaces, it can be extended, up to non-canonical isotopy, to a collar-gluing $\ov{X} \xra{\w{f}}[-1,1]$ among stratified spaces \emph{with boundary}, with $\ov{X}$ compact.  
By composing $\w{f}$ with the constructible bundle $[-1,1]\to [-1,1]$ determined by declaring its restriction to $(-\frac{1}{2},\frac{1}{2})$ to be an orientation preserving isomorphism onto $(-1,1)$, we can assume that the closure of the image of $\RR\times \ov{X}_0 \hookrightarrow \ov{X}$ intersects $\partial \ov{X}$ as $\RR\times \partial \ov{X}_0$. 

Let $\ov{O}\subset \ov{X}$ be a sub-stratified space with boundary, and choose a collar-neighborhood $[0,1)\times\partial \ov{O}\subset \ov{O}$.
We claim that the canonical morphism of pointed presheaves on $\bsc$
\[
\ast \underset{{\sf Entr}(0,1)\times \partial \ov{O})} \coprod {\sf Entr}(O)\xra{~\simeq~}{\sf Entr}(O_{X^+})
\]
is an equivalence.
This is the case if and only if for each singularity type $[U]$, the likewise map of pointed spaces 
\[
\ast \underset{{\sf Entr}_{[U]}((0,1)\times \partial \ov{O})} \coprod {\sf Entr}_{[U]}(O)\xra{~\simeq~}{\sf Entr}_{[U]}(O_{X^+})
\]
is an equivalence.
Write $U\cong \RR^i\times \sC(Z)$, and use the notation $\ov{U} := [0,1)\times \RR^{i-1}\times \sC(Z)$ for the basic with boundary and $\partial \ov{U} := \RR^{i-1}\times \sC(Z)$ for its boundary.  
Developments of~\cite{aft1} give that the lefthand space in this last expression is canonically identified as the pushout 
\[
\ast \underset{((0,1)\times \partial\ov{O})_{[U]}} \coprod O_{[U]}
\]
in terms the underlying spaces of the $[U]$- and $[\partial \ov{U}]$-strata.
Along the same lines, the righthand pointed space in that expression is canonically identified as the pushout 
\[
\ast \underset{(\partial \ov{O})_{\partial[\ov{U}]}} \coprod \ov{O}_{[U]\leq [\ov{U}]}
\]
in terms of the underlying spaces of the $([U]\leq [\ov{U}])$- and $[\partial \ov{U}]$-strata. 
That the the map of pointed spaces in that expression is an equivalence follows by inspection: $((0,1)\times \partial \ov{O})_{[U]} = \RR\times (\partial \ov{O})_{[\partial \ov{U}]}\simeq (\partial \ov{O})_{[\partial \ov{U}]}$, and the inclusion $O_{[U]} \subset \ov{O}_{[U]\leq [\ov{U}]}$ induces a weak homotopy equivalence of underlying topological spaces.  

Developments in~\cite{aft1} give that each of the two squares of presheaves on $\bsc$
\[
\xymatrix{
{\sf Entr}\bigl(\RR\times X_0\bigr)  \ar[r]  \ar[d]
&
{\sf Entr}\bigl(X_+\bigr)   \ar[d]
\\
{\sf Entr}\bigl(X_-\bigr)  \ar[r]  
&
{\sf Entr}(X)  
}
\]
and
\[
\xymatrix{
{\sf Entr}\bigl(\RR\times (0,1)\times \partial \ov{X}_0\bigr)  \ar[r]  \ar[d]
&
{\sf Entr}\bigl((0,1)\times \partial\ov{X}_+\bigr)   \ar[d]
\\
{\sf Entr}\bigl((0,1)\times \partial \ov{X}_-\bigr)  \ar[r]  
&
{\sf Entr}((0,1)\times \partial \ov{X})  
}
\]
are pushouts. 
The latter maps to the former upon consistent choices of collars of the respective boundaries.  
By taking the levelwise mapping cones of this map of squares, the result follows then from the first paragraph, after observing a standard strata-preserving homeomorphism $(X_-)^+ \cong \bigr(\ov{X}_- \setminus \RR_{>-1}\times \ov{X}_0\bigr)/ \partial$, where $\partial =\bigl( \partial \ov{X}_- \bigcup \{-1\}\times \ov{X}_0\bigr) \setminus \RR_{>-1}\times \ov{X}_0 $, and likewise for $(X_+)^+$.  

\end{proof}

\subsection{Compactly supported cohomology}

Let $\cE$ be a spectrum object of $\Psh_\ast(\cB)$, a model for which is as a functor $\cB^{\op} \to \spectra$.  
Now, Corollary~1.4.4.5 of~\cite{HA} grants that a finite limit preserving functor from a stable $\infty$-category to $\spaces$ canonically factors through $\spectra \xra{\Omega^\infty} \spaces$.  
In particular, for $\cS$ a stable $\infty$-category, the Yoneda functor $\cS^{\op}\times \cS \to \spaces$ canonically factors through $\spectra \xra{\Omega^\infty}\spaces$. 
In particular, there is a factorization of the Yoneda functor
\[
\Map_{\Psh_\ast(\cB)}(-,\Omega^\infty\cE)\colon \Psh_\ast(\cB)^{\op} \xra{\cE^-} \spectra \xra{\Omega^\infty} \spaces_\ast~,
\]
the first of which is symmetric monoidal with respect to coproduct on the source and on the target, and the second of which is symmetric monoidal with respect to coproduct on the source and product on the target.  
We will denote the composite symmetric monoidal functor
\begin{equation}\label{E-c}
\cE_{\sf c}\colon \mfld(\cB)^{\sf fin} \xra{\tau^+}\Psh^{\sf fin}_\ast(\cB)^{\op}\xra{\cE^-} \spectra
\end{equation}
where here the target is equipped with the symmetric monoidal structure given by wedge sum.
We will use the notation
\[
A_\cE ~:=~ (\cE_{\sf c})_{|\disk(\cB)} \in \Alg_{\disk(\cB)}(\spectra)
\]
for the restriction.  
We give notation for the relative version:
\[
\cE_{X^+}\colon \mfld(\cB)_{/X^+} \xra{\tau^{X^+}}\Psh_\ast(\cB)\xra{\cE^-} \spectra~.  
\]

\begin{prop}[Compactly supported cohomology]\label{sections-excision}
Let $\cE$ be a spectrum object of $\Psh_\ast(\cB)$.  
Then the symmetric monoidal functor $\cE_{\sf c}\colon \mfld(\cB) \to \spectra$ satisfies $\ot$-excision.
In particular, the canonical map
\[
\int_- A_\cE\xra{~\simeq~} \cE_{\sf c}(-)
\]
is an equivalence of symmetric monoidal functors $\mfld(\cB)^{\sf fin}\to \spectra$.  

\end{prop}

\begin{proof}
Let $X\cong X_-\underset{\RR\times X_0}\bigcup X_+$ be a collar-gluing among $\cB$-manifolds.  
After Lemma~\ref{pre-excision}, there is the pullback diagram among spectra
\[
\xymatrix{
\cE_{\sf c}(X^+)  \ar[r]  \ar[d]
&
\cE_{X^+}(X_+)  \ar[d]
\\
\cE_{X^+}(X_-)  \ar[r]
&
\cE_{X^+}(\RR\times X_0)~.
}
\]
Thereafter, using the canonical identification $\Omega \cE_{X^+}(\RR\times X_0) \simeq \cE_{\sf c}(\RR\times X_0)$ and Lemma~\ref{pre-excision} again, there is the pushout diagram of spectra
\[
\xymatrix{
\cE_{\sf c}(X_0) \ar[r]  \ar[d]
&
\cE_{\sf c}(X_+)  \ar[d]
\\
\cE_{\sf c}(X_-)\ar[r]
&
\cE_{\sf c}(X^+)~.
}
\]

Because we are equipping spectra with wedge sum, which is its categorical coproduct, the forgetful functor $\Alg_{\Ass^{\sf RL}}(\spectra) \xra{\simeq} \Fun\bigl((-\la 0 \to +), \spectra\bigr)$ is an equivalence; and, for $(P\la A \to Q)$ an object of this functor category, regarded as an $\Ass^{\sf RL}$-algebra in $\spectra$, then the canonical map from the pushout $P\underset{A} \coprod Q \xra{\simeq} P\underset{A}\bigotimes Q$ is an equivalence of spectra.  
And so, we have verified that $\cE_{\sf c}$ satisfies $\ot$-excision.  

\end{proof}

Specializing Proposition~\ref{sections-excision} to the case that $\cE = \SS \colon \cB^{\op} \to \spectra$ is the constant functor at the sphere spectrum, we have the following classical consequence.
For this case, we will use the special notation 
\[
\DD\bigl((-)^+\bigr)~:=~\SS_{\sf c}(-) ~,\qquad \text{ and }\qquad \omega ~:=~ A_{\SS}~.
\]
This notation is invoked because, for each $\cB$-manifold $X$, the spectrum $\DD(X^+)$ is the Spanier-Whitehead dual of the 1-point compactification of the underlying space of $X$; and the value of $\omega$ on a basic $U$ is $\DD(U^+)$, the stalk of the dualizing sheaf for the site $\mfldd(\cB)$ at $U$.  

Immediate from the present definitions is the following result. 
\begin{prop}
Let $\cB = \BO(n)$, so that a $\cB$-manifold is an ordinary smooth $n$-manifold.  
Let $X$ be smooth $n$-manifold that is the interior of a compact manifold with boundary.
Then there is a canonical identification 
\[
X^{-\tau_X} \simeq \int_X\omega
\] 
from the Thom spectrum of the virtual negative of the tangent bundle of $X$.  

\end{prop}

\begin{cor}[Atiyah duality]\label{atiyah}
Let $X$ be a finitary $\cB$-manifold.
Then there is a canonical identification
\[
\int_X \omega  ~\simeq~ \DD(X^+)~.
\]
More generally, for each functor $\cE\colon \cB^{\op} \to \spectra$, there is a canonical identification
\[
\int_X \cE\wedge \omega~\simeq~ \cE_{\sf c}(X)~.
\]

\end{cor}

\begin{example}[Classical Poincar\'e duality]
Let us consider the case $\cB = \sD_n^{\sf or}$ so that a $\cB$-manifold is an oriented smooth $n$-manifold; and an oriented $n$-manifold $X$ that is the interior of a compact manifold with boundary; and $\cE = \sH\ZZ\colon (\sD_n^{\sf or})^{\op} \to \spectra$ is the constant functor at the Eilenberg-MacLane spectrum for $\ZZ$.
Identify $\sH\ZZ_{\sf c}(X)\simeq \sH\ZZ^{X^+}$ as the mapping spectrum from the 1-point compactification.  
For each object $U\in \sD_n^{\sf or}$ there is a canonical identification $\sH\ZZ\wedge \omega(U)\simeq \sH\ZZ\wedge \SS^{-n}\simeq \Omega^n \sH\ZZ$.    
And so, there is a canonical equivalence of spectra $\int_X \sH\ZZ\wedge \omega~\simeq \Omega^n\sH\ZZ\wedge X_+$.
Putting these identification into Corollary~\ref{atiyah} arrives at a canonical equivalence of spectra:
$
\Omega^n \sH\ZZ\wedge X_+~\simeq~ \sH\ZZ^{X^+}.  
$
Upon taking homotopy groups, indexed appropriately, we arrive at classical Poincar\'e duality: 
\[
\sH_\ast(X;\ZZ)\cong \sH^{n-\ast}_{\sf c}(X;\ZZ)~.
\]

\end{example}

\subsection{Non-abelian Poincar\'e duality}

\begin{definition}[Coefficient system]\label{def.coef-sys}
A \emph{coefficient system (for $\cB$-manifolds)} is a pointed presheaf on $\cB$.  
Let $E$ be a coefficient system.
The \emph{compactly supported sections} functor is the composition of the symmetric monoidal functors
\[
\Gamma_{\!\sf c}^E\colon \mfld(\cB)^{\sf fin} \xra{\tau^+} \Psh^{\sf fin}_\ast(\cB)^{\op} \xra{\Map_{\Psh_\ast(\cB)}(-,E)} \spaces_\ast
\]
where the target is equipped with the symmetric monoidal structure given by product.
We denote the resulting $\disk(\cB)$-algebra
\[
A_E~:=~ \bigl(\Gamma_{\!\sf c}^E\bigr)_{|\disk(\cB)}~\in\Alg_{\disk(\cB)}(\spaces_\ast)~.
\]
\end{definition}

\begin{example}
Consider a flag of maps of based spaces $\ZZ_{\geq 0}^{\op} \xra{Z_\bullet} \spaces_\ast$.  It corresponds to a right fibration $Z\to \ZZ_{\geq 0}$, equipped with a section.  
In~\cite{aft1} is constructed the functor ${\sf Depth}\colon \bsc \to \ZZ_{\geq 0}$.  
There results the coefficient system ${\sf Depth}^\ast Z \to \bsc$.  

\end{example}

\begin{example}\label{ordinary-coefs}
Let $\cB = \sD^{\fr}_n\simeq \ast$. Then a coefficient system is equivalent to the datum of a based space $Z$.  
Such a coefficient system is connective, in the sense of Definition \ref{connective-coefs}, exactly if $Z$ is $n$-connective.  

\end{example}

\begin{example}\label{coefs-corners}
A coefficient system on $\sD^\partial_n$ is map of fibrations 
\[
\xymatrix{
E_{n-1}  \ar[r]  \ar[d]
&
E_n   \ar[d]
\\
{\BO}(n-1) \ar[r]  
&
{\BO}(n)
}
\]
together with a pair of (compatible) sections of each.  

Consider the more elaborate example $\sD_{\langle n \rangle}$ codifying smooth $n$-manifolds with corners.  
A coefficient system on $\sD_{\langle n \rangle}$ is the data of a fibration $E_S \to\BO(\RR^S)$ for each subset $S\subset \{1,\dots,n\}$, and for each inclusion $S\subset T$ a map $E_S \to E_T$ over the inclusion $\BO(\RR^S) \xra{-\oplus \RR^{T\smallsetminus S}} \BO(\RR^T)$, which respect composition; together with a compatible section of each of these fibrations.  

\end{example}

We now concern ourselves with the question of when $\Gamma_{\!\sf c}^E$ is a homology theory.
We will only consider this question in the case that there is an integer $n\geq 0$ such that each basic $U\in \cB$ has pure dimension $n$ -- that is to say, the local topological dimension at each point in the underlying stratified space equals $n$.  
To address the problem, it is convenient to extend $\Gamma_{\!\sf c}^E$ to stratified spaces $X$ equipped with a $\cB$-structure on $\RR^{n-k}\times X$ for some $k$.
Here are the relevant temporary definitions.
\begin{definition}[$\mfld(\cB_k)$]\label{B-less-than}
Let $k$ be an integer.
Denote the $\infty$-categories which are the pullbacks
\[
\xymatrix{
\cB_k \ar[r]  \ar[d]
&
\mfld(\cB_{k})  \ar[rr]^-{\RR^{n-r}\times -}  \ar[d]
&&
\mfld(\cB)  \ar[d]
\\
\bsc  \ar[r]^-{\iota}
&
\snglr \ar[rr]^-{\RR^{n-k}\times -}
&&
\snglr.
}
\]
Denote the composite functor 
\[
\Gamma_{\!\sf c}^E\colon \mfld(\cB_k) \xra{\bigl((-1,1)^{n-k} \subset \RR^{n-k}\bigr)\times -} \Fun\bigl([1],\mfld(\cB)\bigr)\xra{{\sf Entr}\bigl({\sf ev}_0(-)_{{\sf ev}_1(-)^+}\bigr)} \spaces_\ast~.
\]
Explicitly, this functor evaluates as
\[
\Gamma_{\!\sf c}^E\colon X~\mapsto~{\sf Entr}\Bigl(\bigl((-1,1)^{n-k}\times X\bigr)_{(\RR^{n-k}\times X)^+}\Bigr)
\]
which we think of as ``sections of $E_{|\RR^{n-k}\times X}$ that are compactly supported in the $X$-direction''.  

\end{definition}

\begin{definition}\label{connective-coefs}
Let $\cB$ be a category of basics each of pure dimension $n$.
Say the coefficient system $E$ is \emph{connective} if the based space $\Gamma_{\!\sf c}^E(V)$ is connected for every $V\in \cB_{k}$ with $k<n$.  
\end{definition}

\begin{example}\label{fr-coefs}

Consider the case $\cB=\sD_n^{\fr}\simeq \ast$.
Here, a connective coefficient system is the datum of an $n$-connective based space $Z$, and the associated $\disk^{\fr}_n$-algebra is $\Omega^n Z$.
This $\disk_n^{\sf fr}$-algebra is \emph{group-like}.  
Through May's recognition principle there is an equivalence of $\infty$-categories between connective coefficient systems for framed $n$-manifolds, and that of group-like $\cE_n$-spaces.  
Accordingly, one can think of the data of a connective coefficient system as a generalization of the notion of a group-like algebra, for the structured stratified setting.
\end{example}

\begin{example}
Let us return to Example~\ref{coefs-corners} for the case of $\sD_n^\partial$.   
For simplicity, let us assume the two fibrations are trivialized with (based) fibers $Z_n$ and $Z_{n-1}$ respectively.  
This coefficient system is connective exactly if $Z_n$ is $n$-connective and the map $Z_{n-1} \to Z_n$ is $n$-connective; the last condition being equivalent to saying the homotopy fiber $F$ of the map $Z_{n-1}\to Z_n$ is $(n-1)$-connective.  
Indeed, for this coefficient system to be connective, then both $\Gamma_{\!\sf c}^E(\RR^k) \simeq \Omega^k Z_n$ and $\Gamma_{\!\sf c}^E(\HH^k)\simeq \Omega^k F$ must be connected for each $k<n$.
Now, recall from Example~\ref{deligne-boundary} a consolidation of the data of a $\disk_n^\partial$-algebra.
The associated $\disk^{\partial}_n$-algebra $A_E$ is the data $(\Omega^n Z_n, \Omega^{n-1}F , \alpha)$ where $\alpha$ is the action of $\Omega^n Z_n$ on $\Omega^{n-1} F$ from the $\Omega$-Puppe sequence of the fibration $F \to Z_{n-1} \to Z_n$.  
\end{example}

\begin{remark}
The previous example can be generalized to the case of $\sD_{\langle n\rangle}$, as we indicate here.
For simplicity, let us assume each of the said fibrations is trivial with respective fibers $Z_S$.  Denote by $F_T = \mathsf{hofib}(Z_T \to \holim_{S\subsetneq T} Z_S)$ the total homotopy fiber of the $T$-subcube.  
This coefficient system is connective exactly if each $F_T$ is $(n-|T|)$-connective.  
The associated $\disk_{\langle n \rangle}$-algebra $A_E$ is the data $\bigl((\Omega^{n\smallsetminus S} F_S\bigr)_{S\subset \{1,\dots,n\}}; \bigl(a_{S\subset T})\bigr)$ where $a_{S\subset T}$ is the action of $\Omega^{n\smallsetminus S}F_S$ on $\Omega^{n\smallsetminus T} F_T$ from an elaboration of the $\Omega$-Puppe sequence.  
\end{remark}

\begin{theorem}(Non-abelian Poincar\'e duality)\label{non-abel}
Let $\cB$ be an $\oo$-category of basics each of pure dimension $n$.  
Let $E$ be a connective coefficient system for $\cB$-manifolds.
Then the canonical arrow
\[
\int_- A_E~\xra{~\simeq~} ~\Gamma_{\!\sf c}^E(-)
\]
is an equivalence of functors $\mfld(\cB)^{\sf fin} \to \spaces_\ast$. 

\end{theorem}

\begin{proof}
From Theorem~\ref{hmlgy=FH}, we must show $\Gamma_{\!\sf c}^E\colon \mfld(\cB) \to \spaces$ satisfies excision.  
Consider a collar-gluing $X\cong X_-\underset{\RR\times X_0}\bigcup X_+$ among $\cB_{k}$-manifolds where $-1\leq k \leq n$.
We will show that the canonical map
\[
\Gamma_{\!\sf c}^E(X_-)\underset{\Gamma_{\!\sf c}^E(\RR\times V)} \bigotimes \Gamma_{\!\sf c}^E(X_+) \longrightarrow \Gamma_{\!\sf c}^E(X)~,
\]
is an equivalence; the desired case is $k=n$.  
There is the sequence of maps of pointed spaces
\begin{equation}\label{conn-base?}
\Gamma_{\!\sf c}^E(X_-)\times \Gamma_{\!\sf c}^E(X_+) \to \Gamma_{\!\sf c}^E(X) \to \Gamma_{\!\sf c}^E(X_0)~,
\end{equation}
which is a fiber sequence because of Lemma~\ref{pre-excision} applied to the modified collar-gluing $X \cong (X_-\sqcup X_+)\underset{\RR\times (X_0\sqcup X_0)}\bigcup \RR\times X_0$.
It is sufficient to show that the base of this fiber sequence is connected.  
We will do this by induction on $k$.  

Let $Y$ be a finitary $\cB_k$ manifold.
If $k=-1$, then $\Gamma_{\!\sf c}^E(Y) \simeq \ast$ is terminal.  In particular it is connected.
So suppose $\Gamma_{\!\sf c}^E(W)$ is connected for each $\cB_j$ manifold $W$ for $j<k$.   
From Theorem~\ref{recollections}, $Y$ can be witnessed as a finite iteration of collar-gluings of basics.  
We prove $\Gamma_{\!\sf c}^E(Y)$ is connected by induction on the minimal number $r$ of iterated collar-gluings to obtain $Y$.  
If $r=0$ the statement is vacuously true.   
If $Y$ is a basic then $\Gamma_{\!\sf c}^E(Y)$ is connected, by the connectivity assumption.  
If $r\geq 2$, write $Y = Y_-\underset{\RR\times W}\bigcup Y_+$ with each of the $\cB_k$-manifolds $Y_\pm$ witnessed through strictly fewer than $r$ collar-gluings, and with $W$ a $\cB_{k-1}$-manifold.  
By induction on $r$, each of the spaces $\Gamma_{\!\sf c}^E(Y_\pm)$ is connected.  
As in~(\ref{conn-base?}), there is the fiber sequence of pointed spaces
\[
\Gamma_{\!\sf c}^E(Y_-)\times \Gamma_{\!\sf c}^E(Y_+) \to \Gamma_{\!\sf c}^E(Y) \to \Gamma_{\!\sf c}^E(W)~;
\]
now the base is connected, by induction on $k$.
It follows that $\Gamma_{\!\sf c}^E(Y)$ is connected.  
\end{proof}

\begin{remark}\label{remark:grouplike}
We elaborate on Example~\ref{fr-coefs}.  
The space of stratified continuous maps is a homotopy invariant of the underlying stratified space $X$. 
In this sense, Theorem~\ref{non-abel} tells us that connective coefficient systems (think, `$\disk(\cB)$-spaces that are \emph{group-like}') cannot detect more than the stratified proper homotopy type of $\cB$-manifolds. 

\end{remark}

\section{Examples of factorization homology theories}\label{section:examples}

In this section we give examples of factorization homology over stratified spaces. To illustrate the relevance to low-dimensional topology, we show that the free $\disk_{3,1}^{\fr}$-algebra can distinguish the homotopy type of link complements, and in particular defines a non-trivial link invariant.

\subsection{Factorization homology of stratified 1-manifolds}

When the target symmetric monoidal $\oo$-category $\cV$ is $\m_\Bbbk^\ot$, the category of $\Bbbk$-modules for some commutative algebra $\Bbbk$, then factorization homology of closed 1-manifolds gives variants of Hochschild homology.

The simplest and most fundamental example is factorization homology for framed 1-manifolds, $\mfld_1^{\fr}$. In this case, there is an equivalence between framed 1-disk algebras and associative algebras in $\cV$, $\Alg_{\disk_1^{\fr}}(\cV)\simeq \Alg(\cV)$, and we have the following immediate consequence of the excision property of factorization homology (Theorem~\ref{fact-excisive}). 

\begin{prop}\label{hochschild} For an associative algebra $A$ in $\m_d$, there is an equivalence \[\int_{S^1}A \simeq \hh_*(A)\] between the factorization homology of the circle with coefficients in $A$ and the Hochschild homology of $A$ relative $\Bbbk$.
\end{prop}

\begin{proof} We have the equivalences \[\int_{S^1}A\simeq \int_{\RR^1}A\underset{{\underset{{S^0\times\RR^1}}\int\!\! A}}\ot\int_{\RR^1}A\simeq A\underset{A\ot A^{\op}}\ot A\simeq \hh_*(A)\]using excision and a decomposition of the circle by two slightly overlapping hemispheres.\end{proof}

\begin{remark} Lurie in \S5.5.3 of ~\cite{HA} shows further that the obvious circle action by rotations on $\int_{S^1}A$ agrees with the usual simplicial circle action on the cyclic bar construction.
\end{remark}

It is interesting to probe this example slightly further and see the algebraic structure that results when one introduces marked points and singularities into the 1-manifolds. Recall the $\oo$-category $\mfld^{\fr}_{1,0}$ of framed 1-manifolds with marked points, and the $\infty$-subcategory $\sD^{\fr}_{1,0}$ of framed 1-disks with at most one marked point -- its set of objects is the two-element set $\{U^1_{\emptyset^{-1}} , U^1_{S^0}\}$ whose elements we justifiably denote as $\RR^1:= U^1_{\emptyset^{-1}}$ and $(\RR^1,\{0\}) := U^1_{S^0}$.   So $\Alg_{\disk^{\fr}_{1,0}}(\cV)$ is equivalent to the $\oo$-category whose objects are pairs $(A_1, A_{\sf b})$ consisting of an algebra $A_1$ and a unital $A_1$-bimodule $A_{\sf b}$, i.e., a bimodule with an invariant map from the unit. Specifically, $A_1 \simeq A(\RR^1)$ and $A_{\sf b}\simeq A(\RR^1, \{0\})$. The proof of the Proposition~\ref{hochschild} extends mutatis mutandis to the following.

\begin{prop} There is an equivalence \[\int_{(S^1,\ast)} A\simeq \hh_*(A_1,A_{\sf b})\] between the factorization homology of the pointed circle $(S^1,\ast)$ with coefficients in $A =(A_1,A_{\sf b})$ and the Hochschild homology of $A_1$ with coefficients in the bimodule $A_{\sf b}$.
\end{prop}

Finally, we mention the example of factorization homology for $\snglr_1^{\fr}$, the category of 1-dimensional, framed, stratified spaces. In this case, the $\oo$-category of basic opens $\bsc_1^{\fr}$ has as its set of objects $\{\RR\}\coprod \{(\sC(J), \sigma)\}$ where the latter set is indexed by finite sets $J$ together with an orientation $\sigma$ of the ordinary $1$-manifold $\bigsqcup_J \RR_{>0} = \sC(J)\smallsetminus \ast$.  

An object $A$ in $\Alg_{\disk(\bsc^{\fr}_1)}(\cV)$ is then equivalent, by evaluating on directed graphs with a single vertex, to the data of an associative algebra $A(\RR)$ in $\cV$ and for each pair $i,j\geq 0$ an object $A(i,j) \in \cV$ equipped with $i$ intercommuting left $A(\RR)$-module structures and $j$ intercommuting right compatible $A(\RR)$-module structures. One can see, for instance, that the factorization homology of a wedge of two circles with a marked point on each circle, ${(S^1\cup_{\{0\}}S^1, \{1, -1\})}$, can be calculated as \[\int_{(S^1\cup_{\{0\}}S^1, \{1, -1\})}A\simeq A(1,1)\underset{A_1\ot A_1^{\op}}\ot A(2,2)\underset{A_1\ot A_1^{\op}}\ot A(1,1)~.\]

\subsection{Intersection homology}

Recall from \cite{aft1} that the underlying space of a $n$-dimensional stratified space $X\in \snglr_n$ has a canonical filtration by the union of its strata $X_0 \subset X_1 \subset \ldots \subset X_n= X$ where each $X_i\smallsetminus X_{i-1}$ is a smooth $i$-dimensional manifold. As such, the definition of Goresky \& MacPherson's intersection homology~\cite{goreskymacpherson} applies verbatim. Namely, we restrict to stratified spaces $X$ have no codimension-1 strata, $X_{n-1}=X_{n-2}$ and for which the $n$-dimensional open stratum $X_n\smallsetminus X_{n-1}$ is nonempty.

For the definition below we use $j^{\rm th}$-stratum functor $(-)_j\colon \snglr_n \to \snglr_{\leq j}$ which assigns to stratified space $X$ the maximal substratified space $X_j$ whose strata are of dimension less than $j+1$.
\begin{definition} 
Denote the left ideal $\bsc_n^{\sf ps} \to \bsc_n$ spanned by those basics $U$ for which $U_{n-1} = U_{n-2}$.  Define the category of \emph{pseudomanifolds} as $\snglr_n^{\sf ps} = \mfld(\bsc_n^{\sf ps})$ -- its objects are those $n$-dimensional stratified spaces for which $X_{n-1} = X_{n-2}$.  
\end{definition}

Continuing, choose a perversity function $p$, i.e., a mapping $p:\{2, 3, \ldots, n\} \ra \ZZ_{\geq 0}$ such that $p(2)=0$ and for each $i>2$ either $p(i)=p(i-1)$ or $p(i) = p(i-1)+1$. Recall the following definition.

\begin{definition}[\cite{goreskymacpherson}] A $j$-simplex $g: \Delta^j \ra X$ is $p$-allowable if, for every $i$ the following bound on the dimensions of intersections holds:
\[
{\sf dim} \bigl(g(\Delta^j)\cap X_i\bigr) \leq i + j -n +p(n-i)~.
\]
A singular chain $\phi \in \sC_j(X)$ is $p$-allowable if both $|\phi|$ is $p$-allowable and $|\partial\phi|$ is $p$-allowable.
\end{definition}

This gives the following definition of Goresky--MacPherson, of intersection homology with perversity $p$.

\begin{definition}[\cite{goreskymacpherson}] The intersection homology $\sI_p\sC_\ast(X)$ of $X\in \snglr_n^{\sf ps}$ is the complex of all $p$-allowable singular chains.
\end{definition}

The condition of a simplex being $p$-allowable is clearly preserved by embeddings of stratified spaces: if $f:X \ra Y$ is a morphism in $\snglr_n^{\sf ps}$ and $g:\Delta^j \ra X$ is $p$-allowable, then $f\circ g: \Delta^j\ra Y$ is $p$-allowable. Further, being $p$-allowable varies continuously in families of embeddings. That is, there is a natural commutative diagram:
\[\xymatrix{
\snglr_n(X,Y) \ar[r]\ar[d]&\Map(X,Y)\ar[d]\\
\Map\bigl(\sI_p\sC_\ast(X), \sI_p\sC_\ast(Y)\bigr) \ar[r]&\Map\bigl(\sC_*(X), \sC_*(Y)\bigr)\\}\]
Consequently, intersection homology is defined on the $\oo$-category $\snglr_n^{\sf ps}$ of $n$-dimensional pseudomanifolds. 
Obviously $\sI_p\sC_\ast (X\sqcup Y) \cong \sI_p\sC_\ast (X) \oplus \sI_p\sC_\ast (Y)$.  
We have the following:
\begin{prop} The intersection homology functor
\[\sI_p\sC_\ast: \snglr_n^{\sf ps} \longrightarrow {\sf Ch}\] defines a homology theory in $\bcH( \snglr_n^{\sf ps} , {\sf Ch}^\oplus)$.
\end{prop}
The proof is exactly that intersection homology satisfies excision, or has a version of the Mayer-Vietoris sequence for certain gluings.
\begin{proof} Let $X\cong  X_- \cup_{\RR\times V}X_+$ be a collar-gluing. Then 
\[\xymatrix{
\sI_p\sC_\ast(\RR\times V)\ar[d]\ar[r]&\sI_p\sC_\ast(X_+)\ar[d]\\
\sI_p\sC_\ast(X_-)\ar[r]&\sI_p\sC_\ast(X)\\}\]
is a pushout diagram in the $\oo$-category of chain complexes. I.e., the natural map 
\\
$\sI_p\sC_\ast(X_-)\oplus_{\sI_p\sC_\ast(\RR\times V)}\sI_p\sC_\ast(X_+) \ra \sI_p\sC_\ast(X)$ is a quasi-isomorphism.
\end{proof}

\subsection{Link homology theories and $\disk^{\fr}_{d\subset n}$-algebras}

We now consider one of the simplest, but more interesting, classes of $n$-dimensional stratified spaces -- that of $n$-manifolds together with a distinguished properly embedded $d$-dimensional submanifold. 
While we specialize to this class of stratified spaces, the techniques for their analysis are typical of techniques that can be used for far more general classes.

Recall from \cite{aft1} the $\oo$-category $\mfld_{d\subset n}^{\fr}$ whose objects are framed $n$-manifolds $M$ with a properly embedded $d$-dimensional submanifold $L\subset M$ together with a splitting of the framing along this submanifold, and the full $\infty$-subcategory $\disk^{\fr}_{d\subset n}\subset \mfld_{d\subset n}^{\fr}$ generated under disjoint union by the two objects $\RR^n := U^n_{\emptyset^{-1}}$ and $(\RR^d\subset \RR^n):= U^n_{S^{n-d-1}}$ with their standard framings.

\subsubsection{Explicating $\disk^{\fr}_{d\subset n}$-algebras}
Fix a symmetric monoidal $\oo$-category $\cV$ which is $\ot$-sifted cocomplete.  
Recall from the discussion of the push-forward and Corollary \ref{fubini} the map of $\oo$-categories $\int_Y\colon \Alg_{\disk_n^{\fr}}(\cV) \to \Alg_{\disk_{d+1}^{\fr}}(\cV)$
defined for any framed $(n-d-1)$-manifold $Y$.  

\begin{prop}\label{nk} There is a pullback diagram:

\[\xymatrix{
\Alg_{\disk^{\fr}_{d\subset n}}(\cV)\ar[d]&\ar[l]{\Alg_{\disk^{\fr}_{d+1}}}\Bigl(\displaystyle\int_{S^{n-d-1}}A~,~ \hh^\ast_{\sD^{\fr}_d}(B)\Bigr)\ar[d]\\
\Alg_{\disk^{\fr}_{n}}(\cV)\times\Alg_{\disk^{\fr}_{d}}(\cV)&\ar[l]\{(A,B)\}\\}\]
\end{prop}
That is, the space of compatible $\disk^{\fr}_{d\subset n}$-algebra structures on the pair $(A,B)$ is equivalent to the space of $\disk^{\fr}_{d+1}$-algebra maps from $\int_{S^{n-d-1}}A$ to the Hochschild cohomology $\hh^\ast_{\sD^{\fr}_d}(B)$; the datum of a $\disk^{\fr}_{d\subset n}$-algebra is equivalent to that of a triple $(A,B, \alpha)$, where $A$ is a $\disk^{\fr}_n$-algebra, $B$ is a $\disk^{\fr}_d$-algebra, and $\alpha$ is a map of $\disk^{\fr}_{d+1}$-algebras
	\[
	\alpha: \int_{S^{n-d-1}} A \longrightarrow \hh^*_{\sD^{\fr}_d} (B)
	\]
-- this is an $S^{n-d-1}$ parametrized family of central $\disk^{\fr}_d$-algebra actions of $A$ on $B$. In essence, Proposition~\ref{nk} is a parametrized version of the higher Deligne conjecture, and in the proof we will rely on the original version of the higher Deligne conjecture (proved in \S5.3 of \cite{HA}).

\begin{proof}

The $\oo$-category $\disk_{d\subset n}^{\fr}$ has a natural filtration by the number of components which are isomorphic to the stratified space $(\RR^d\subset \RR^n)$:
\[
\disk_n^{\fr}=(\disk_{d\subset n}^{\fr})_{\leq 0} \ra (\disk_{d\subset n}^{\fr})_{\leq 1} \ra \ldots \ra \colim_i ~(\disk_{d\subset n}^{\fr})_{\leq i} \simeq \disk^{\fr}_{d\subset n}
\] 
Consider the second step in this filtration, the full subcategory $(\disk_{d\subset n}^{\fr})_{\leq 1}$ of $\disk_{d\subset n}^{\fr}$ whose objects contain at most one connected component equivalent to $(\RR^d\subset \RR^n)$. 

Disjoint union endows $(\disk_{d\subset n}^{\fr})_{\leq 1}$ with a \emph{partially defined} symmetric monoidal structure. 
This partially defined symmetric monoidal structure can be articulated as follows.  Consider the pullback $(\disk^{\fr,\sqcup}_{d\subset n})_{\leq 1} := (\disk^{\fr}_{d\subset n})_{\leq 1} \times_{\disk^{\fr}_{d\subset n}} \disk^{\fr,\sqcup}_{d\subset n}$ where here we are using the map from the right factor $\sqcup\colon \disk^{\fr,\sqcup}_{d\subset n} \to \disk^{\fr}_{d\subset n}$.  
The coCartesian fibration $\disk^{\fr, \sqcup}_{d\subset n} \to \mathsf{Fin}_\ast$ restricts to a map $(\disk^{\fr,\sqcup}_{d\subset n})_{\leq 1} \to \mathsf{Fin}_\ast$ which is an inner fibration and for each edge $f$ in $\mathsf{Fin}_\ast$ with a lift $\w{J}_+$ of its source in $(\disk^{\fr,\sqcup}_{d\subset n})_{\leq 1}$ there is either a coCartesian edge over $f$ with source $\w{J}_+$ or the simplicial set of morphisms over $f$ with source $\w{J}_+$ is empty.  
In this way, by a symmetric monoidal functor from $(\disk^{\fr,\sqcup}_{d\subset n})_{\leq 1}$ over $\mathsf{Fin}_\ast$ it is meant a map over $\mathsf{Fin}_\ast$ which sends coCartesian edges to coCartesian edges.  

It is immediate that such a symmetric monoidal functor $F$ is equivalent to the data of a $\disk^{\fr}_n$-algebra $F(\RR^n)$ and a $\disk^{\fr}_{n-d}$-$F(\RR^n)$-module given by $F(\RR^d\subset \RR^n)$. Extending such a symmetric monoidal functor $F$ to $\disk_{d\subset n}^{\fr}$ is thus equivalent to giving a $\disk^{\fr}_d$-algebra structure on $F(\RR^d\subset \RR^n)$ compatible with the $\disk^{\fr}_{n-d}$-$F(\RR^n)$-module. That is, the following is a triple of pullback squares of $\oo$-categories
\[\xymatrix{
\Alg_{\disk_{d\subset n}^{\fr}}(\cV)\ar[d]&\ar[l]\Alg_{\disk^{\fr}_{d}}\bigl(\m_A^{\disk^{\fr}_{n-d}}(\cV)\bigr)\ar[d]\\
\Fun^\ot\bigl((\disk_{d\subset n}^{\fr})_{\leq 1},\cV\bigr)\ar[d]&\ar[l]\m_A^{\disk^{\fr}_{n-d}}(\cV)\ar[d]\\
\Alg_{\disk^{\fr}_{n}}(\cV)&\ar[l]\{A\}~.
\\}\]

Using the equivalence $\m_A^{\disk^{\fr}_{n-d}}(\cV)\simeq \m_{\int_{S^{n-d-1}}A}(\cV)$ of~\cite{cotangent}, we can then apply the higher Deligne conjecture to describe an object of the $\oo$-category \[\Alg_{\disk^{\fr}_{d}}\Bigl(\m_A^{\disk^{\fr}_{n-d}}(\cV)\Bigr) \simeq \Alg_{\disk^{\fr}_{d}}\Bigl(\m_{\int_{S^{n-d-1}}A}(\cV)\Bigr)~.\] That is, to upgrade a $\disk^{\fr}_d$-algebra $B$ to the structure of a $\disk^{\fr}_d$-algebra in $\int_{S^{n-d-1}}A$-modules is equivalent to giving a $\disk^{\fr}_{d+1}$-algebra map $\alpha: \int_{S^{n-d-1}}A\ra \hh^*_{\sD^{\fr}_d}(B)$ to the $\sD^{\fr}_d$-Hochschild cohomology of $B$.
\end{proof}

\subsubsection{Hochschild cohomology in spaces}

We now specialize our discussion of $\disk^{\fr}_{d\subset n}$-algebras to the case where $\cV=\cS^\times$ is the $\oo$-category of spaces with Cartesian product, but any $\infty$-topos would do just as well. In this case, the $\sD^{\fr}_n$-Hochschild cohomology of an $n$-fold loop space has a very clear alternate description which is given below.

\begin{prop}\label{HH-Aut}
Let $Z=(Z,\ast)$ be a based space which is $n$-connective.  
In a standard way, the $n$-fold based loop space $\Omega^n Z$ is a $\disk^{\fr}_n$-algebra.  
There is a canonical equivalence of $\disk^{\fr}_{n+1}$-algebras in spaces \[\hh^*_{\sD^{\fr}_n}(\Omega^n Z) \simeq \Omega^n \Aut(Z)\] between the $\sD^{\fr}_n$-Hochschild cohomology space of $\Omega^nZ$ and the $n$-fold loops, based at the identity map, of the space of homotopy automorphisms of $Z$.
\end{prop} 

\begin{proof} 
In what follows, all mapping spaces will be regarded as based spaces, based at either the identity map or at the constant map at the base point of the target argument -- the context will make it clear which of these choices is the appropriate one.  

There are equivalences 
\[
\m_{\Omega^n Z}^{\disk^{\fr}_n}(\cS) \simeq \m_{\int_{S^{n-1}}\Omega^n Z}(\cS) \simeq \m_{\Omega Z^{S^{n-1}}}(\cS)\simeq \cS_{/Z^{S^{n-1}}}
\] 
sending the object $\Omega^n Z$ with its natural $\disk^{\fr}_n$-$\Omega^nZ$-module self-action to the space $Z$ with the natural inclusion of constant maps $Z\ra Z^{S^{n-1}}$. 
Thus, to describe the mapping space \[\hh^*_{\sD^{\fr}_n}(\Omega^nZ) \simeq {\m_{\Omega^nZ}^{\disk^{\fr}_n}}(\Omega^nZ,\Omega^nZ)\] it suffices to calculate the equivalent mapping space $\Map_{/Z^{S^{n-1}}}(Z,Z)$. By definition, there is the (homotopy) pullback square of spaces
\[\xymatrix{
\Map_{/Z^{S^{n-1}}}(Z,Z)\ar[r]\ar[d]& \ast\ar[d]\\
\Map(Z,Z)\ar[r]&\Map(Z,Z^{S^{n-1}})~.
\\}\]

Choose a base point $p\in S^{n-1}$.  The restriction of the evaluation map $ev_p^\ast\colon\Map(Z,Z) \ra \Map(Z,Z^{S^{n-1}})$ is a map of based spaces.   
Thus, the pullback diagram above factorizes as the (homotopy) pullback diagrams
\[\xymatrix{
\Map_{/Z^{S^{n-1}}}(Z,Z)\ar[r]\ar[d]& \ast\ar[d]\\
\Map(Z,Z)^{S^n}\ar[r]\ar[d]&\Map(Z,Z)\ar[d]\\
\Map(Z,Z)\ar[r]&\Map(Z,Z^{S^{n-1})}\\
}\]
where the space of maps from the suspension $S^n= \Sigma S^{n-1}$ to $\Map(Z,Z)$ is realized as the homotopy pullback of the two diagonal maps $\Map(Z,Z) \ra \Map(Z,Z)^{S^{n-1}}$; this is a consequence of the fact that the functor $\Map(-,Z)$ sends homotopy colimits to homotopy limits, applied to the homotopy colimit $\colim(\ast \la S^{n-1}\ra \ast)\simeq S^n$.

Applying the adjunction between products and mapping spaces, we obtain that the Hochschild cohomology space $\Map_{/Z^{S^{n-1}}}(Z,Z)$ is the homotopy fiber of the map $\Map\bigl(S^n, \Map(Z, Z)\bigr) \ra \Map(Z,Z)$ over the identity map of $Z$, which recovers exactly the definition of the based mapping space $\Map_*\bigl(S^n,\Map(Z,Z)\bigr) \simeq \Map_*\bigl(S^n,\Aut(Z)\bigr)$, where the last equivalence follows by virtue of $S^n$ being connected.

\end{proof}

\begin{cor}
Let $Z$ and $W$ be pointed spaces.  Suppose $Z$ is $n$-connective and $W$ is $d$-connective. A $\disk^{\fr}_{d\subset n}$-algebra structure on the pair $(\Omega^n Z,\Omega^d W)$ is equivalent to the data of a pointed map of spaces
	\[
	Z^{S^{n-d-1}} \longrightarrow {\sf BAut}(W)~.
	\]

\end{cor}

\begin{proof} 
Proposition~\ref{nk} informs us that giving the structure of a $\disk^{\fr}_{d\subset n}$-algebra on $(\Omega^nZ,\Omega^dW)$ is equivalent to defining a $\disk^{\fr}_{d+1}$-algebra map \[\int_{S^{n-d-1}}\Omega^nZ \longrightarrow \hh^*_{\sD^{\fr}_d}(\Omega^dW)~.\] 
By way of nonabelian Poincar\'e duality (Theorem~\ref{non-abel}), the factorization homology $\int_{S^{n-d-1}}\Omega^nZ$ is equivalent as $\disk^{\fr}_{d+1}$-algebras to the mapping space $\Omega^{d+1}Z^{S^{n-d-1}}$.  Proposition~\ref{HH-Aut} gives that the Hochschild cohomology $\hh^*_{\sD^{\fr}_d}(\Omega^dW)$ is equivalent to the space of maps to $\Omega^d \Aut(W)$.

Finally, a $(d+1)$-fold loop map $\Omega^{d+1}Z^{S^{n-d-1}} \ra \Omega^d \Aut(W)$ is equivalent to a pointed map between their $(d+1)$-fold deloopings. The $(d+1)$-fold delooping of $\Omega^{d+1}Z^{S^{n-d-1}}$ is $Z^{S^{n-d-1}}$, since $Z$ is $n$-connective; the $(d+1)$-fold delooping of $\Omega^d \Aut(W)$ is $\tau_{\geq d+1}{\sf BAut}(W)$, the $d$-connective cover of ${\sf BAut}(W)$. However, since $Z^{S^{n-d-1}}$ is already $(d+1)$-connective, the space of maps from it into $\tau_{\geq d+1}{\sf BAut}(W)$ is homotopy equivalent to the space of maps into ${\sf BAut}(W)$.
\end{proof}

\subsubsection{Free $\disk^{\fr}_{d\subset n}$-algebras}\label{free-examples}

Fix a symmetric monoidal $\oo$-category $\cV$ whose underlying $\oo$-category is presentable and whose monoidal structure distributes over colimits in each variable.  
We analyze the factorization homology theory resulting from one of the simplest classes of $\disk^{\fr}_{d\subset n}$-algebras, that of freely generated $\disk^{\fr}_{d\subset n}$-algebras. 
That is, there is a forgetful functor
	\begin{equation}\label{forget-two}
	\Alg_{\disk^{\fr}_{d\subset n}}(\cV) \to \cV \times \cV~,
	\end{equation}
given by evaluating on the objects $\RR^n$ and $(\RR^d\subset \RR^n)$, and this functor admits a left adjoint $\free_{d\subset n}$. 
To accommodate more examples, we modify~(\ref{forget-two}).  
Consider the maximal sub-Kan complex $\cE\subset \sD_{d\subset n}$.  
 $\cE$ is a coproduct $\End_{\sD_{d\subset n}}(\RR^n)\coprod \End_{\sD_{d\subset n}}(\RR^d\subset \RR^n) \simeq \sO(n) \coprod \sO(d\subset n)$, here $\sO(d\subset n):= \sO(n-d)\times \sO(d)$.  
There results a map of $\oo$-categories $\cE \to \sD_{d\subset n} \to \disk_{d\subset n} \to \disk_{d\subset n}^\sqcup$, restriction along which gives the map of $\oo$-categories
\[
\Alg_{\disk_{d\subset n}}(\cV) \to \Map(\cE,\cV) \simeq \cV^{\sO(n)}\times \cV^{\sO(d\subset n)}
\]
to the $\oo$-category of pairs $(P,Q)$ consisting of an $\sO(n)$-object in $\cV$ and an $\sO(d\subset n)$-object in $\cV$.  
We will denote the left adjoint to this map as $\free_\cE$.  
Denote the inclusion as $\delta \colon \cV\times \cV \to \Map(\cE,\cV)$ as the pairs $(P,Q)$ whose respective actions are trivial.  

For $X$ a $\sD_{d\subset n}$-manifold (not necessarily framed) define $\int_X \free^{(P,Q)}_{d\subset n}:= \int_X \free^{\delta(P,Q)}_\cE$.  
When $X$ is framed (i.e., is a $\sD^{\fr}_{d\subset n}$-manifold) the lefthand side of this expression has already been furnished with meaning as the factorization homology of $X$ with coefficients in the free $\disk^{\fr}_{d\subset n}$-algebra generated by $(P,Q)$.  
The following lemma ensures that the two meanings agree.
Recall the forgetful map $\disk^{\fr, \sqcup}_{d\subset n} \to \disk_{d\subset n}^\sqcup$.  

\begin{lemma}\label{frame-or-not}
Let $(P,Q)$ be a pair of objects of $\cV$.  
Then the universal arrow
\[
\free_{d\subset n}^{(P,Q)} \xra{\simeq} \bigl(\free_\cE^{\delta(P,Q)}\bigr)_{|\disk^{\fr, \sqcup}_{d\subset n}}
\]
is an equivalence.  

\end{lemma}

\begin{proof}

Denote the pullback $\oo$-category $\cE^{\fr} = \cE\times_{\sD_{d\subset n}} \sD^{\fr}_{d\subset n}$.  
Like $\cE$, $\cE^{\fr}\subset \sD^{\fr}_{d\subset n}$ is the maximal sub-Kan complex (=$\infty$-groupoid).  
The projection $\cE^{\fr} \to \cE$ is a Kan fibration with fibers $\sO(n)$ or $\sO(d\subset n)$, depending on the component of the base.  
As so, the inclusion of the two objects with their standard framings $\{\RR^n\}\coprod \{\RR^d\subset \RR^n\} \xra{\simeq} \cE^{\fr}$ is an equivalence of Kan complexes.  
Therefore $\Map(\cE^{\fr},\cV) \xra{\simeq} \cV\times \cV$.  

Let us explain the following diagram of $\oo$-categories
\[
\xymatrix{
\Alg_{\disk_{d\subset n}}(\cV)   \ar[rr]^-{\rho}  \ar@(l,l)[dd]_-{\rho}
&&
\Map(\cE,\cV)   \ar@(r,r)[dd]^-{\rho}    \ar@(u,u)[ll]^-{\free_\cE} 
\\
&&
\\
\Alg_{\disk^{\fr}_{d\subset n}}(\cV)  \ar[rr]_-{\rho}  \ar[uu]_-{\mathsf{RKan}}  
&&
\Map(\cE^{\fr},\cV)~.   \ar[uu]^-{\mathsf{RKan}}  \ar@(d,d)[ll]^-{\free_{d\subset n}}
}
\]
Each leg of the square is an adjunction.
All maps labeled by $\rho$ are the evident restrictions.  
The maps denoted as $\mathsf{RKan}$ are computed as point-wise right Kan extension.  
(That is, $\mathsf{RKan}(A)\colon U\mapsto \lim_{U\to U'} A(U')$ where this limit is taking place in $\cV$ and is indexed by the appropriate over category.  
We emphasize that, unlike the case for left extensions, this point-wise right Kan extension agrees with operadic right Kan extension.)
As so, the straight square of right adjoints commutes.
It follows that the outer square of left adjoints also commutes.  

The right downward map is equivalent to that which assigns to a pair of objects $(P,Q)$ with respective actions of $\sO(n)$ and $\sO(d\subset n)$, the pair $(P,Q)$.  
The map $\delta\colon \cV\times \cV \to \cV^{\sO(n)}\times \cV^{\sO(d\subset n)}$ is a section to this right downward map $\rho$.    
We have established the string of canonical equivalences
\[
\free_{d\subset n} \simeq \free_{d\subset n}\circ\bigl(\rho \circ \delta\bigr) = \bigl(\free_{d\subset n}\circ \rho\bigr)\circ \delta \simeq \bigl(\rho\circ \free_\cE\bigr)\circ \delta~.
\]
This completes the proof.  

\end{proof}

In order to formulate our main result, we first give the following definition.

\begin{definition} For $M$ a topological space and $P$ an object of $\cV$, the \emph{configuration object of points in $M$ labeled by $P$} is 
\[
\conf^P(M) = \coprod_{j\geq 0} \conf_j(M)\underset{\Sigma_j}\ot P^{\ot j} \in \cV
\] 
where $\conf_j (M)\subset M^{\times j}$ is the configuration space of $j$ ordered and distinct points in $M$.
\end{definition}

For the remainder of the section, assume that the monoidal structure of $\cV$ distributes over small colimits.

\begin{prop}\label{prop:conf} Let $(P,Q)$ be a pair of objects of $\cV$.  
Let $(L\subset M)$ be a $\sD_{d\subset n}$-manifold, i.e., a smooth $n$-manifold and a properly embedded $d$-submanifold.  
There is a natural equivalence
	\[
	\int_{(L\subset M)} \free_{d\subset n}^{(P,Q)}
	\simeq
	\conf^{P}(M\smallsetminus L) \ot \conf^{Q}(L)\]
between the factorization homology of $(L\subset M)$ with coefficients in the $\disk_{d\subset n}$-algebra freely generated by $(P,Q)$ and the tensor product of the configurations objects of the link complement $M\smallsetminus L$ and the link $L$ labeled by $P$ and $Q$, respectively.
\end{prop}

We make some remarks before proceeding with the proof of this result.

\begin{remark} We see from this result with $(d\subset n)=(3,1)$ that factorization homology in particular distinguish knots whose knot complements have distinct homotopy types. For instance, the unknot, whose knot group is $\ZZ$, and the trefoil knot, whose knot group is presented by $\langle x,y | x^2 = y^3 \rangle$, give rise to different factorization homologies.
\end{remark}

\begin{remark} Specializing to the case where the link $L$ is empty, we obtain the equivalence $\int_M \free_n^P \simeq \conf^P(M)$. Consequently, factorization homology is not a homotopy invariant of $M$, in as much as the homotopy types of the configuration spaces $\conf_j(M)$ are sensitive to the homeomorphism (or, at least, the simple homotopy) type of $M$, see~\cite{simple}. This is in contrast to the case in which the $\disk^{\fr}_n$-algebra $A$ comes from an $n$-fold loop space on an $n$-connective space, in which case nonabelian Poincar\'e duality~(Theorem~\ref{non-abel}) implies that factorization homology with such coefficients is a proper homotopy invariant. However, note that the factorization homology $\int_M \free_n^P$ is independent of the framing on $M$; this is a consequence of the fact that the $\disk_n^{\fr}$-algebra structure on $\free_n^P$ can be enhanced to a $\disk_n$-algebra.
\end{remark}

Recall the maps of symmetric monoidal $\oo$-categories $\disk^{\fr, \sqcup}_n \to \disk^{\fr,\sqcup}_{d\subset n}$ and $\disk^{\fr,\sqcup}_d \to \disk^{\fr,\sqcup}_{d\subset n}$ indicated by the assignments $\RR^n\mapsto \RR^n$ and $\RR^d \mapsto (\RR^d\subset \RR^n)$, respectively.
The following lemma describes the free $\disk_{d\subset n}^{\fr}$-algebras in terms of free $\disk^{\fr}_n$-algebras and free $\disk^{\fr}_d$-algebras. 

\begin{lemma}\label{free-free}
Let $(P,Q)$ be a pair of objects of $\cV$. Then the universal arrows to the restrictions
	\[
	 \free_n^P\xra{\simeq} \bigl(\free_{d\subset n}^{(P,Q)}\bigr)_{|\disk^{\fr,\sqcup}_n} 
	\qquad \&
	\qquad
	 \free_d^Q \ot\int_{S^{n-d-1}\times \RR^{d+1}}\free_n^P~{}~\xra{\simeq} ~{}~\bigl(\free_{d\subset n}^{(P,Q)}\bigr)_{|\disk^{\fr,\sqcup}_d}
	\]
are equivalences.  
\end{lemma}

\begin{proof} 
Recall from the proof of Proposition~\ref{nk} that a $\disk^{\fr}_{d\subset n}$-algebra structure $A$ on $(A_n, A_d)$, where $A_n$ is a $\disk^{\fr}_n$-algebra and $A_d$ is a $\disk^{\fr}_d$-algebra, is equivalent to the structure of a $\disk_{n-d}^{\fr}$-$A_n$-module structure on $A_d$ in the $\oo$-category $\Alg_{\disk^{\fr}_d}(\cV)$. 
The forgetful functor factors as the forgetful functors
\[
\Alg_{\disk^{\fr}_{d\subset n}}(\cV)\longrightarrow\Alg_{\disk^{\fr}_n}(\cV)\times\Alg_{\disk^{\fr}_d}(\cV)\longrightarrow\cV\times\cV
\] 
and thus, passing to the left adjoints, we can write the free algebra $A$ on a pair $(P,Q)$ as the composite of the two left adjoints, which gives the free $\disk^{\fr}_n$-algebra on $P$ and the free $\disk^{\fr}_{n-d}$-module on the free $\disk^{\fr}_d$-algebra on $Q$; the latter is calculated by tensoring with the factorization homology $\int_{S^{n-d-1} \times \RR^{d+1}}\free_n^P$, which is a special case of the equivalence between ${\disk^{\fr}_j}$-$R$-modules and left modules for $\int_{S^{j-1}}R$, see Proposition 3.16 of~\cite{cotangent}, applied to $R =\free_n^P$ and $j=n-d$.

\end{proof}

\begin{proof}[Proof of Proposition ~\ref{prop:conf}]

Recall the construction of the $\infty$-operad $\cE^{\amalg}_{\sf inert}$ -- it is the free $\infty$-operad on $\cE$.  That is, the map 
\[
\mathsf{Fun}^\ot(\cE^{\amalg}_{\sf inert},\cV) \xra{\simeq} \mathsf{Fun}(\cE,\cV)\simeq \cV^{\sO(n)}\times \cV^{\sO(d\subset n)}~, 
\]
induced by restriction along the inclusion of the underlying $\oo$-category $\cE \to \cE^\amalg_{\sf inert}$, is an equivalence of Kan complexes -- here we are using exponential notation for simplicial sets of maps.
Explicitly, a vertex of $\cE^\amalg_{\sf inert}$ is a pair of finite sets $(J_n,J_d)$ while an edge is the data of a pair of based maps $(J_n)_+ \xra{a} (J_n')_+$ and $(J_d)_+ \xra{b} (J_d')_+$ which are \emph{inert}, which is to say, the fibers over non-base points $a^{-1}(j')$ and $b^{-1}(j'')$ are each singletons, together with a pair of elements $\alpha\in \sO(n)^{J_n'}$ and $\beta\in \sO(d\subset n)^{J_d'}$.  
We will denote a typical object of $\cE^\amalg_{\sf inert}$ as $E=(J_n,J_d)$.  

The the standard inclusion $\cE \to \mfld_{d\subset n}$ then induces the map of $\infty$-operads $\cE^\amalg_{\sf inert} \xra{i} \mfld^{\sqcup}_{d\subset n}$ whose value on vertices is
\[
i\colon (J_n,J_d)\mapsto \bigl(\bigsqcup_{J_n} \RR^n\bigr)\sqcup \bigl(\bigsqcup_{J_d} (\RR^d\subset \RR^d)\bigr)~.
\] 
Likewise, let $(\sO(n)\xra{\w{P}}\cV~,~\sO(d\subset n)\xra{\w{Q}} \cV)$ be a pair $(P,Q)$ of objects in $\cV$ each equipped with actions of $\sO(n)$ and $\sO(d\subset n)$, respectively.  
These data then determine the solid diagram of $\infty$-operads
\[
\xymatrix{
\cE^{\amalg}_{\sf inert}  \ar[rrrrrd]^-{(P^\bullet,Q^\bullet)}   \ar[d]^i
&&&&&
\\
\mfld^{\sqcup}_{d\subset n}  \ar@{.>}[rrrrr]^-{\mathsf{Free}_\cE^{(\w{P},\w{Q})}}
&&&&&
\cV
}
\]
in where the value of $(P^\bullet,Q^\bullet)$ on $(J_n,J_d)$ is canonically equivalent to $ P^{\ot J_n} \ot Q^{\ot J_d}$ as a $\bigl(\sO(n)^{J_n} \times \sO(d\subset n)^{J_d}\bigr)$-objects.  
The filler $\mathsf{Free}_\epsilon^{(\w{P},\w{Q})}$ is the desired free construction, and is computed as operadic left Kan extension.  Explicitly, for $X\in \mfld_{d\subset n}$, the value
\begin{equation}\label{free-colimit}
\mathsf{Free}_\cE^{(\w{P},\w{Q})}(X) = \colim_{E\xra{\mathsf{act}} X}  (P^\bullet,Q^\bullet)(E) = \colim_{(J_n,J_d)\xra{\sf act} X} P^{\ot J_n} \ot Q^{\ot J_d}
\end{equation}
where the colimit is over the $\oo$-category $\cE^\amalg_X  := \cE^\amalg_{\sf inert} \times_{\mfld^{\sqcup}_{d\subset n}} \bigl((\mfld^{\sqcup}_{d\subset n})_{\sf act}\bigr)_{/X}$ of active morphisms in $\mfld^{\sqcup}_{d\subset n}$ from the image under $i$ of $\cE^\amalg_{\sf inert}$ to the object $X$.  

We will now compute the colimit in~(\ref{free-colimit}).  
By construction, the projection $\cE^\amalg_X \to \cE^\amalg_{\sf inert}$ is a right fibration whose fiber over $E$ is the Kan complex $\mfld_{d\subset n}(i(E),X)$. 
Consider the subcategory of isomorphisms $\cE^\amalg_{\sf iso}\subset \cE^\amalg_{\sf inert}$ -- it is isomorphic to the category of pairs of finite sets and pairs of bijections among them. 
Denote $\cG = \cE^\amalg_{\sf iso}\times_{\cE^\amalg_{\sf inert}} \cE^\amalg_X$.  
Because the inclusion $\cE^\amalg_{\sf iso}\subset \cE^\amalg_{\sf inert}$ is final, so is the inclusion $\cG \subset \cE^\amalg_X$.  So the colimit~(\ref{free-colimit}) is canonically equivalent to the colimit of the composite $\cG \subset \cE^\amalg_X \xra{(P^\bullet,Q^\bullet)} \cV$.  
Because $\cE^\amalg_{\sf iso}$ is a coproduct of $\infty$-groupoids (Kan complexes) indexed by isomorphism classes of its objects, then $\cG$ is a coproduct of $\infty$-groupoids indexed by isomorphism classes of objects of $\cE^\amalg_{\sf iso}$.  As so, the colimit~(\ref{free-colimit}) breaks up as a coproduct over isomorphism classes of objects of $\cE^\amalg_{\sf iso}$.  

We will now understand the $[E]^{\rm th}$ summand of this colimit.  
Choose a representative $E=(J_n,J_d)\in \cE^\amalg_{\sf iso}$ of this isomorphism class. 
We point out that the Kan complex of $\Aut(E)$ fits into a Kan fibration sequence $\sO(n)^{ J_n}\times \cO(d\subset n)^{ J_d} \to \Aut(E) \to \Sigma_{J_n}\times \Sigma_{J_d}$.  
Consider the right fibration $(\cE^\amalg_{\sf iso})_{/E} \to \cE^\amalg_{\sf iso}$ whose fiber over $E'$ is the Kan complex $\Iso(E',E)$ which is a torsor for the Kan complex $\Aut(E)$ is $E'$ if isomorphic to $E$ and is empty otherwise.  
Denote the resulting right fibration $\cG_E = (\cE^\amalg_{\sf iso})_{/E} \times_{\cE^\amalg_{\sf iso}} \cG~\longrightarrow~ \cG$ whose fibers are either a torsor for $\Aut(E)$ or empty.  
The composite $\cG_E \to \cG \subset \cE^\amalg_X \xra{(P^\bullet,Q^\bullet)} \cV$ is canonically equivalent to the constant map at $P^{\ot J_n}\ot Q^{\ot J_d}$.  
It follows from the definition of the tensor over spaces structure, that the colimit of this composite  is 
\[
\Bigl(\mfld_{d\subset n}\bigl(i(J_n, J_d),X \bigr)\Bigr) \ot \bigl(P^{\ot J_n} \ot Q^{\ot J_d}\bigr)~.
\]  
We conclude from this discussion that the colimit of the composite $\cG \subset \cE^\amalg_X \xra{(P^\bullet,Q^\bullet)} \cV$ is 
\begin{equation}\label{tensor-time}
\coprod_{[(J_n,J_d)]} \Bigl(\mfld_{d\subset n}\bigl(i(J_n,J_d),X \bigr)\Bigr) \ot_{\Aut(J_n,J_d)} \bigl(P^{\ot J_n} \ot Q^{\ot J_d}\bigr)~.
\end{equation}

We make expression~(\ref{tensor-time}) more explicit for the case that $(\w{P},\w{Q}) = \delta(P,Q)$ is a pair of objects with trivial group actions.  
As so, the map $\cG \subset \cE^\amalg_X \xra{(P^\bullet,Q^\bullet)} \cV$ factors through the projection $\cG \to \cE^\amalg_{\sf iso} \to (\mathsf{Fin}_\ast)_{iso}$ the groupoid of finite sets and bijections -- we denote this groupoid as $\Sigma$.   

Recall that the $\sD_{d\subset n}$-manifold $X=(L\subset M)$ is the data of a framed $n$-manifold $M$, a properly embedded smooth submanifold $L$, and a splitting of the framing along $L$.  Evaluation at the origins of $i(J_n,J_d) = \bigl(\bigsqcup_{J_n} \RR^n\bigr)\sqcup \bigl(\bigsqcup_{J_d} (\RR^d\subset \RR^n)\bigr)$ gives a map
\[
\mfld_{d\subset n}\bigl((\bigsqcup_{J_n} \RR^n)\sqcup (\bigsqcup_{J_d} \bigl(\RR^d\subset \RR^n)\bigr)~,~(L\subset M)\bigr) ~\longrightarrow~ \mathsf{Conf}_{J_n}(M\smallsetminus L)\times \mathsf{Conf}_{J_d}(L)~.  
\]
This map is evidently natural among morphisms among the variable $(J_n,J_d)\in \cG$ where the action of $\cG$ on the righthand side factors through the projection $\cG \to \Sigma$.   
There results a $\Sigma$-equivariant map
\[
\Bigl(\mfld_{d\subset n}\bigl(i(J_n,J_d),(L\subset M)\bigr)\Bigr)_{/\sO(n)^{ J_n}\times \sO(d\subset n)^{ J_d}} ~\overset{\sim}\longrightarrow~ \mathsf{Conf}_{J_n}(M\smallsetminus L)\times \mathsf{Conf}_{J_d}(L)~;
\]
and for standard reasons it is an equivalence of Kan complexes.  
We conclude that
\[
\free_\cE^{\delta(P,Q)}(L\subset M) \xra{\simeq} \conf^P(M\smallsetminus L)\ot \conf^Q(L)~.
\]
Finally, the formula 
\[
\free_\cE^{\delta(P,Q)}(L\subset M) = \int_{(L\subset M)} \free_\cE^{\delta(P,Q)}
\]
is a formal consequence of commuting left Kan extensions (here we are using the same notation for $\free_\cE^{\delta(P,Q)}$ and its restriction to $\disk^{\sqcup}_{d\subset n}$).  With Lemma~\ref{frame-or-not}, this completes the proof of the proposition.  

\end{proof}

\begin{remark}
The methods employed here in~\S\ref{free-examples} have been intentionally presented to accommodate much greater generality.  
For instance, with appropriate modifications of the statements, the role of $\sD_{d\subset n}$ (or its framed version) could be replaced by any category of basics $\cB$.  Likewise, the maximal sub-Kan complex $\cE\subset \sD_{d\subset n}$ could be replaced by any map $\cE\to \cB$ of $\oo$-categories.  

\end{remark}


\begin{thebibliography}{99}

\bibitem[At]{atiyah} Atiyah, Michael. Thom complexes, Proc. London Math. Soc. (3), no. 11 (1961), 291--310.

\bibitem[AF1]{Fact} Ayala, David; Francis, John. Factorization homology of topological manifolds. J. Topol. 8 (2015), no. 4, 1045--1084.

\bibitem[AF2]{ZP} Ayala, David; Francis, John. Zero-pointed manifolds. Preprint.

\bibitem[AFR]{striation} Ayala, David; Francis, John; Rozenblyum, Nick. A stratified homotopy hypothesis. Preprint.

\bibitem[AFT]{aft1} Ayala, David; Francis, John; Tanaka, Hiro Lee. Local structures on stratified spaces. Preprint.

\bibitem[Ba]{barwick} Barwick, Clark. On the Q-construction for exact $\infty$-categories. Available at \url{http://math.mit.edu/~clarkbar/papers.html}.

\bibitem[BaDo]{baezdolan} Baez, John; Dolan, James. Higher-dimensional algebra and topological quantum field theory. J. Math. Phys. 36 (1995), no. 11, 6073--6105.

\bibitem[BeDr]{bd} Beilinson, Alexander; Drinfeld, Vladimir. Chiral algebras. American Mathematical Society Colloquium Publications, 51. American Mathematical Society, Providence, RI, 2004.

\bibitem[BV]{bv} Boardman, J. Michael; Vogt, Rainer. Homotopy invariant algebraic structures on topological spaces. Lecture Notes in Mathematics, Vol. 347. Springer-Verlag, Berlin-New York, 1973.




\bibitem[Co]{kevin}  Costello, Kevin. Renormalization and effective field theory. Mathematical Surveys and Monographs, 170. American Mathematical Society, Providence, RI, 2011.

\bibitem[CG]{kevinowen} Costello, Kevin; Gwilliam, Owen. Factorization algebras in perturbative quantum field theory. Preprint. Available at \url{http://www.math.northwestern.edu/~costello/renormalization}


\bibitem[DI]{Dugger--Isaksen}  Dugger, Daniel; Isaksen, Daniel. Topological hypercovers and $\AA^1$-realizations. Math. Z. 246 (2004), no. 4, 667--689.

\bibitem[Fr]{cotangent} Francis, John. The tangent complex and Hochschild cohomology of $\cE_n$-rings. Compos. Math. 149 (2013), no. 3, 430--480.

\bibitem[FG]{fg} Francis, John; Gaitsgory, Dennis. Chiral Koszul duality. Selecta Math. (N.S.) 18 (2012), no. 1, 27Ð87. 

\bibitem[GM1]{goreskymacpherson} Goresky, Mark; MacPherson, Robert. Intersection homology theory. Topology 19 (1980), no. 2, 135--162.

\bibitem[GM2]{goreskymacpherson2} Goresky, Mark; MacPherson, Robert. Intersection homology. II. Invent. Math. 72 (1983), no. 1, 77--129.

\bibitem[Jo]{joyal} Joyal, Andr\'e. Quasi-categories and Kan complexes. Special volume celebrating the 70th birthday of Professor Max Kelly. J. Pure Appl. Algebra 175 (2002), no. 1-3, 207--222.

\bibitem[LS]{simple} Longoni, Riccardo; Salvatore, Paolo. Configuration spaces are not homotopy invariant. Topology 44 (2005), no. 2, 375--380. 

\bibitem[Lu1]{HTT} Lurie, Jacob. Higher topos theory. Annals of Mathematics Studies, 170. Princeton University Press, Princeton, NJ, 2009.

\bibitem[Lu2]{HA} Lurie, Jacob. Higher algebra. Preprint, Sept. 2014. Available at \url{http://www.math.harvard.edu/~lurie/}

\bibitem[Lu3]{cobordism} Lurie, Jacob. On the classification of topological field theories. Current developments in mathematics, 2008, 129--280, Int. Press, Somerville, MA, 2009.


\bibitem[Mc]{mcduff} McDuff, Dusa. Configuration spaces of positive and negative particles.  Topology 14 (1975), 91--107. 

\bibitem[Sa]{salvatore} Salvatore, Paolo. Configuration spaces with summable labels. Cohomological methods in homotopy theory (Bellaterra, 1998), 375--395, Progr. Math., 196, Birkh\"auser, Basel, 2001.

\bibitem[Se1]{segal} Segal, Graeme. Configuration-spaces and iterated loop-spaces. Invent. Math. 21 (1973), 213--221. 

\bibitem[Se2]{segallocal} Segal, Graeme. Locality of holomorphic bundles, and locality in quantum field theory. The many facets of geometry, 164--176, Oxford Univ. Press, Oxford, 2010. 

\bibitem[Sp]{spivak} Spivak, Michael. A comprehensive introduction to differential geometry, Volume 1. Publish or Perish Inc, Wilmington, DE, 1979. xiv+668 pp.

\bibitem[Th]{thomas} Thomas, Justin. Kontsevich's Swiss Cheese Conjecture. Thesis (PhD) -- Northwestern University. 2010.


\bibitem[Vo]{voronov} Voronov, Alexander. The Swiss-cheese operad. Homotopy invariant algebraic structures (Baltimore, MD, 1998), 365--373, Contemp. Math., 239, Amer. Math. Soc., Providence, RI, 1999. 



\end{thebibliography}
\end{document}